%% file: main.tex
\RequirePackage{fix-cm}

\documentclass[smallextended,envcountsect]{svjour3}
\usepackage[margin=1in]{geometry}

\usepackage[english]{babel}
\usepackage{natbib}

\input{header.tex}
\usepackage{cleveref}

\journalname{Mathematical Programming B}

\title{Provable non-accelerations of the heavy-ball method}
\titlerunning{Provable non-accelerations of the heavy-ball method}

\author{Baptiste Goujaud \and Adrien Taylor \and Aymeric Dieuleveut}
\authorrunning{B. Goujaud \and A. Taylor \and A. Dieuleveut}
\institute{B. Goujaud \at CMAP, École Polytechnique, Institut Polytechnique de Paris. \\ \email{baptiste.goujaud@polytechnique.edu.}
\and
A. Taylor \at INRIA, École Normale Supérieure, CNRS, PSL Research University, Paris. \\ \email{adrien.taylor@inria.fr.}
\and
A. Dieuleveut \at CMAP, École Polytechnique, Institut Polytechnique de Paris. \\ \email{aymeric.dieuleveut@polytechnique.edu.}
}

\date{Received: date / Accepted: date}

\begin{document}
\maketitle
\addtocontents{toc}{\protect\setcounter{tocdepth}{0}}

\begin{abstract}
    In this work, we show that the heavy-ball ($\HB$) method provably does not reach an accelerated convergence rate on smooth strongly convex problems. More specifically, we show that for any condition number and any choice of algorithmic parameters, either the worst-case convergence rate of $\HB$ on the class of $L$-smooth and $\mu$-strongly convex \textit{quadratic} functions is not accelerated (that is, slower than $1 - \mathcal{O}(\kappa)$), or there exists an $L$-smooth $\mu$-strongly convex function and an initialization such that the method does not converge.
    
    To the best of our knowledge, this result closes a simple yet open question on one of the most used and iconic first-order optimization technique.
    
    Our approach builds on finding functions for which $\HB$ fails to converge and instead cycles over finitely many iterates. We analytically describe all parametrizations of $\HB$ that exhibit this cycling behavior on a particular cycle shape, whose choice is supported by a systematic and constructive approach to the study of cycling behaviors of first-order methods. We show the robustness of our results to perturbations of the cycle, and  extend them to class of functions that also satisfy higher-order regularity conditions.
\end{abstract}

\keywords{Convex optimization - Accelerated convergence - Heavy-ball dynamics}

\section{Introduction}\label{sec:introduction}
    In this paper, we consider the unconstrained minimization of a convex function $f: \R^d \to \R$:
    \begin{equation}
        x_\star \eqdef \arg\min_{x\in \mathbb R^d} f(x), \tag{OPT}\label{eq:opt}
    \end{equation}
    where $f$ belongs to a given class of functions $\mathcal F$ (e.g.,~the set of convex quadratic functions, or the set of strongly convex and smooth functions). 
    In this context, first-order optimization methods have recently attracted a lot of attention due to their generally low cost per iteration and their practical success in many applications.
    They are particularly relevant in applications not requiring very accurate solutions, such as in machine learning (see, e.g.,~\cite{bottou2007tradeoffs}).
    
    A major weakness of those methods is that their convergence speed is typically slow, and crucially affected by the so-called \emph{conditioning} of the function to be minimized (more in the sequel). In this context, the theoretical foundations for first-order methods played a crucial role in their success, among others by enabling the development of momentum-type methods for mitigating the impact of the conditioning on the convergence rates. That is, momentum-type methods usually behave much better both theoretically and practically as compared to the vanilla \emph{gradient descent}~(GD), arguably the most well-known and iconic first-order method. Momentum-type methods are usually classified in two categories depending on how the momentum appears in the iterative update equations. As a reference, the update rule for the vanilla \emph{gradient descent}~(GD) method for solving~\eqref{eq:opt} is
    \begin{equation}
        x_{t+1}=x_t-\gamma \nabla f(x_t), \tag{GD} \label{eq:gd}
    \end{equation}
    for some step-size $\gamma$. In~\citep{polyak1964some}, Polyak introduced the celebrated \emph{heavy-ball}~($\HB$) update rule:
    \begin{equation}
        x_{t+1} = x_t - \gamma \nabla f(x_t) + \beta (x_t - x_{t-1}).
        \label{eq:hb} \tag{HB}
    \end{equation}
    $\HB$ is notorious for being, among others, an \emph{optimal} method for minimizing convex quadratic functions, as its computational complexity matches that of the corresponding lower complexity bounds~\citep{nemirovskinotes1995}. A few years later, Nesterov introduced the~\emph{accelerated gradient} method~\citep{nesterov1983method} which consists in iterating
    \begin{equation}
        x_{t+1} = x_t - \gamma \nabla f(x_t + \beta (x_t - x_{t-1})) + \beta (x_t - x_{t-1}),
        \label{eq:nag} \tag{NAG}
    \end{equation}
    which is often referred to as Nesterov's accelerated gradient (NAG).
    
    In words, the gradient descent update is complemented with a momentum term $\beta (x_t - x_{t-1})$ being either applied after the gradient evaluation (for $\HB$) or before it (for NAG) at each iteration. As compared to the $\HB$ method, NAG is an \emph{optimal} algorithm on the class of smooth strongly convex functions~$f$ (i.e.,~beyond quadratics). Yet, $\HB$ is known to be asymptotically twice faster than Nesterov on the class of quadratic functions~(\Cref{tab:rates_fml}). Moreover, $\HB$ is among the most prevalent practical first-order optimization paradigms used in optimization software (e.g.~for machine learning. See e.g. the \href{https://docs.pytorch.org/docs/stable/generated/torch.optim.SGD.html#sgd}{Pytorch documention of Heavy-ball's practical implemention}): it works well in many situations, though the question of its theoretical convergence speed is still open beyond quadratics.
    
    \paragraph{Conditioning and its effects.} This work primarily focuses on a particular classical set of functions, namely the set of \emph{smooth and strongly convex functions} (additional higher-order regularity assumptions are considered later in~\Cref{sec:HL}). This set is very standard in the first-order optimization literature; see, e.g.,~\cite{Book:polyak1987,nemirovskinotes1995,Nest03a}.
    
    \begin{Def}[Set $\Fml$]
        Let $0<\mu\leq L<\infty$. A continuously differentiable function $f:\R^d\mapsto\R$ is $L$-smooth and $\mu$-strongly convex (notation $f\in\Fml$) if:
        \begin{itemize}
            \item ($L$-smoothness) for all $x,y\in\R^d$, it holds that
            \[ f(x)\leq f(y)+\langle \nabla f(y),\, x-y\rangle +\frac{L}{2}\|x-y\|^2,\]
            \item ($\mu$-strong convexity) for all $x,y\in\R^d$, it holds that
            \[ f(x)\geq f(y)+\langle \nabla f(y),\, x-y\rangle +\frac{\mu}{2}\|x-y\|^2.\]
        \end{itemize}
    \end{Def}
    
    In particular, twice differentiable $L$-smooth $\mu$-strongly convex functions corresponds to functions whose Hessian have bounded eigenvalues between $\mu$ and $L$ (i.e.,~$\mu\, I \preccurlyeq \nabla^2 f(x) \preccurlyeq L\, I$, or $\Sp\left(\nabla^2 f(x)\right)\subseteq [\mu,\, L]$, for all $x\in\R^d$). A particular subclass of $\Fml$ is that of quadratic functions (i.e.,~with constant Hessian).
    \begin{Def}
        Let $0\leq\mu\leq L<\infty$. A continuously differentiable function $f(x)=\langle x,\, H x\rangle + \langle b,\, x\rangle + c$ is an $L$-smooth $\mu$-strongly convex quadratic function (notation $f\in\Qml$) if and only if $\Sp\left(H\right)\subseteq [\mu,\, L]$.
    \end{Def}
    
    Those sets of functions are very standard for the analyses of first-order methods and are characterized by the (inverse) condition number $\kappa \eqdef \tfrac{\mu}{L}$. In this context, it is known that GD converges exponentially fast both for minimizing functions in $\Qml$ and $\Fml$. More precisely, $\|x_t-x_\star\|\leq \left(\frac{1-\kappa}{1+\kappa}\right)^{t} \|x_0-x_\star\|$ for suitable choices of the step-size $\gamma$. On $\Qml$, an optimally-tuned $\HB$ has $\|x_t-x_\star\|\leq C \left(\frac{1-\sqrt{\kappa}}{1+\sqrt{\kappa}}\right)^{t}\|x_0-x_\star\|$ (for some $C>0$) and is therefore (much) faster for small values of $\kappa$. However, the optimal behavior of $\HB$ beyond quadratics, on $\Fml$, is unknown to the best of our knowledge. Finally, NAG has $\|x_t-x_\star\|\leq C \left({1-\sqrt{\kappa}}\right)^{t/2}\|x_0-x_\star\|$ on $\Fml$ but is suboptimal compared to $\HB$ when working on $\Qml$, though much faster than GD.
    
    In short, albeit generally good practical performances, it remains unclear for which choices of $(\gb)$ (if any), $\HB$ allows obtaining \emph{fast} convergence speeds on the standard class of $L$-smooth $\mu$-strongly convex problems. In this work, we answer this question by proving that it is suboptimal for minimizing smooth strongly convex functions. More precisely, we show that the only choices of $(\gb)$ for which $\HB$ is guaranteed to converge on $\Fml$ only marginally improve upon the convergence speed of gradient descent.

\subsection{Related works}

    As $\HB$ is one of the most widely used methods, understanding its worst-case convergence rate on larger class of functions is a natural problem.
    
    \paragraph{Behavior of HB on quadratics.} The $\HB$ method was originally coined for optimizing quadratic functions~\cite{polyak1964some}. A classical approach to the analysis of~\eqref{eq:hb} on $\mathcal{Q}_{\mu, L}$ consists in exploiting links between first-order methods and polynomials (see e.g.~\citep{fischer2011polynomial,nemirovskinotes1995} or~\citep[Chapter 2][]{d2021acceleration} for a recent introduction--for more recent exploitation of those links, see, e.g.,~\citep{berthier2020accelerated,pedregosa2020acceleration,goujaud2022super,goujaud2022quadratic,kim2022extragradient}). In this context, a standard choice of $\HB$ parameters is $(\gbqml)= \left((\frac{2}{\sqrt{L}+\sqrt{\mu}})^2,(\frac{\sqrt{L}-\sqrt{\mu}}{\sqrt{L}+\sqrt{\mu}})^2\right)$ and is often referred to as the \emph{optimal tuning} for quadratics.

    \paragraph{Behavior of HB beyond quadratics.} The class of $L$-smooth and $\mu$-strongly convex functions (notation $\Fml$) is extensively studied in the first-order optimization literature (see, e.g.,~\citep{Book:polyak1987,Nest03a,bubeck2015convex}), and is a natural candidate for exploring the performance of $\HB$ beyond quadratics. Notable results on $\HB$ beyond quadratics include the work of~\citet{ghadimi2015global} who provide non-accelerated convergence results of $\HB$ for a large set of parameters $(\gb)$ (more details in the subsequent sections). Also,~\citet{lessard2016analysis} provide a non-convergence result of $\HB$ (with the optimal parameter choice for quadratics) on $\Fml$, based on a counter-example, i.e.,~some $f\in\Fml$ for which given a particular initialization, the~$\HB$ algorithm cycles over a finite number of values, without ever approaching the set of minimizers of~$f$. Moreover, other attempts were made to obtain an accelerated rate on~$\HB$ on the class~$\Fml$, without definitive results~\cite{dobson2023connections}; or on even larger class of functions~\citep{goujaud2022optimal}.

    As a conclusion, despite the existence of both positive and negative convergence results for $\HB$ on $\Fml$, the optimal tuning of $\HB$ on this specific class as well as its worst-case convergence rate, remain unknown.
    
    \paragraph{Erroneous convergence results on HB.} A few recent works either claim or use the fact that $\HB$ does converge with an accelerated convergence rate beyond quadratics.
    For instance,~\citet{wang2022provable} prove convergence of $\HB$ by implicitly assuming co-diagonalisation. As a consequence, their results do not hold in dimension larger than one. Moreover, we conjecture (see~\citet{goujaud2025open}), that the result also does not hold in dimension 1.
    This result is thereby complemented by this work, which shows we cannot achieve acceleration as soon as the dimension is at least two. Another example is that of~\citet[Corollary 22]{gupta2021path} whose proof relies on an hypothetical accelerated convergence rate of $\HB$.

    \paragraph{Optimal methods on $\Fml$.} Whereas convergence rates for $\HB$ were unclear on $\Fml$, there exist alternative optimal methods (which are, however, slower than $\HB$ on $\Qml$) on this class, commonly referred to as \emph{accelerated gradient methods}, such as~\eqref{eq:nag}, see~\citep{nesterov1983method,Nest03a}. Whereas the dependency w.r.t. $\kappa$ is essentially optimal for NAG (whose convergence rate is $(1-\sqrt{\kappa})^{1/2}$), it can be improved to $(1-\sqrt{\kappa})$ (see~\citep{taylor2023optimal} for the optimal algorithm and~\citep{van2017fastest} for its stationary version) thereby matching the exact lower complexity bound for $\Fml$~\citep{drori2022oracle} (which is more technical than that for $\Qml$, whose lower bound on the rate is $\left(\frac{1-\sqrt{\kappa}}{1+\sqrt{\kappa}}\right)$ which is thereby more often used).

    \subsection{Contributions} In this work, we show that the heavy-ball~($\HB$) cannot attain an accelerated (worst-case) convergence rate on the standard class of smooth strongly convex problems in general. We further show that this result is stable to additional assumptions. In particular, it remains true even under higher-order regularity assumptions, and for perturbed initial conditions and gradient computations. 

    More precisely, in \Cref{sec:preliminary-materials}, we recall a few known results on~$\HB$, and set up a few key concepts for building up the following sections, including those of \emph{cycling behaviors}. Next, in \Cref{sec:noaccel}, we show our main result on $\HB$, namely that we cannot obtain accelerated convergence guarantees for $\HB$ on standard classes of problems beyond quadratics. To obtain that result, we analyze parameters resulting in a simple cycle shape. 
    Additionally to this paper's figures, a 3D graph of our counter-example can be found by running the code available in the GitHub repository \href{https://github.com/bgoujaud/Heavy-ball_does_not_accelerate}{Heavy-ball\_does\_not\_accelerate}.

    Then, in \Cref{sec:robustness} we show that our non-acceleration results are stable to small perturbations of the initial iterates, of the parameters and of the gradients.
    Then, in \Cref{sec:HL}, we also show that adding additional natural regularity assumptions does not result in acceleration either.
    Finally, in \Cref{sec:cycles}, we detail a constructive approach for finding counter-examples to the convergence of $\HB$, and more generally of any \emph{stationary} first-order method. We demonstrate that if a parametrization results in a cycle, then it also results in a cycle having a very specific shape, and that numerically, this can be solved as a linear problem.

    \subsection{Key concepts}
    The following definition introduces the concept of \emph{asymptotic} convergence rate and of convergence rate of a method over a class of functions. Formally, this definition is necessary because we are looking for negative results on the value of $\rho$ throughout the paper. Furthermore, momentum-type methods such as $\HB$ or $\NAG$ (contrary to $\GD$) are not monotone method (i.e.,~$(\|x_t - x_\star\|)_t$ is not a decreasing sequence in general).

    \begin{Def}[(Asymptotic) convergence rate]\label{def:awcc}
        Let $\mathcal{F}$ be a class of functions. We say that a given first-order method has a worst-case asymptotic convergence rate $\rho$ over $\F$ if for all $ \varepsilon>0,$ there  exists $  T_0 $ such that for all $ f \in \F, x_0 $ and $ T\geq T_0$, for $(x_t)_{t\geq0}$ the sequence of iterates generated from $x_0$ by running the method on the function $f$, we have 
        \begin{equation*}
            \|x_T - \xs\| \leq (\rho + \varepsilon)^T \|x_0 - x_\star\|.
        \end{equation*}
    \end{Def}
    Informally, this means that $\|x_T-\xs\|$ is (almost) of the order of $\rho^T$ uniformly over the class $\F$. In the sequel, we may refer to the worst-case asymptotic rate as the \textit{rate}. Our interest is to understand the impact of parameters $(\gb)$ in $\HB$. For clarity, when necessary, we denote $\eqref{eq:hb}_{\gb}$ the heavy-ball method  with coefficients $(\gb)$, and $\eqref{eq:hb}_{\gb}(f)$ the heavy-ball method  with coefficients $(\gb)$ applied to the function $f$. We then introduce the following two definitions.

    \begin{Def}[Rate $\rho_{\gb}(\F)$]\label{def:rho_gam_bet}
        For any class $\F$ and any $(\gb) \in \R \times \R$, we denote $\rho_{\gb}(\F)$ the smallest worst-case asymptotic rate of $\eqref{eq:hb}_{\gb}$ on $\F$.
    \end{Def}

    \begin{Def}[Convergence region  $\OcvF$]\label{def:Omega_cv}
        We denote  $\OcvF$ the set of parameters $(\gb) \in \R\times \R$ for which $\eqref{eq:hb}_{\gb}$ has a worst-case asymptotic convergence rate  $\rho_{\gb}(\F)$ strictly below $1$.
    \end{Def}

    In the sequel, if $(\gb)\in \OcvF$, we may abusively say that $\eqref{eq:hb}_{\gb}$ \textit{converges} on $\F$ (instead of ``has a worst-case asymptotic convergence rate  $\rho_{\gb}(\F)$ strictly below $1$''), and conversely, that $\eqref{eq:hb}_{\gb}$ \textit{does not converge} if $(\gb)\notin \OcvF$ (i.e.,~there exists~$f\in\mathcal{F}$ on which $\HB_{\gb}$ does not converge).
    
    Before giving preliminary results on \eqref{eq:hb}, let us recall a few notations.
    
    \paragraph{Notation.} For $x\in \R^d$ and $r>0$,  $B(x, r)$  is the Euclidean ball with center $x$ and radius $r$, and $\|x\|$ the Euclidean norm of $x$. We denote $\conv(x_i, i\in I)$ the convex hull of a family of points $(x_i)_{i\in I}$ indexed by a set $I$. We denote $\mathrm{Int}(A)$ the interior of a set $A\subseteq \R^d$. For vectors $x, y\in \R^d$, $\langle x,y\rangle =x^\top y$ is the Euclidean inner product.
    
    We denote $\Id_d$ the identity  matrix in dimension $d$. We denote $\mathcal S_K(\R)$ the set of symmetric matrices in dimension $K\in \mathbb N$,  and $\mathcal S^+_K(\R)$ the set of positive semi-definite matrices. For $M,N$ in $\mathcal S_K(\R)$,  we denote $M\preccurlyeq N$ if $N-M\in \mathcal S_K^+(\R)$.  For a matrices $M, N\in \R^{K\times K'}$, we denote $\langle M,N\rangle =\Tr(M^\top N)$ the standard inner product, with $\Tr$ the trace operator. We denote $\|M\|_{\op}$ the operator norm of the matrix $M$, and $\mathrm{Sp}(M)$ the spectrum of $M$.
    
    We denote $\mathcal C^k (\R^d)$ the set of $k$ times continuously differentiable functions from $\R^d \to \R$. We denote $f^*$ the Fenchel-transform of a function~$f$. Note that all classes of functions considered hereafter belong to the set of closed proper convex functions, and hence satisfy $f=f^{**}$.
    
     Finally, for any integer $K\geq 2$, we denote by $\theta_K$ the angle $\frac{2\pi}{K}$. By convention, and to avoid unnecessary case disjunctions, for $\beta<0$, we use the convention $\sqrt{\beta}=\mathrm{NAN}$ (Not A Number), and $\min(a,\mathrm{NAN})=\max(a, \mathrm{NAN})=a$.

\section{Preliminary results on heavy-ball}\label{sec:preliminary-materials}

    This section summarizes a few well-known results and open questions regarding the heavy-ball method. We start by describing a link between convergence guarantees and the choice of parameters $(\gb)$ on the class of quadratic functions $\mathcal Q_{\mu, L}$. Those results are well-known nowadays and date back to at least~\cite{polyak1964some}. We leverage them in~\Cref{sec:cycles}.
    
    \subsection{Known behavior of the heavy-ball method on quadratics \texorpdfstring{($\mathcal{Q}_{\mu, L}$)}{(QmL)}}\label{subsec:hb_on_quad}
    
        In this section, we consider a function $f_H\in \Qml$ parameterized by its Hessian matrix~$H$, i.e.,~such that $f(x)-\fs =\frac{1}{2} (x-\xs)^\top H (x-\xs)$ with $\mu \Id \preccurlyeq H\preccurlyeq L \Id$---or equivalently $\Sp(H)\in[\mu, L]$. The heavy-ball update is:
        \begin{align}\label{eq:HB-quad}
            x_{t+1} & ~ = x_t - \gamma \nabla f(x_t) + \beta (x_t - x_{t-1})  = x_t - \gamma H(x_t - x_\star) + \beta (x_t - x_{t-1}). \tag{HB-Q}
        \end{align}
        The worst-case asymptotic convergence rate (see~\Cref{def:awcc}) of~\eqref{eq:HB-quad} is provided by the following.

        \begin{restatable}{Prop}{quadrates}
        \textbf{\emph{(\citet{polyak1964some})}} \label{prop:conv_quad}
            Consider $\beta \in \mathbb R$ and $\gamma \in \mathbb R $.
            The worst-case asymptotic convergence rate $\rhogbQml$ of  $\eqref{eq:HB-quad}_{\gb}$, over the class $\Qml$ is:
            \begin{enumerate}[leftmargin=*]
                \item \textbf{Lazy region}: If $0 < \gamma \leq \min\left(\frac{2(1+\beta)}{L+\mu}, \frac{\left(1 - \sqrt{\beta}\right)^2}{\mu}\right) $, then $\rhogbQml = \frac{1 + \beta - \mu\gamma}{2} + \sqrt{\left(\frac{1 + \beta - \mu\gamma}{2}\right)^2-\beta}$.
                \item \textbf{Robust region}: If $\beta\geq 0$, and  $\frac{\left(1 - \sqrt{\beta}\right)^2}{\mu}\leq\gamma\leq\frac{\left(1 + \sqrt{\beta}\right)^2}{L}$,  then $ \rhogbQml = \sqrt{\beta}$.
                \item \textbf{Knife's edge}: If $\max\left(\frac{2(1+\beta)}{L+\mu}, \frac{\left(1 + \sqrt{\beta}\right)^2}{L}\right) \leq \gamma < \frac{2(1+\beta)}{L}$ , then  $ \rhogbQml =\frac{L\gamma - (1 + \beta)}{2} + \sqrt{\left(\frac{L\gamma - (1 + \beta)}{2}\right)^2-\beta}$.
                \item \textbf{No convergence}: if $\tfrac{\gamma}{1 + \beta} \geq \frac{2}{L}$ or $\gamma \le 0$, then $\rhogbQml >1$.
            \end{enumerate}
        \end{restatable}

        \begin{figure}[t]
          \begin{subfigure}[t]{.49\linewidth}
            \centering
            \includegraphics[width=\linewidth]{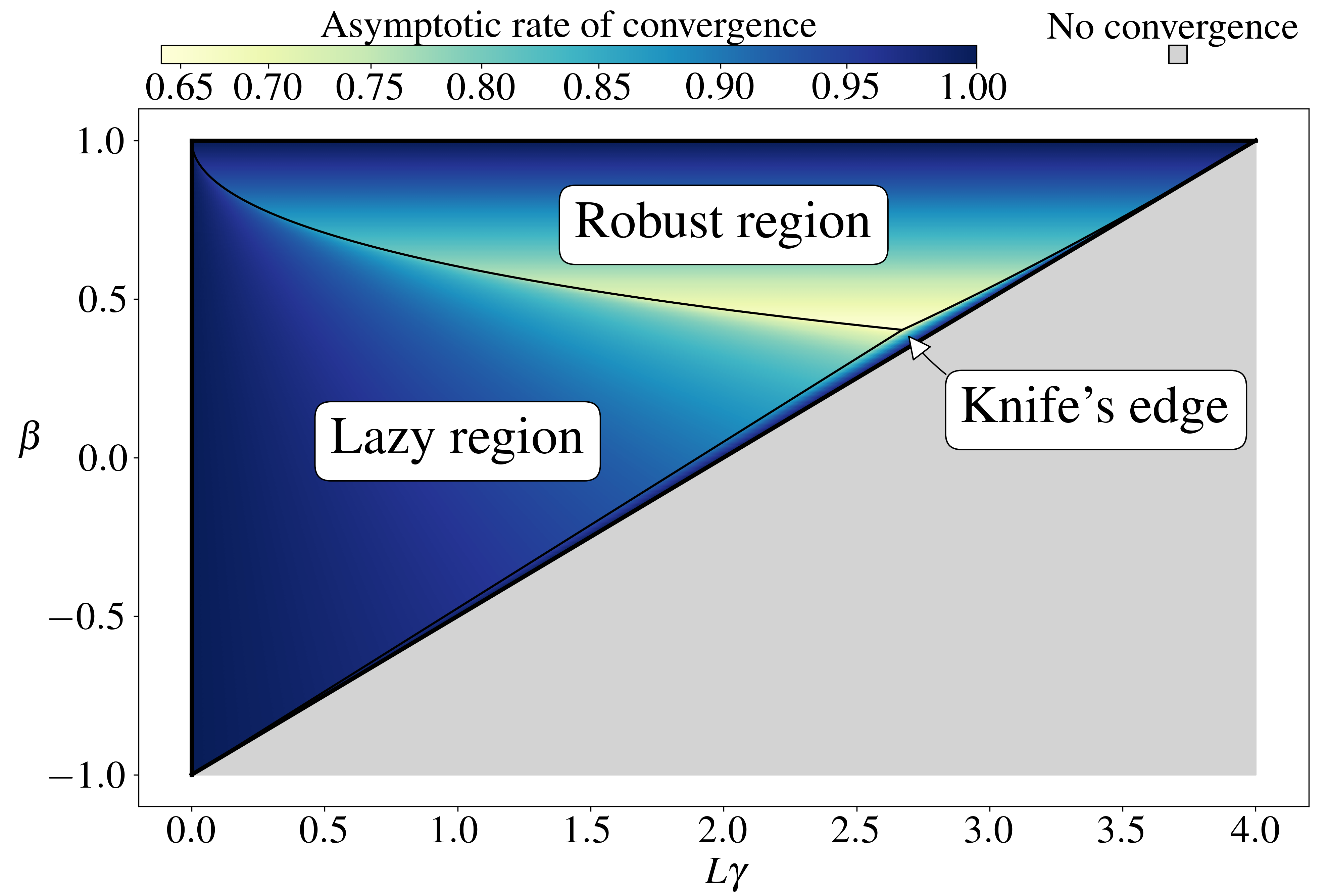}
        
            \caption{Schematic view of the three convergence regions described by \Cref{prop:conv_quad} together with the asymptotic worst-case convergence rate $\rho_{\gb}(\Qml)$ as a color-scale, with respect to $\gamma$ and $\beta$.}
            \label{fig:1a}
          \end{subfigure}
          \hfill
          \begin{subfigure}[t]{.49\linewidth}
            \centering
            \includegraphics[width=\linewidth]{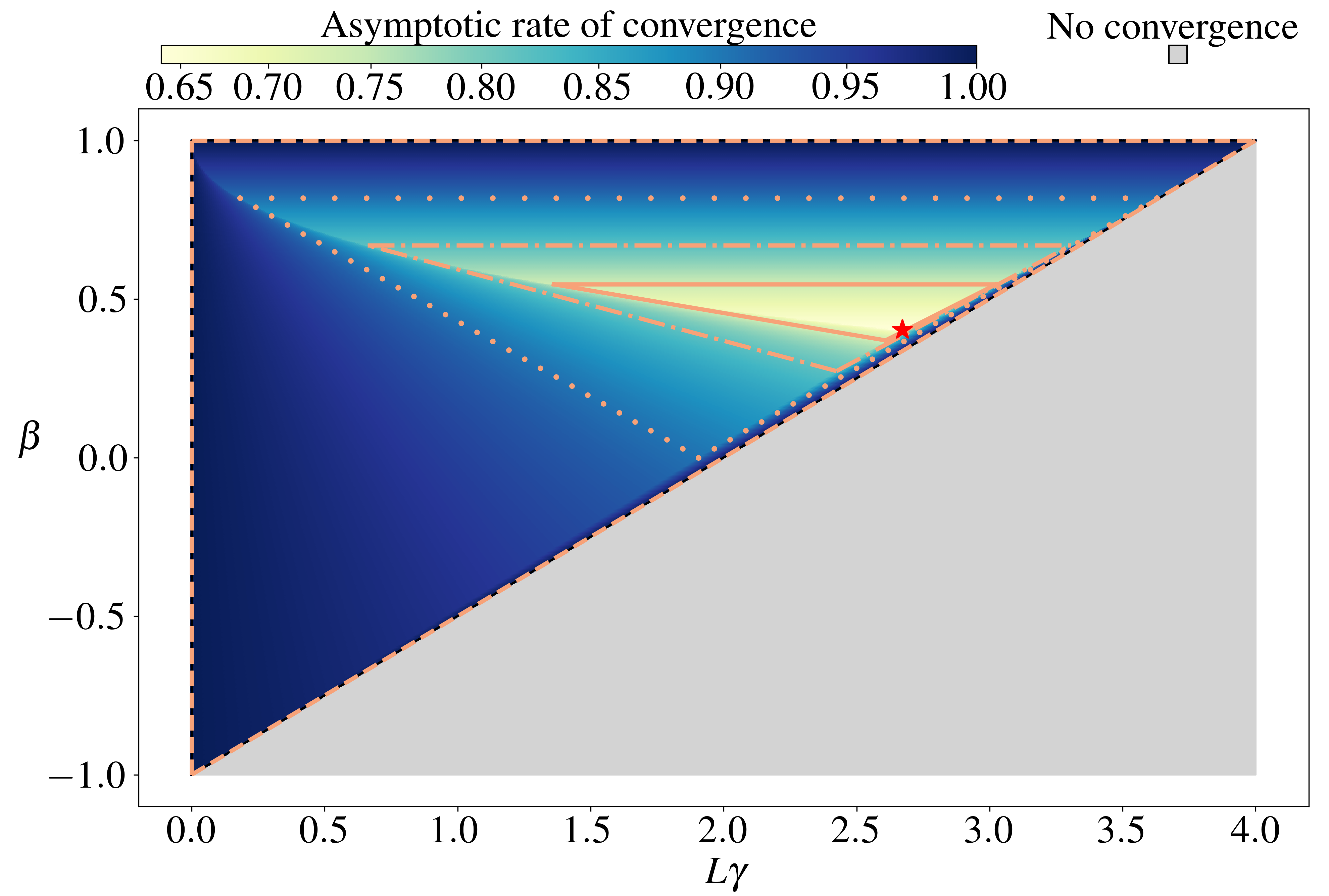}
            \subcaption{Level sets $\LS_{\ml}(\rho) $ of  $\gb \mapsto \rho_{\gb}(\Qml)$ as orange triangles. Levels $\rho$ correspond to $\rho=1$ {(\textcolor{orange}{-\,-})}, ${\rho=\frac{1-\kappa}{1+\kappa}}$~(\textcolor{orange}{$\cdots$}) corresponding to the rate of $\GD$, then ${\rho=\frac{1-2\kappa}{1+2\kappa}}$ (\textcolor{orange}{$- \cdot -$}), and   ${\rho=\frac{1-3\kappa}{1+3\kappa}}$ (\textcolor{orange}{---}), and finally ${\rho=\frac{1-\sqrt{\kappa}}{1+\sqrt{\kappa}}}$ (\textcolor{red}{$\star$}), i.e.,~$\rho^\star(\Qml)$.}
            \label{fig:hb_heatmap}
            \label{fig:1b}
          \end{subfigure}
          \caption{Asymptotic convergence rate $\rho_{\gb}(\Qml)$ of \eqref{eq:HB-quad}, for $\kappa = 1/20$.}
          \label{fig:1}
        \end{figure}

        \begin{sketch}
            The complete proof is provided in~\Cref{app:proofquad}. In short, we rewrite \eqref{eq:HB-quad} as a linear system and decompose it over the eigenspaces of~$H$.  Ultimately, the asymptotic convergence rate is given by  $\max_{\lambda \in [\mu, L]} \bar \rho (P_{\beta, \lambda,\gamma})$ with $ P_{\beta, \lambda,\gamma}\eqdef
            \begin{pmatrix}
                1 + \beta - \gamma \lambda & - \beta \\
                1 & 0
            \end{pmatrix}.$
            The four cases arise from the nature of the eigenvalues of $P_{\beta, \lambda,\gamma}$, that can either be two complex conjugates number with modulus $\sqrt{\det(P_{\beta, \lambda,\gamma})}=\sqrt{\beta}$ or two real numbers, and the fact the $\max$ may be attained on either $\mu $ or $L$. 
        \end{sketch}

        \paragraph{Comments on \Cref{prop:conv_quad}.}  \Cref{fig:hb_heatmap} illustrates the asymptotic rate of the heavy-ball method for each value of the parameters $\gb$.  It shows the three parameter regions resulting in convergence, as  given by \Cref{prop:conv_quad}. 
        First, the {\bfseries left} region, called the {\bfseries lazy region} where the step-size is small, thus the rate is driven by the convergence of the iterates' component aligned with the eigenvector of $H$ associated with $\mu$.
        Second, the {\bfseries right} region, called the {\bfseries knife's edge} where the step-size is large and  the rate is driven by the oscillations of the iterates' component aligned with the eigenvector of $H$ associated with $L$.
        Third, the {\bfseries upper} region, called  the {\bfseries robust region} where the step-size does not impact the convergence rate. 

        In particular, \Cref{prop:conv_quad} enables to define the set of parameters for which \eqref{eq:HB-quad} converges. Following \Cref{def:Omega_cv}, we denote by~$\Oqml$ the set of parameters $(\gb)$ for which~\eqref{eq:HB-quad} has a worst-case asymptotic convergence rate strictly below $1$ (i.e.,~for which $\HB$ converges on any function of $\mathcal{Q}_{\ml}$). This set is naturally the union of three regions of convergence provided by~\Cref{prop:conv_quad}.

        \begin{Cor}
            By \Cref{prop:conv_quad}, we have $$\Oqml = \left\{(\gb)\in \R\times \R \text{ s.t. } \beta\in (-1;1),  0< {\gamma}{} < \frac{2}{L}(1 + \beta) \right\}.$$
        \end{Cor}
        Furthermore, the optimal parameter choice is achieved at the intersection of the three regions (or equivalently at point of the robust region with the smallest $\beta$), as provided by the following result.
        
        \begin{Cor}\label{cor:HB-quad-opt-tun}
            The optimal worst-case asymptotic rate of \eqref{eq:HB-quad} on $\Qml$, for parameters  $(\gb) \in \Oqml$ is
            \[\rho^\star(\Qml) \eqdef \min_{(\gb)\in \Oqml } \rho_{\gb}(\Qml) = \sqrt{\beta^\star(\Qml)}=\frac{1-\sqrt{\kappa}}{1 + \sqrt{\kappa}}, \] 
            which is achieved for $\beta^\star(\Qml) \eqdef  \left(\frac{1-\sqrt{\kappa}}{1 + \sqrt{\kappa}}\right)^2$, $\gamma^\star(\Qml) \eqdef\frac{2}{L+\mu}  (1+\beta^\star(\Qml))$.
        \end{Cor}
        In a nutshell, on the one hand, for $\beta \leq \left(\tfrac{1 - \sqrt{\kappa}}{1 + \sqrt{\kappa}}\right)^2$, the optimal rate is achieved for a single value of~$\gamma$, such that $\gamma = \frac{2}{L+\mu} (1+\beta)$, that corresponds to the limit between the lazy region and  the  knife's edge. In this region $\beta \in \left[0; \left(\tfrac{1 - \sqrt{\kappa}}{1 + \sqrt{\kappa}}\right)^2\right]$, the rate decreases (improves) as $\beta$ increases. 
        On the other hand, in the robust region,  i.e.,~when~$\beta > \left(\tfrac{1 - \sqrt{\kappa}}{1 + \sqrt{\kappa}}\right)^2$, the asymptotic rate $\sqrt{\beta}$ is achieved for any $\gamma\in\left[ \frac{\left(1 - \sqrt{\beta}\right)^2}{\mu}, \frac{\left(1 + \sqrt{\beta}\right)^2}{L} \right]$ (which allows to use any value in this set). In this region $\beta \in \left[ \left(\tfrac{1 - \sqrt{\kappa}}{1 + \sqrt{\kappa}}\right)^2; 1 \right]$, the rate increases (degrades) with $\beta$. The optimal rate is thus achieved at the limit, $\beta = \left(\tfrac{1 - \sqrt{\kappa}}{1 + \sqrt{\kappa}}\right)^2$. 

        This asymptotic rate matches the lower complexity bound provided by~\cite{nemirovskij1983problem}.
        As~$\kappa\to 0$, $\rho^\star(\Qml) \sim 1-2\sqrt{\kappa}$ which is commonly referred to as an accelerated convergence rate, as compared to that of the classical gradient descent algorithm, obtained with $(\gamma=2/(L+\mu),\beta=0)$, whose rate is $\rho_{\gamma=2/(L+\mu),\beta=0}(\Qml) \sim 1-2\kappa$ (as $\kappa\to 0$).
        
        The level sets of the asymptotic convergence rate of \eqref{eq:hb} are triangles, as stated in the following lemma and illustrated on \Cref{fig:hb_heatmap}. This property will be used in the proof of our main result, in \Cref{sec:noaccel}. 
        \begin{Lemma}[Sublevel set $\SLS_{\ml}(\rho)$ of the heavy-ball convergence rates]
            \label{lem:level_sets}
            Let $0<\mu \le L$. The level sets of $\gb\mapsto\rhogbQml$ are triangles.
            More precisely, for any $\rho \in [\rho^\star(\Qml), 1]$, the set of parameters $\gb$ for which~$\eqref{eq:hb}_{\gb}$ has rate $\rho$ is the union of the three segments parametrized by:
            \begin{itemize}
                \item Segment in the Lazy Region: $\beta \in \left[\frac{\frac{1-\kappa}{1+\kappa}-\rho}{\frac{1}{\rho} - \frac{1-\kappa}{1+\kappa}}, \rho^2\right]$ and $\gamma = \frac{(1-\rho)(1-\beta / \rho)}{\mu}$  .
                \item Segment in the Robust Region: $\gamma \in \left[\frac{\left(1 - \rho\right)^2}{\mu}, \frac{\left(1 + \rho\right)^2}{L}\right]$ and $\beta = \rho^2$.
                \item Segment in the Knife Edge: $ \beta \in \left[\frac{\frac{1-\kappa}{1+\kappa}-\rho}{\frac{1}{\rho} - \frac{1-\kappa}{1+\kappa}}, \rho^2 \right]$ and $\gamma = \frac{(1+\rho)(1+\beta / \rho)}{L}$.
            \end{itemize}        
            We denote this level set by $\LS_{\ml}(\rho)$ (which is a triangle), and the corresponding sublevel set (that is extensively used in the sequel) by $ \SLS_{\ml}(\rho) =\cup_{\rho'\leq\rho} \LS_{\ml}(\rho')$.
        \end{Lemma}
        As a summary, this section provided a complete picture of the behavior of \eqref{eq:hb} over~$\Qml$. As the convergence rate for $\beta<0$ is never better than the one for $\beta=0, \gamma=\frac{2}{L+\mu}$, we restrict the analysis to $\beta\geq 0$ in the following. Next, we move to existing results on the class $\Fml$.

    \subsection{Known behaviors of the heavy-ball method on \texorpdfstring{$\mathcal{F}_{\mu, L}$}{Fml}}\label{subsec:convergence_hb}

        Convergence of \eqref{eq:hb} on the set of $L$-smooth and $\mu$-strongly convex functions $\Fml$ has attracted a lot of attention over the last decade. First,
        recall that we denote  $\Ocvml$ the set of parameters $(\gb)$ for which \eqref{eq:hb} has a worst-case asymptotic convergence rate strictly below 1 on $\Fml$ (see~\Cref{def:Omega_cv}).

        \subsubsection{Convergence results on \texorpdfstring{$\Fml$}{Fml}}\label{subsubsec:convergence_results}
            \citet{ghadimi2015global} establish that \eqref{eq:hb} converges on $\Fml$ when
            \[\gamma\in(0, \tfrac{2}{L}) \text{ and } 0\leq \beta < \tfrac{1}{2}\left(\frac{\mu\gamma}{2} + \sqrt{\left(\frac{\mu\gamma}{2}\right)^2 + 4(1 - \tfrac{L\gamma}{2})}\right),\] see~\citep[][Theorem.4]{ghadimi2015global}. For comparison purposes, we denote this set of parameters $\Oghml$ in what follows.
            Unfortunately, this result does not lead to an acceleration of \eqref{eq:hb} on $\mathcal F_{\ml}$.
            Indeed, the following lemma shows that the best rate for $(\gb)\in \Oghml$ parameters is not accelerated, \textit{even on $\Qml$}.
            \begin{Lemma}[Optimal asymptotic rate of \eqref{eq:HB-quad} on $\Qml$ for $ (\gb) \in \Oghml$]\label{lem:ghad}
                The optimal worst-case asymptotic rate of \eqref{eq:HB-quad} on $\Qml$, for parameters  $(\gb) \in \Oghml$ is
                \[\rho^\star_{\Ghadimi}(\Qml) \eqdef \min_{(\gb)\in \Oghml } \rho_{\gb}(\Qml) = \sqrt{\beta^\star_{\Ghadimi} (\Qml)}\underset{\kappa \rightarrow 0}{=} 1 - 8\kappa + o(\kappa),\] 
                which is achieved for
                \begin{align*}
                    \sqrt{\beta^\star_{\Ghadimi}(\Qml)} &= (\kappa^{-1}-1)^{1/3}\left[ \left(\sqrt{\frac{\kappa^{-1}+26}{27}}+1\right)^{1/3} - \left(\sqrt{\frac{\kappa^{-1}+26}{27}}-1\right)^{1/3} \right] - 1, \\
                    \gamma^\star_{\Ghadimi}(\Qml) &= \frac{2(1+\beta^\star_{\Ghadimi}(\Qml))}{L+\mu} .
                \end{align*}
            \end{Lemma}
            
            \begin{sketch}
                The result of \citet{ghadimi2015global} corresponds to using a Lyapunov function of the form $V_t = f(x_t) - f_\star + A(f(x_{t-1}) - f_\star) + B\| x_t - x_{t-1} \|^2$ with $A, B \geq 0$.
            \end{sketch}

            \begin{figure}[t]
                \begin{subfigure}[t]{.3\linewidth}
                    \centering
                    \includegraphics[width=\linewidth]{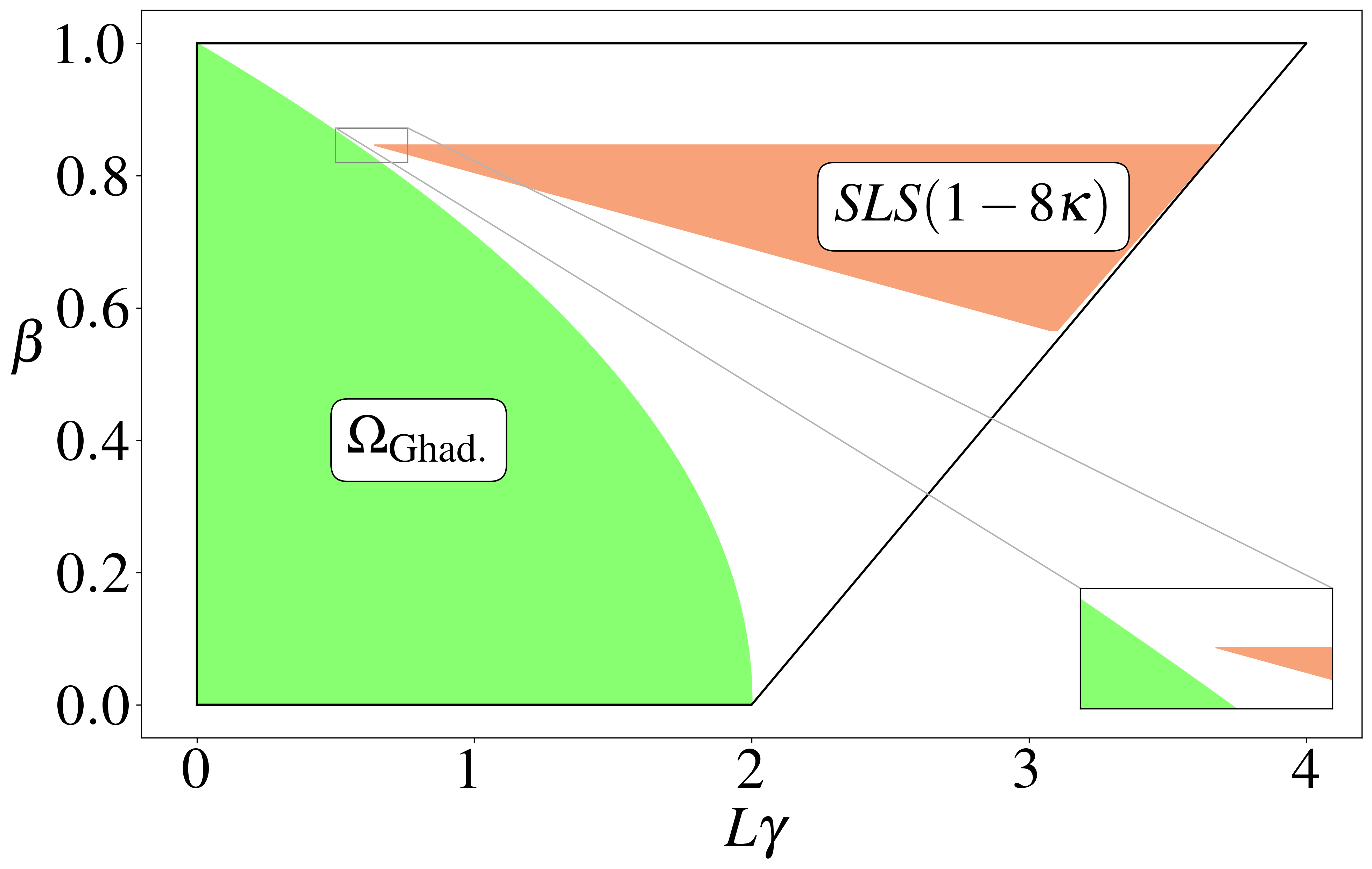}
                    \caption{$\kappa=0.01$}
                    \label{fig:ghad_a}
                \end{subfigure}
                \hfill
                \begin{subfigure}[t]{.3\linewidth}
                    \centering
                    \includegraphics[width=\linewidth]{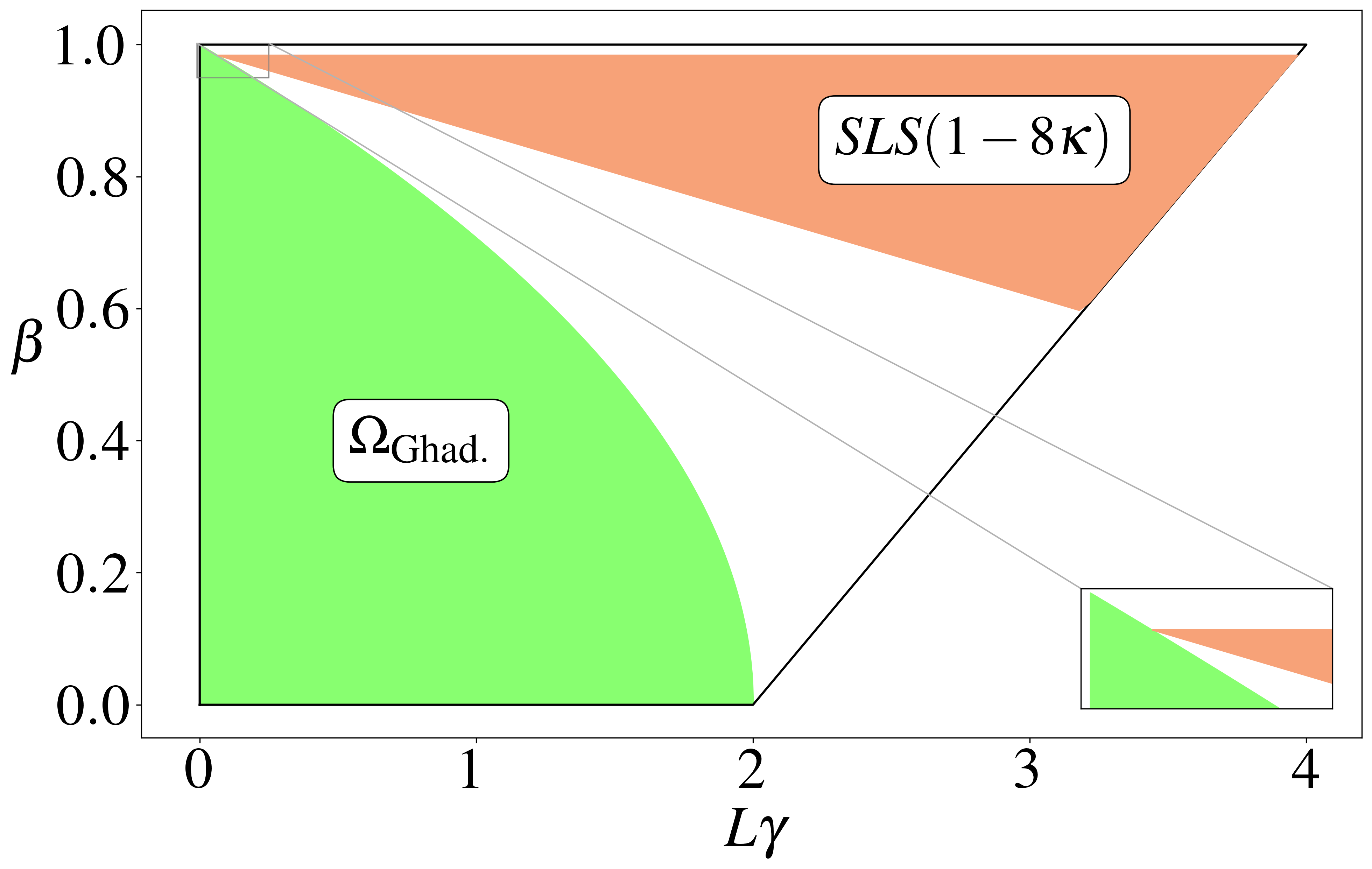}
                    \caption{$\kappa=0.001$}
                    \label{fig:ghad_b}
                \end{subfigure}
                \hfill
                \begin{subfigure}[t]{.3\linewidth}
                    \centering
                    \includegraphics[width=\linewidth]{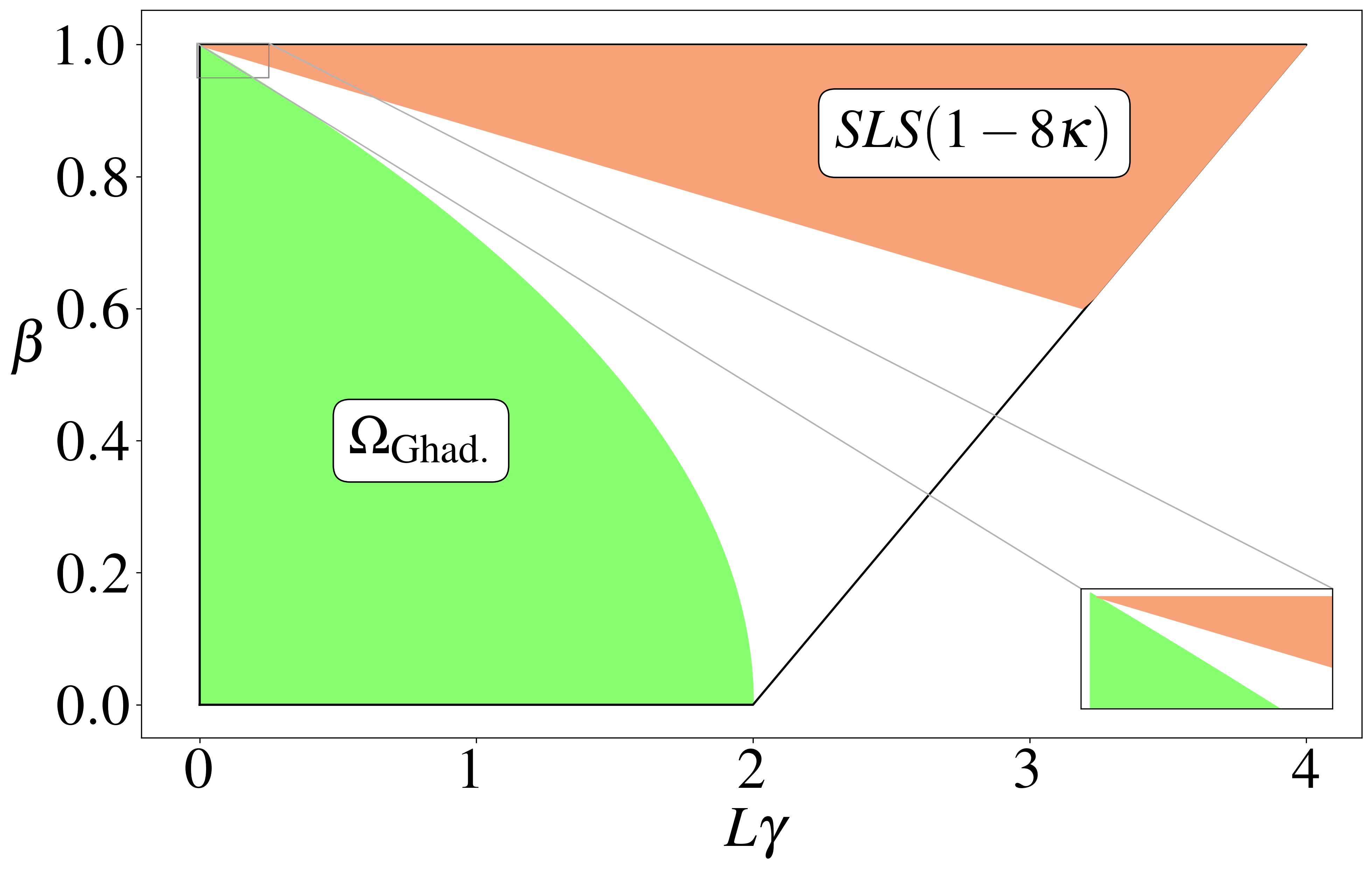}
                    \caption{$\kappa=0.0001$}
                    \label{fig:ghad_c}
                \end{subfigure}
                \caption{Illustration of \Cref{lem:ghad}. Region of parameters $\Oghml$ and sublevel set $\SLS_{\ml}\left({1-8\kappa}\right)$ for three values of $\kappa$: the rate of $\rho_{\gb}(\Qml)$ on  $\Oghml$ is at best $\frac{1-4\kappa}{1+4\kappa}$.}
                \label{fig:ghad}
            \end{figure}

            \vspace{1em} \noindent
            In short, the set $\Oghml$ of parameters covered from~\citep{ghadimi2015global} is not large enough to guarantee acceleration of \eqref{eq:hb}: indeed for any $C>8$, there exists $\kappa_0$ such that for all $\kappa<\kappa_0$, $\Oghml$ does not intersect the sublevel set $\SLS_{\ml}(1-C\kappa)$ given by \Cref{lem:level_sets}. This is illustrated in \Cref{fig:ghad}, with $C=8$ and for 3 different~$\kappa$. Therefore, a natural question is that of properly identifying the set $\Ocvml$. There exist a few approaches for trying to get better approximations to this set; see, e.g.,~the work~\citep{taylor2018lyapunov} which provides a tool to numerically identify valid Lyapunov functions---see~\cite{goujaud2023counter} for a detailed treatment of $\HB$ with this technique.

        \subsubsection{Non-convergence results on \texorpdfstring{$\Fml$}{Fml}}
            In parallel to the positive results of~\Cref{subsubsec:convergence_results},
            that establish that some points provably belong to $\Ocvml$,
            another line of work has focused on demonstrating that some points provably do not belong to $\Ocvml$.
            As we will see later on, this is the approach we will build upon.
            Two main results have to be mentioned in that direction:~\citep{lessard2016analysis} (expliciting a counter-example for the single tuning corresponding to the optimal one on the class of quadratic functions),
            and~\citep{goujaud2023counter} (automating the search for counter-examples on any tunings for which there exists one, but only numerically), that we review hereafter.

            \paragraph{Non-convergence for the optimal tuning $(\gamma^\star(\Qml), \beta^\star(\Qml))$ on quadratics.}
            
            \citet{lessard2016analysis} prove that for the optimal tuning $(\gamma^\star(\Qml), \beta^\star(\Qml))$ on $\Qml$ given by \Cref{cor:HB-quad-opt-tun}, for $\kappa=1/25$, there exist an $L=25$-smooth and $\mu=1$-strongly convex function such that if $x_0$ is in a specific neighborhood, then the iterates generated by~\eqref{eq:hb} oscillate between the neighborhoods of three values, and thereby never converge towards~$\xs$.
            
            At this stage, it is important to note that this counter-example does not exclude the existence of another tuning~$(\gb)$ for which an accelerated convergence rate is achieved. Indeed, there is no reason for the optimal tuning on~$\Fml$ to  correspond to that on~$\Qml$.
            
            \paragraph{Non-convergence on multiple tunings.}
            Recently, \citet{goujaud2023counter} proposed a numerical approach to compute cyclic trajectories of various first-order methods that include $\HB$. For~\eqref{eq:hb}, given the period $K\geq 2$ of the cycles, this technique consists in solving the following optimization problem which can be cast and solved as an SDP:
            \begin{equation}
                \left|
                \begin{array}{cc}
                    \underset{\substack{d \geq 1, f \in\Fml \\ (x_t)_{t}\text{ is generated by } \eqref{eq:hb}}}{\text{minimize }} &  \|x_0 - x_{K}\|^2 + \|x_1 - x_{K+1}\|^2 \\
                    \text{subject to } & 
                    \|x_1 - x_0\|^2 \geq 1.
                \end{array}
                \right.
                \tag{$\mathcal{P}$} \label{eq:cycle_search_problem}
            \end{equation}
            \citet{goujaud2023counter} prove that the value of the optimization problem \eqref{eq:cycle_search_problem} is 0 if and only if there exists a function $f \in {\Fml} $ and an initialization $(x_0, x_{1}) \in  (\mathcal X)^2$ such that the method $\eqref{eq:hb}(f)$ initialized at $(x_0, x_1)$ cycles on $K$ values, i.e.~that the sequence of iterates generated is $(x_0, \dots, x_{K-1}, x_0, \dots, x_{K-1}, x_0, \dots)$.

            Although these negative results are limited to either a single $\kappa$ and tuning $(\gamma^\star(\Qml), \beta^\star(\Qml))$ for \citep{lessard2016analysis}, or only numerical in \citep{goujaud2023counter}, analyzing the parameter choices for which the worst-case uniform convergence of \eqref{eq:hb} on $\Fml$ can be disproved by establishing the existence of a cycle appears to be a promising direction.

        \subsection{Our approach to comprehensive behaviors of heavy-ball}

            In this section we therefore introduce $\Ocyml$ as the set of parameter values $(\gb)$ for which \eqref{eq:hb} cycles on a function of $\Fml$.

            \begin{Def}[Cycles]\label{def:cycle}
                Let $(\gb)\in \Oqml$, and $\mathcal{F}$ a class of functions.
                \begin{enumerate}[leftmargin=*]
                  \item  \label{item:1}
                For $K$ a positive integer, referred to as the \textit{period}, $(x_{t})_{t\in\range{0}{K-1}} \neq (x_0, \dots, x_0)$ a family of $K$ points not all-equal, and $f$ a function, we say that 
                $$\eqref{eq:hb}_{\gb}(f)  \text{ cycles on } (x_{t})_{t\in\range{0}{K-1}}$$ 
                if the sequence $(z_t)_{t\in\mathbb{N}}$ generated by \eqref{eq:hb} applied on $f$ with initial points $z_0=x_0$ and $z_1=x_1$ cycles on $(x_{t})_{t\in\range{0}{K-1}}$, i.e.,~verifies $\forall t\geq 0, z_t = x_{t  \mod{K}}$.
                \item \label{item:2}  Moreover, for such a $K$ and family of $K$ points $(x_{t})_{t\in\range{0}{K-1}} \neq (x_0, \dots, x_0)$,
                we say that  $$\eqref{eq:hb}_{\gb} \text{ cycles on } (x_{t})_{t\in\range{0}{K-1}} \text{ on } \mathcal F$$ 
                if there  exists an  $f\in \mathcal{F}$, for which 
                $\eqref{eq:hb}_{\gb}(f)$ cycles on  $(x_{t})_{t\in\range{0}{K-1}}$.
                \item \label{item:3} Finally, we say that  $$\eqref{eq:hb}_{\gb}  \text{ has a cycle on } \mathcal F$$ 
                if there  exist such a period  $K$, and cycle $(x_{t})_{t\in\range{0}{K-1}}\neq (x_0, \dots, x_0)$, and  $f\in \mathcal{F}$, for which 
                $\eqref{eq:hb}_{\gb}(f)$ cycles on  $(x_{t})_{t\in\range{0}{K-1}}$.
                \end{enumerate}  
            \end{Def}
            Note that we exclude the constant cycle $ (x_0, \dots, x_0)$, that would correspond to a trivial cycle $(x_\star, \dots, x_\star)$ of~$\eqref{eq:hb}(f)$ for any function $f$ such that $x_0  =x_\star = \argmin f$.
            This corresponds to non-problematic situations as the algorithm already converged. 
            We underline the  following equivalent point of  view on a cycle.
            \begin{Rem}\label{rem:equiv_pov_cycle}
                $\eqref{eq:hb}_{\gb}(f)$ cycles on $(x_{t})_{t\in\range{0}{K-1}}$ if and only if for any $s\in \range{0}{K-1}$,  $x_{s+1} = x_s - \gamma \nabla f(x_s) + \beta (x_s - x_{s-1})$, where the sequence $(x_t)_t$ is extended $K$-periodically to $t\in \mathbb Z$  as $(x_t)_{t\in \mathbb Z} \eqdef (x_{t\mod{K}})_{t\in  \mathbb Z}$ (in particular as $x_K\eqdef x_0$ and $x_{-1}\eqdef  x_{K-1}$).
            \end{Rem} 
            From \Cref{def:cycle}, we define the region $\Ocyml$.
            \begin{Def}[Cycling region $\Ocyml$]\label{def:cycle_region}
                 For any $0<\mu\le L $, we denote:
                 \begin{enumerate}[leftmargin=*]
                    \item  $ \Ocyml$ the subset of $\Oqml$ for which $\eqref{eq:hb}$ has a cycle on $\Fml$.
                    \item $ \Ocyml^c$ the complementary of $\Ocyml$ in $\Oqml$.
                \end{enumerate}
            \end{Def}
            
            In the following, we leverage that for any parameters $(\gb)$ for which $\eqref{eq:hb}_{\gb}$  has a cycle on  $\Fml$, then the method does not converge. 
            \begin{Fact}\label{fact:Cv_in_cycle_compl}
                The set of parameters for which \eqref{eq:hb} has a worst-case (asymptotic) convergence rate (strictly below $1$) on $\Fml$ is included in $\Ocyml^c$:
                $$\Ocvml\subseteq\Ocyml^c$$
            \end{Fact}
            
            In \Cref{sec:noaccel}, we demonstrate that \eqref{eq:hb} cannot accelerate by focusing on a particular cycle shape. We study the set of parameters such that there exists a function in $\Fml $ that cycles over that particular set of iterates. \Cref{sec:cycles} explains why such a choice of a cycle is in fact natural.

    \section{Non-acceleration of heavy-ball on \texorpdfstring{$\mathcal F _{\mu,L}$}{Fml} via two-dimensional cycles}
    \label{sec:noaccel}

        In this section, we demonstrate the main result of the paper, which is that $\HB$ method does not accelerate on the class $\Fml$. 
        To obtain this result, we introduce  in \Cref{subsec:ROU_cycles} a simple two-dimensional cycle of length $K$ and study the set of $(\gb)$ such that there exists a function in $\Fml$ that cycles over those specific iterates. Then,  in \Cref{subsec:noaccel_main}, for some appropriate $C>0$, we show that the sublevel set of level $1-C\kappa$ of $\eqref{eq:hb}$ on the set of quadratic function is excluded from the parameters that do not have such a cycle.

    \begin{table}[b]
        \caption{Summary of parameter regions for which convergence is established or disproved.      \label{tab:not_regions}}
        \centering
        \begin{tabular}{ll}
            \toprule
            Notation & Region \\
            \midrule
            $\Oqml $ & Convergence on $\Qml$.\\
            $\Ocvml$ & Convergence on $\Fml$ \\
            $\Oghml$ & Convergence is established by \citet{ghadimi2015global}\\
            $\Ocyml$ & Subset of $\Oqml $ where \eqref{eq:hb} has a cycle on $\Fml$\\
            $\Orouml$ & Subset of $\Oqml $ where \eqref{eq:hb} has a roots-of-unity cycle. \\
            \bottomrule
        \end{tabular}
    \end{table}
    
\subsection{Studying a specific type of cycling behavior}\label{subsec:ROU_cycles}
We focus on the cycles that are supported by the $K$-th roots of unity.
\begin{Def}[Roots-of-unity cycle]
\label{def:rou_cycle}
    For $K\in \mathbb N$, and $\theta_K\eqdef \tfrac{2\pi}{K}$, we define the roots-of-unit cycle as $$\TikCircle_K = (x_0^\circ, x_1^\circ,\dots, x_t^\circ \dots, x_{K-1}^\circ) \eqdef \left(\begin{pmatrix}
        1 \\ 0
    \end{pmatrix},
    \begin{pmatrix}
        \cos{\theta_K} \\ \sin{\theta_K}
    \end{pmatrix},
    \dots,
    \begin{pmatrix}
        \cos{t\theta_K} \\ \sin{t\theta_K}
    \end{pmatrix},
    \dots,
    \begin{pmatrix}
        \cos{(K-1)\theta_K} \\ \sin{(K-1)\theta_K}
    \end{pmatrix} \right).
    $$ 
    We introduce the rotation operator $R = \begin{pmatrix}
          \cos{\theta_K} & -\sin{\theta_K} \\ 
          \sin{\theta_K} & \cos{\theta_K}
    \end{pmatrix}$ such that for any $t\in \llbracket {1};K-1\rrbracket$, $x_t^\circ = R  x_{t-1}^\circ =  R^t  x_{0}^\circ $ and $R^K=\Id$.
\end{Def}
This corresponds to a completely symmetrical cycle shape. Such a cycle is pictured in \Cref{fig:cycle}. We now introduce the set of parameters $(\gb)$ for which~\eqref{eq:hb} results in such a cycle on at least one function in $\Fml$.

\begin{Def}[Roots-of-unity cycling region $\Orouml$]\label{def:rou_cycle_region}
    For any $0<\mu\le L,$ we define
    \begin{enumerate}[leftmargin=*]
        \item  for any  $ K\in \mathbb N$,
        $ \OKrouml $ the subset of $\Oqml$ for which $\eqref{eq:hb}$  cycles on $\TikCircle_K$ on $\Fml$ (in the sense of \Cref{def:cycle}, item \ref{item:2}).
        
        \item $\displaystyle \Orouml = \bigcup_{K=2}^\infty \OKrouml$.
        
        \item $\displaystyle  (\Orouml)^c = \Oqml \backslash \Orouml$ the complementary  of $\Orouml $ in $\Oqml$.
    \end{enumerate}
\end{Def}
\begin{figure}[b]
    \centering
    \begin{minipage}[t]{0.48\textwidth}
        \centering
        \includegraphics[width=0.9\linewidth]
        {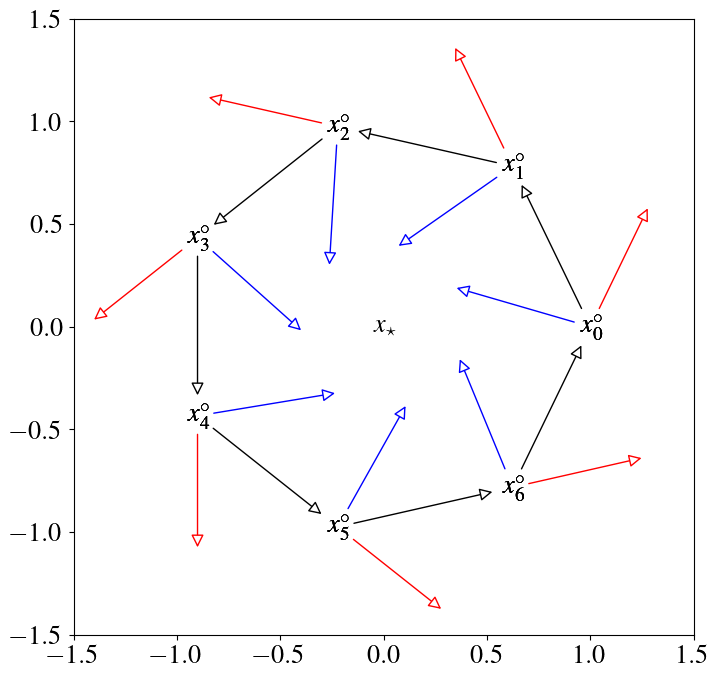}
        \caption{\label{fig:cycle} Cycle $ \TikCircle_7= (x^\circ_0, \dots, x^\circ_{6})$ of the $K=7^{\text{th}}$-roots-of-unity.
        For~$(\gb) \in {\Omega}_{7\text{-}\circ\text{-}\cycle}(\Fml)$, the red  arrows~\textcolor{red}{$(\rightarrow)$} correspond momentum component of~$\eqref{eq:hb}_{\gb}$ and the blue  arrows~\textcolor{blue}{$(\rightarrow)$} to the gradients of~$\psi^K_{\gb,\ml}$ such the~$\eqref{eq:hb}_{\gb}(\psi^K_{\gb,\ml})$ cycles over~$\TikCircle_7$.
        Here, $L=1$, $\mu=0.005$, $\gamma=3.5$ and $\beta=0.75$.
        }
    \end{minipage}
    \hfill
    \begin{minipage}[t]{0.48\textwidth}
        \centering
        \includegraphics[width=0.9\linewidth]{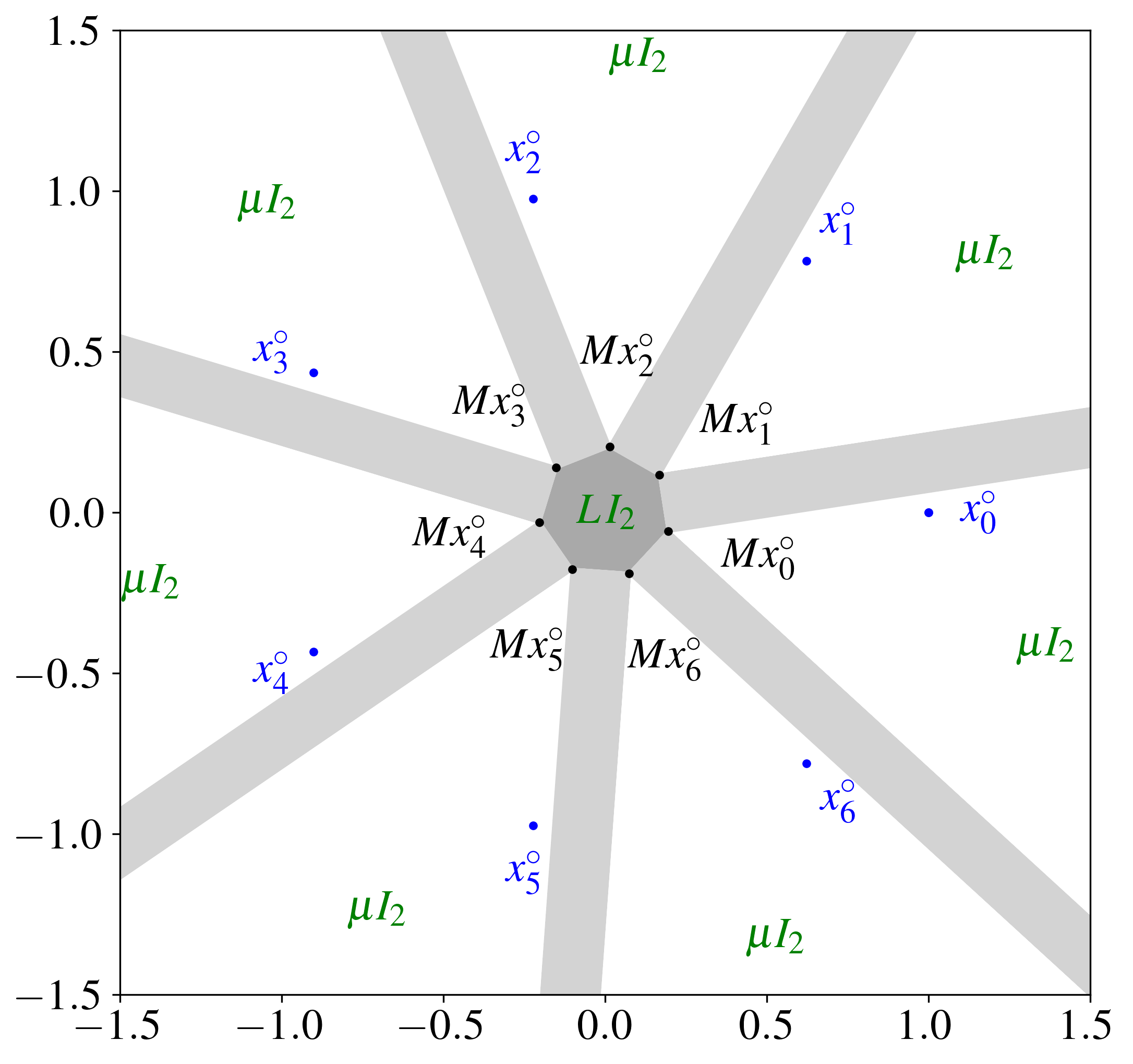}
        \caption{Shape of the counter-example function~$\psi^K_{\gb,\ml}$ given by~\eqref{eq:def_function_qui_tue}, for~$(\mu,L )= (0.005, 1)$ and~$(\gb)=(3.3, 0.75) \in {\Omega}_{7\text{-}\circ\text{-}\cycle}(\Fml)$. The function is locally quadratic, with Hessian~${\color{green!50!black} L \Id_2}$ inside~$\conv\left\{M x_t^\circ, t\in\range{0}{6}\right\}$ (gray background),~${\color{green!50!black}\mu \Id_2}$ in the white-background around~$\TikCircle_7$ and quadratic with Hessian spectrum~$\{\mu,L\}$ in the light gray area. \label{fig:psi}}
    \end{minipage}
\end{figure}
In other words, $\Orouml$ is a subset of $\Ocyml$ for which we limit the cycles to be (i) in dimension~$d=2$ (there was no restriction on the dimension earlier on) and (ii) with a specific shape (cycling over the roots of unity). The fact that limiting ourselves to $\Orouml$ is a reasonable restriction will be discussed in \Cref{sec:cycles}.
For clarity, notations of the various regions are summarized in \Cref{tab:not_regions}.  Putting things together and leveraging \Cref{fact:Cv_in_cycle_compl}, we have the following fact.
\begin{Fact}For any $0,\mu\le L$: 
\begin{equation*}
     \Oghml \subseteq 
     \Ocvml \subseteq 
     (\Ocyml)^c \subseteq 
     (\Orouml)^c 
     \subseteq \Oqml.
\end{equation*}
\end{Fact}
Next,  we observe that  $\eqref{eq:hb}(f)$ cycles over $\TikCircle_K$, if and only if we have a simple expression for $ (\nabla f(x_t^\circ))_{t\in \llbracket 0, K-1 \rrbracket}$. Indeed, given the iterates in \eqref{eq:hb}$(f)$, the value of the gradients to obtain those iterates are uniquely obtained. 
\begin{Lemma}
    \label{lem:gt_def}
    Let $K \geq 2$ an integer and $\theta_K\eqdef \tfrac{2\pi}{K}$. Let $(x_t^\circ)_{t \in \llbracket 0, K-1 \rrbracket} = \TikCircle_K$ be the roots-of-unity cycle of length~$K$.
    Let $\gb\in \Oqml$. For any differentiable function $f$, $\HB_{\gb}(f)$ cycles on $\TikCircle_K$ if and only if
    \begin{equation} \label{eq:cns_gradients_cycle}
        \forall t \in \llbracket 0, K-1 \rrbracket, ~ \nabla f(x_t^\circ) = g_t \eqdef \frac{(1+\beta) I_2 - R - \beta R^{-1}}{\gamma} x_t^\circ.
    \end{equation}
\end{Lemma}

\begin{proof}{}
    Let $f$ any differentiable function. By \Cref{def:cycle} and \Cref{rem:equiv_pov_cycle}, $\eqref{eq:hb}_{\gb}(f)$ cycles on $(x_t^\circ)_{t \in \llbracket 0, K-1 \rrbracket}$ if and only  if for any 
    $t \in \llbracket 0, K-1 \rrbracket $,  $ x_{t+1}^\circ = x_t^\circ - \gamma \nabla f(x_t^\circ) + \beta (x_t^\circ - x_{t-1}^\circ)$, with $x_{t}^\circ$ extended $K$-periodically (i.e.,~$x^\circ_{-1}\eqdef x^\circ_{K-1}$ and  $x^\circ_{K}\eqdef x^\circ_{0}$).
    Since $\gamma \neq 0$, this  system is equivalently written as, for any  
    $ t \in \llbracket 0, K-1 \rrbracket,$ 
    \begin{equation}\label{eq:hb_inv}
    ~ \nabla f(x_t^\circ) = g_t \eqdef \frac{(1+\beta)x_t^\circ - x_{t+1}^\circ - \beta x_{t-1}^\circ}{\gamma}.   
    \end{equation}
    Replacing the expressions $x^\circ_{t+1} = R x^\circ_t$ and $x^\circ_{t-1} = R^{-1} x^\circ_t$, we obtain the desired result.
\end{proof}

\vspace{1em}
\noindent
The values of the  gradients at points $(x_t^\circ)_{t\in \llbracket 0, K-1 \rrbracket}$ are depicted as red arrows on \Cref{fig:cycle}.
We now use \Cref{lem:gt_def} to obtain an analytical form of $\Orouml$.

\begin{restatable}{Th}{roucyclesregion}\textbf{\emph{(Analytical form of Roots-of-unity cycle region)}}
\label{thm:analytical_ROU_region}
    For any $K\geq 2$, the $K^{\text{th}}$-roots-of-unity cycling region is, for $\theta_K = \frac{2\pi}{K}$:
    \begin{align*}
        \OKrouml = & \left\{(\gb)\in\Oqml ~|~ \right.\\
        &\left. (\mu\gamma)^2 - 2\left[\beta - \cos\theta_K + \kappa(1 - \beta\cos\theta_K)\right](\mu\gamma) + 2\kappa(1 - \cos\theta_K)(1 + \beta^2 - 2\beta\cos\theta_K) \leq 0.
        \right\}
    \end{align*}
    Moreover for any $K\geq 2$, and any $(\gb) \in \OKrouml$,
    \begin{equation}\label{eq:def_function_qui_tue}
        \psi^K_{\gb,\ml}: x \mapsto \frac{L}{2}\|x\|^2 - \frac{L-\mu}{2}d(x, \conv\left\{M  x_t^\circ, t\in\range{0}{K-1}\right\})^2
    \end{equation}
    is a function such that \eqref{eq:hb}$\,_{\gb}(\psi^K_{\gb,\ml})$ cycles on $\TikCircle_K$, with $M $ the linear operator $M \eqdef \frac{(1+\beta - \mu\gamma) I_2 - R - \beta R^{-1}}{(L - \mu)\gamma}$.
\end{restatable}

For given $K, \ml$, \Cref{thm:analytical_ROU_region} provides a second-order equation on $(\gb)$, such that~$\eqref{eq:hb}_{\gb}$ cycles over~$\TikCircle_K$ on $\Fml$. We use this formula to plot the regions~$\OKrouml$ in \Cref{fig:hb_cycles_vs_lyapunov}, for increasing cycle length $K$, for two values of $\kappa$. \Cref{eq:def_function_qui_tue} gives an explicit formula for the function that realizes the cycle: this function is a quadratic by part: its shape is described in \Cref{fig:psi}.

\begin{sketch}
    By \Cref{lem:gt_def}, $(\gb) \in \OKrouml$ if and only if, there exists a function $f\in \Fml$ such that \eqref{eq:cns_gradients_cycle} holds. Establishing the \textit{existence} of a function in the class $\Fml$ having specific gradient values at a finite  number  of specific points can be cast as verifying a finite number of simple inequalities. Those conditions, often referred to as~\textit{interpolation conditions} (see, e.g.,~\citep{taylor2017smooth}) come along with a systematic construction of the given function as a Moreau envelope (similar in spirit with the proof of~\citep[Theorem 4]{taylor2017smooth}). The complete proof is given in \Cref{app:proof_of_analytical_ROU_region}.
\end{sketch}

\noindent
\Cref{thm:analytical_ROU_region} shows that, $(\gb)$ is in $\OKrouml$ if and only if
\begin{align}
    \gamma \in & \left[ \gamma_{-}(\beta,K,\ml), \gamma_{+}(\beta,K,\ml)\right],
    \label{eq:gamma_bound}
\end{align}
with $(\gamma_{-}(\beta,K,\ml), \gamma_{+}(\beta,K,\ml))$ obtained as the roots of the second order polynomial given in  \Cref{thm:analytical_ROU_region}, i.e.,~$\gamma \mapsto (\mu\gamma)^2 - 2\left[\beta - \cos\theta_K + \kappa(1 - \beta\cos\theta_K)\right](\mu\gamma) + 2\kappa(1 - \cos\theta_K)(1 + \beta^2 - 2\beta\cos\theta_K)$ -- and the set is empty if this polynomial is always positive.
Hence, for any $\beta$, the set of all $\gamma$ such that $(\gb)\in\Orouml$ is the union (over $K\geq 2$) of intervals given by \eqref{eq:gamma_bound}, that are not necessarily connected. 
In the proof of the next result, we will rely on the fact that for  $\kappa\le \left(\frac{3-\sqrt{5}}{4}\right)^2$, this union actually is a \textit{unique} interval of the form~$[\gamma_{\min}(\beta,\ml), \frac{2(1+\beta)}{L}]$, as 
illustrated in \Cref{fig:hb_cycles_vs_lyapunov1}. 
For larger values of  $\kappa$, the  union is not a single interval, which can be expected as the region shrinks as $\kappa$ approaches~1. Such a behavior  is 
illustrated in \Cref{fig:hb_cycles_vs_lyapunov2}. However, large values of $\kappa$ are not problematic as for those, the difference between $\sqrt{\kappa}$ and $\kappa$ is not significant. 
This is made formal in \Cref{thm:analytical_ROU_region_bis}, stated in \Cref{apx:convenient_expression_of_omega_cycle} is essential for the next section.

\begin{figure}[t]
  \begin{subfigure}[t]{.49\linewidth}
    \centering
    \includegraphics[width=\linewidth]{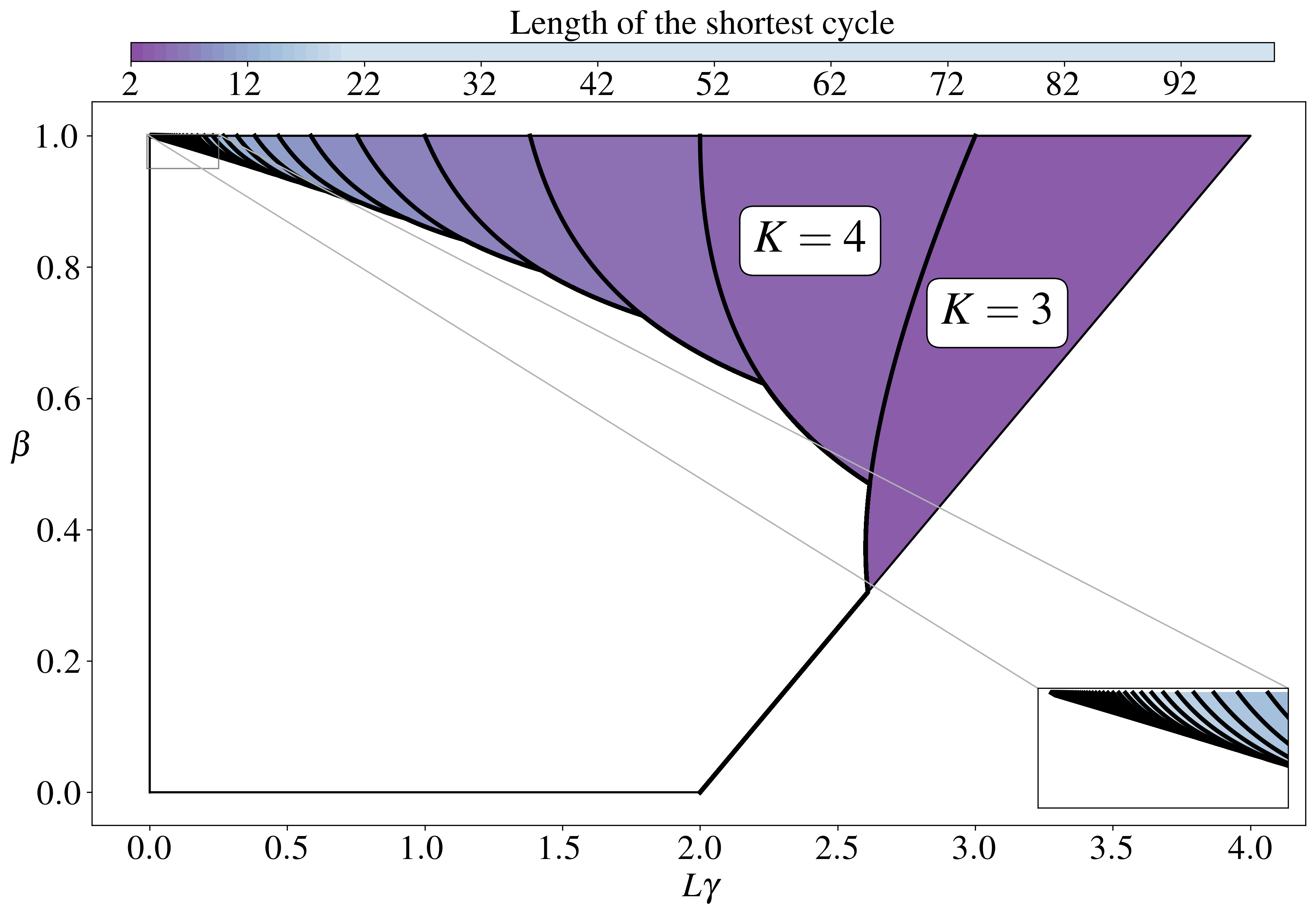}
    \caption{$\kappa=0.01$. Typical region  shape for small $\kappa$}
    \label{fig:hb_cycles_vs_lyapunov1}
  \end{subfigure}
  \hfill
  \begin{subfigure}[t]{.49\linewidth}
    \centering
     \includegraphics[width=\linewidth]{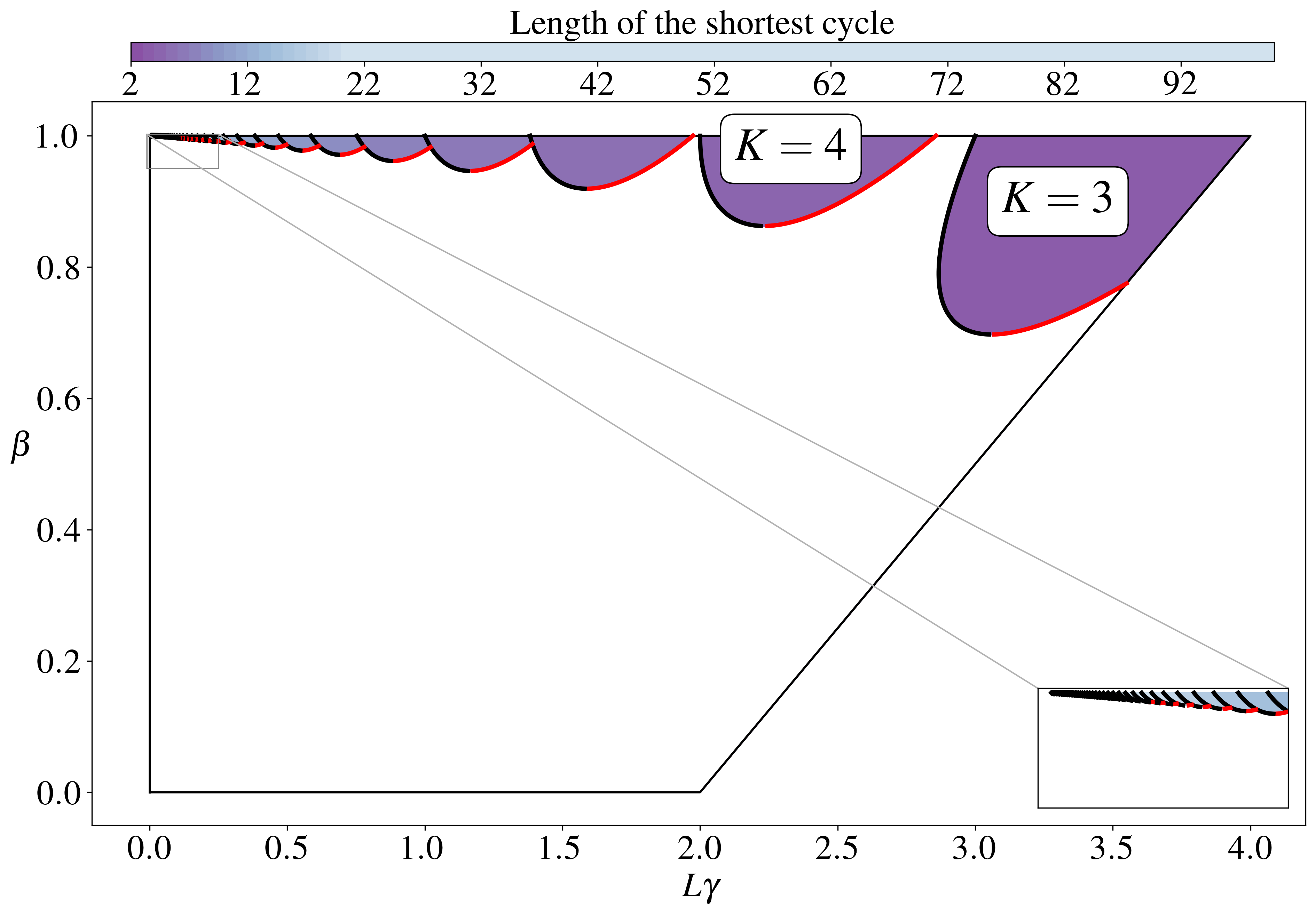}
    \subcaption{$\kappa=0.7$. Typical region  shape  shape for  $\kappa\simeq 1$.}
    \label{fig:hb_cycles_vs_lyapunov2}
  \end{subfigure}
  \caption{Regions $\OKrouml$ for increasing values of $K$ between 2 and 100. The limit of the regions is obtained from the  analytical formula given by \Cref{thm:analytical_ROU_region}. }
  \label{fig:hb_cycles_vs_lyapunov}
\end{figure}

\subsection{Non-acceleration on \texorpdfstring{$\Fml$}{Fml}}\label{subsec:noaccel_main}

In this section, we finally obtain the main non-acceleration result of the paper, by comparing $\OKrouml$ to the sublevel sets of the convergence rate of \eqref{eq:hb} on $\Qml$.

\begin{restatable}{Th}{incompatibility}
    \label{thm:incompatibility}
    There exists an absolute constant $C>0$ (any $C > \tfrac{50}{3}$), such that for any $0 < \mu < L$, we have:
    \begin{equation}
    (\Orouml)^c \cap \SLS_{\ml}\left(\frac{1 - C\kappa}{1 + C\kappa}\right) = \emptyset.
    \end{equation}
\end{restatable}

\begin{sketch}
    The complete proof is provided in~\Cref{apx:main-result}. In short, we show that, if $(\gb) \in (\Orouml)^c$, then using \Cref{thm:analytical_ROU_region_bis} for any $\kappa\leq\left(\frac{3-\sqrt{5}}{4}\right)^2$, necessarily, $\mu < C\kappa(1-\beta)$ for any constant $C> \tfrac{50}{3}$ which is excluded from $\SLS_{\ml}\left(\frac{1 - C\kappa}{1 + C\kappa}\right)$ by \Cref{lem:level_sets}.
    For $\kappa\geq \left(\frac{3-\sqrt{5}}{4}\right)^2$, we have $\sqrt{\kappa}\leq (3 + \sqrt{5})\kappa\leq C\kappa$ for any $C>\frac{50}{3}$, hence the result.
\end{sketch}

\vspace{1em} \noindent
\Cref{thm:incompatibility} is illustrated on \Cref{fig:triangles_and_WCAR}: we represent, for three values of $\kappa$ in decreasing order, the set $ (\Orouml)^c$ and the set $\SLS_{\ml}\left(\frac{1 - C\kappa}{1 + C\kappa}\right)$. 
This means  that for any $(\gb)$ such that \eqref{eq:hb} does not cycle over a roots-of-unity cycle on $\Fml$, the asymptotic convergence rate of \eqref{eq:HB-quad} over $\Qml$, is worse (i.e.,~larger) than $\frac{1 - C\kappa}{1 + C\kappa}$.
Formally, we have the following corollary.

\begin{Cor}\label{cor:noaccel}
    \eqref{eq:hb} does not accelerate over on the class $(\Fml)_{0 < \mu < L}$.

\vspace{0.4em}
    \noindent
    More precisely, there exists a constant $C$ such that for all $0<\mu<L$, for all $(\gb)\in \R\times \R $ the  worst-case asymptotic convergence rate \eqref{eq:hb} over $\Fml$ is lower bounded by $\tfrac{1 - C\kappa}{1 + C\kappa}$:
\[\forall 0<\mu<L, \rho^\star(\Fml) \eqdef \min_{(\gb)\in \R\times \R } \rho_{\gb}(\Fml) \geq \tfrac{1 - C\kappa}{1 + C\kappa},\] 
\end{Cor}

\begin{proof}[\Cref{cor:noaccel}]
    Let $(\gb) \in \Oqml$.
    Possibilities are twofold:
    \begin{itemize}[leftmargin=*]
        \item if $(\gb) \in \SLS_{\ml}\left(\frac{1 - C\kappa}{1 + C\kappa}\right)$, we know from \Cref{thm:incompatibility}, that $(\gb) \in \SLS_{\ml}\left(\frac{1 - C\kappa}{1 + C\kappa}\right) \subseteq \Orouml \subseteq \Ocyml \subseteq \Ocvml^c$, i.e.~there exists a function $f\in\Fml$ such that~$\eqref{eq:hb}_{\gb}(f)$ does not converge.
        \item if $(\gb) \in \SLS_{\ml}\left(\frac{1 - C\kappa}{1 + C\kappa}\right)^c$, then by definition, there exists a function $f\in\Qml\subseteq\Fml$ such that~$\eqref{eq:hb}_{\gb}(f)$ achieves an asymptotic rate larger that $\frac{1 - C\kappa}{1 + C\kappa}$.
    \end{itemize}
\end{proof}

\begin{figure}[t]
  \begin{subfigure}[t]{.3\linewidth}
    \centering
    \includegraphics[width=\linewidth]{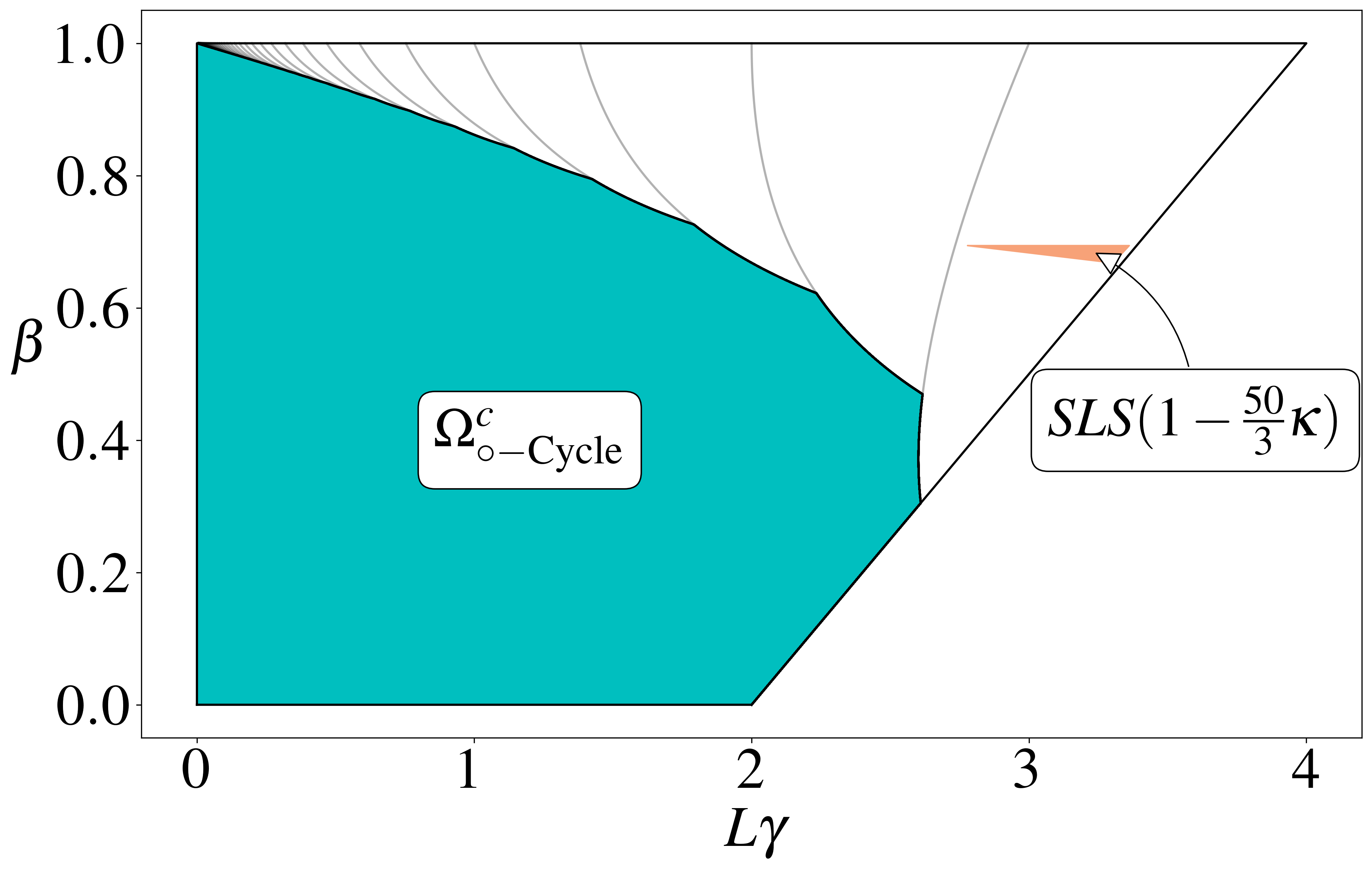}
    \caption{$\kappa=0.01$}
    \label{fig:rou_cycles_negative_with_sublevel_sets_a}
  \end{subfigure}
  \hfill
   \begin{subfigure}[t]{.3\linewidth}
    \centering
    \includegraphics[width=\linewidth]{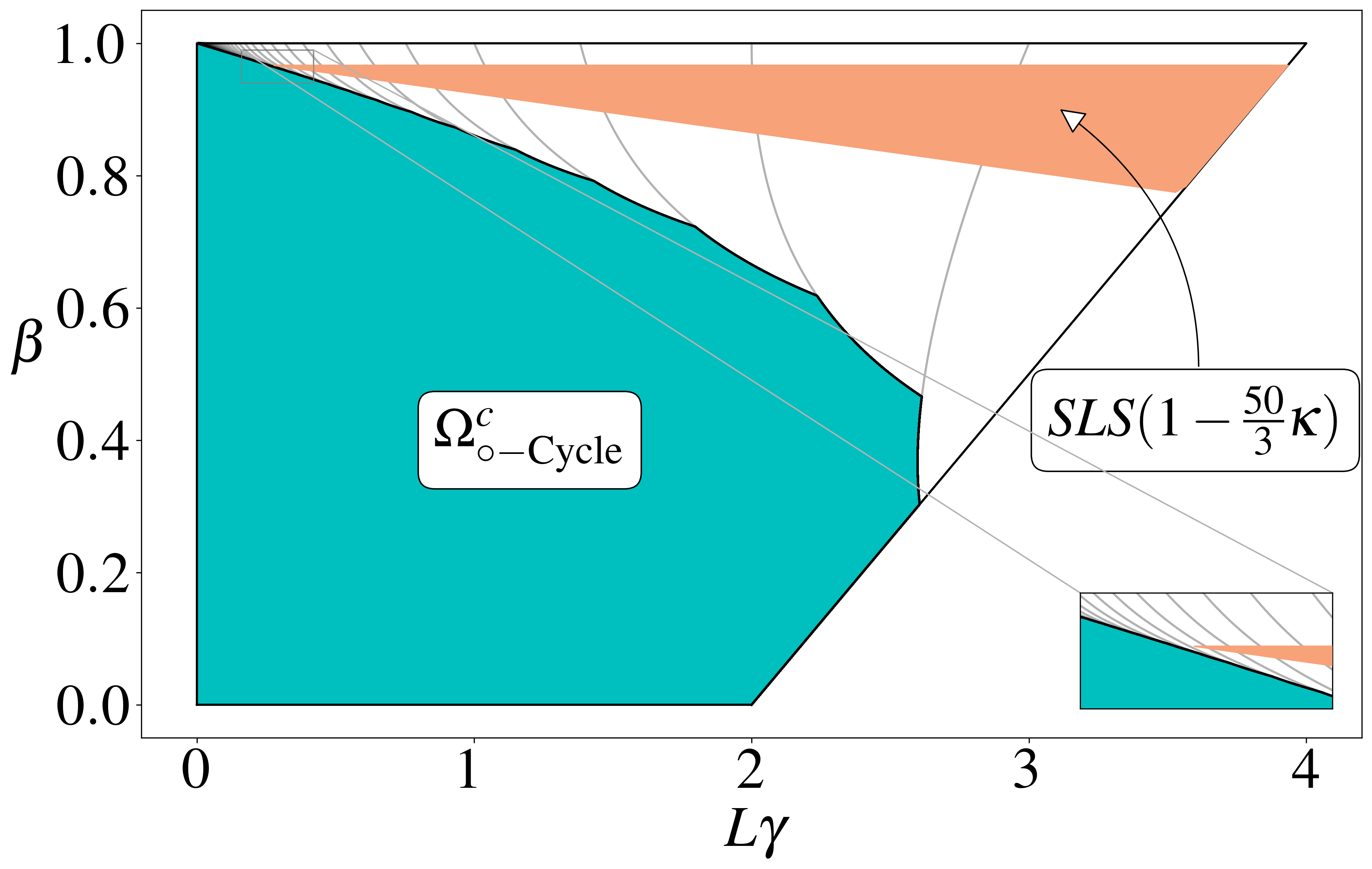}
    \caption{$\kappa=0.001$}
    \label{fig:rou_cycles_negative_with_sublevel_sets_b}
  \end{subfigure}
    \hfill
   \begin{subfigure}[t]{.3\linewidth}
    \centering
    \includegraphics[width=\linewidth]{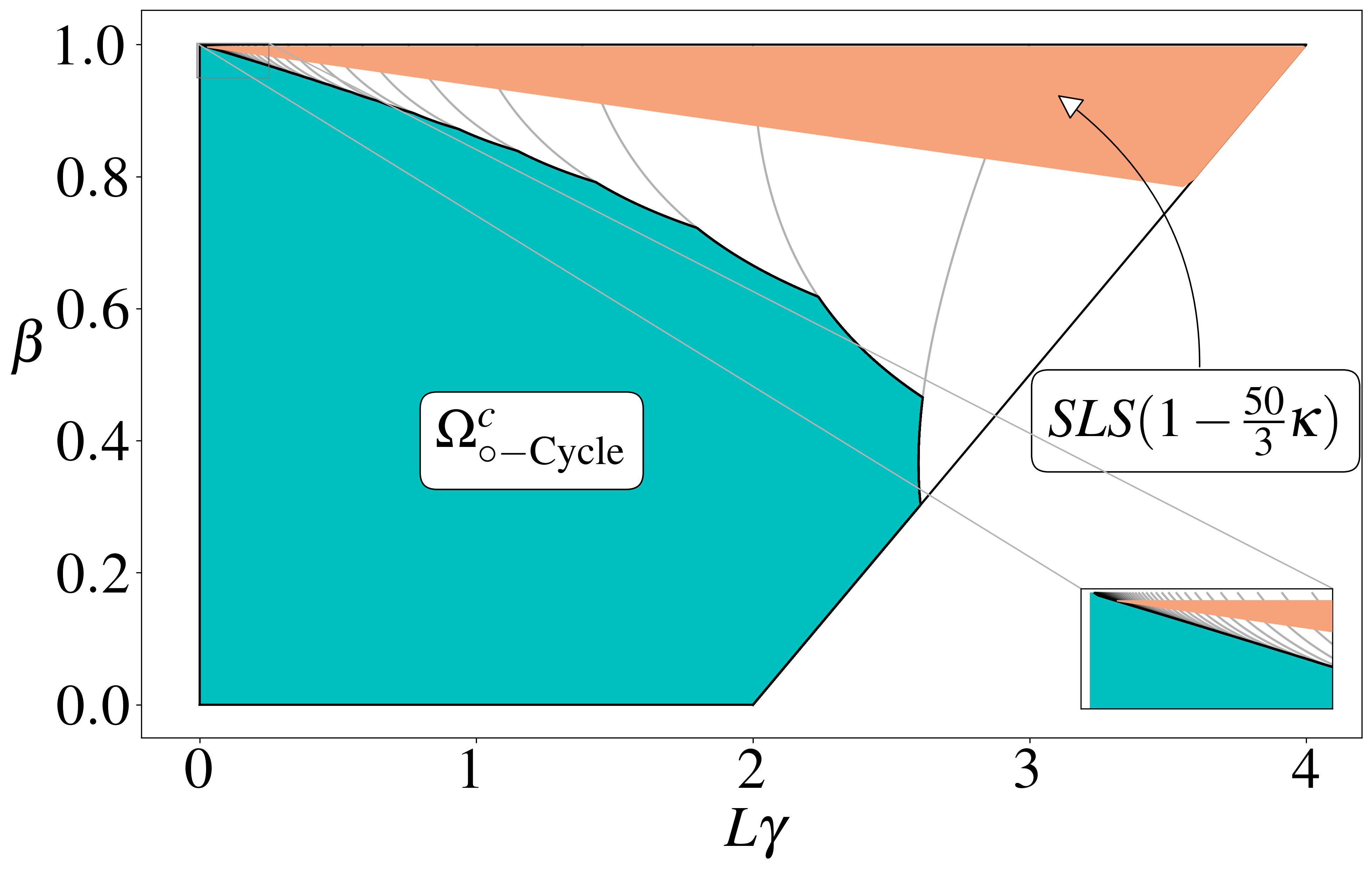}
    \caption{$\kappa=0.0001$}
    \label{fig:rou_cycles_negative_with_sublevel_sets_c}
  \end{subfigure}
  \caption{Illustration of the incompatibility  result \Cref{thm:incompatibility}. For three values of $\kappa$, we represent the sublevel set $\SLS_{\mu,L}\left(\frac{1 - C\kappa}{1 + C\kappa}\right)$ and the region $\left(\Orouml\right)^c$ and notice that their intersections are empty.}
  \label{fig:triangles_and_WCAR}
\end{figure}

\vspace{1em} \noindent 
This result concludes the first part of our study of \eqref{eq:hb}, as \Cref{cor:noaccel} closes a long-standing open question on the behavior of \eqref{eq:hb}.

In the next section, we demonstrate the cycles obtained in this section are robust to a perturbation and to small variations of $\gb$, thereby naturally strengthening our results.

\section{Robustness of the roots-of-unity cycle}
\label{sec:robustness}

\begin{figure}[t]
   \begin{subfigure}[t]{.49\linewidth}
    \centering
    \includegraphics[width=\linewidth]{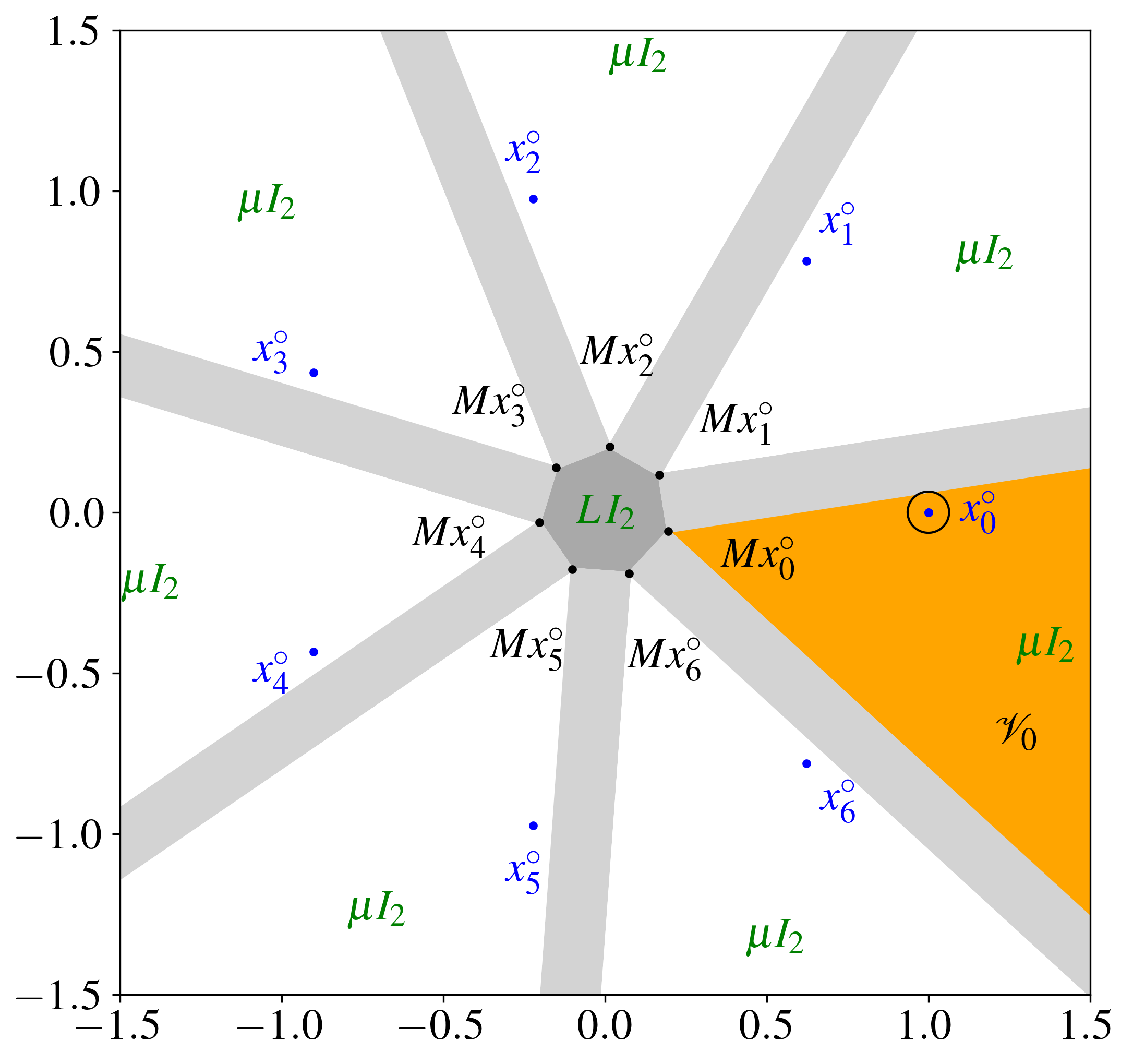}
    \caption{$(\gb) =(3.3, 0.75)$, $(\ml)=(.005, 1)$, $K=7$ }
    \label{fig:psi_good_gammabeta}
  \end{subfigure}
\hfill
 \begin{subfigure}[t]{.49\linewidth}
    \centering
    \includegraphics[width=\linewidth]{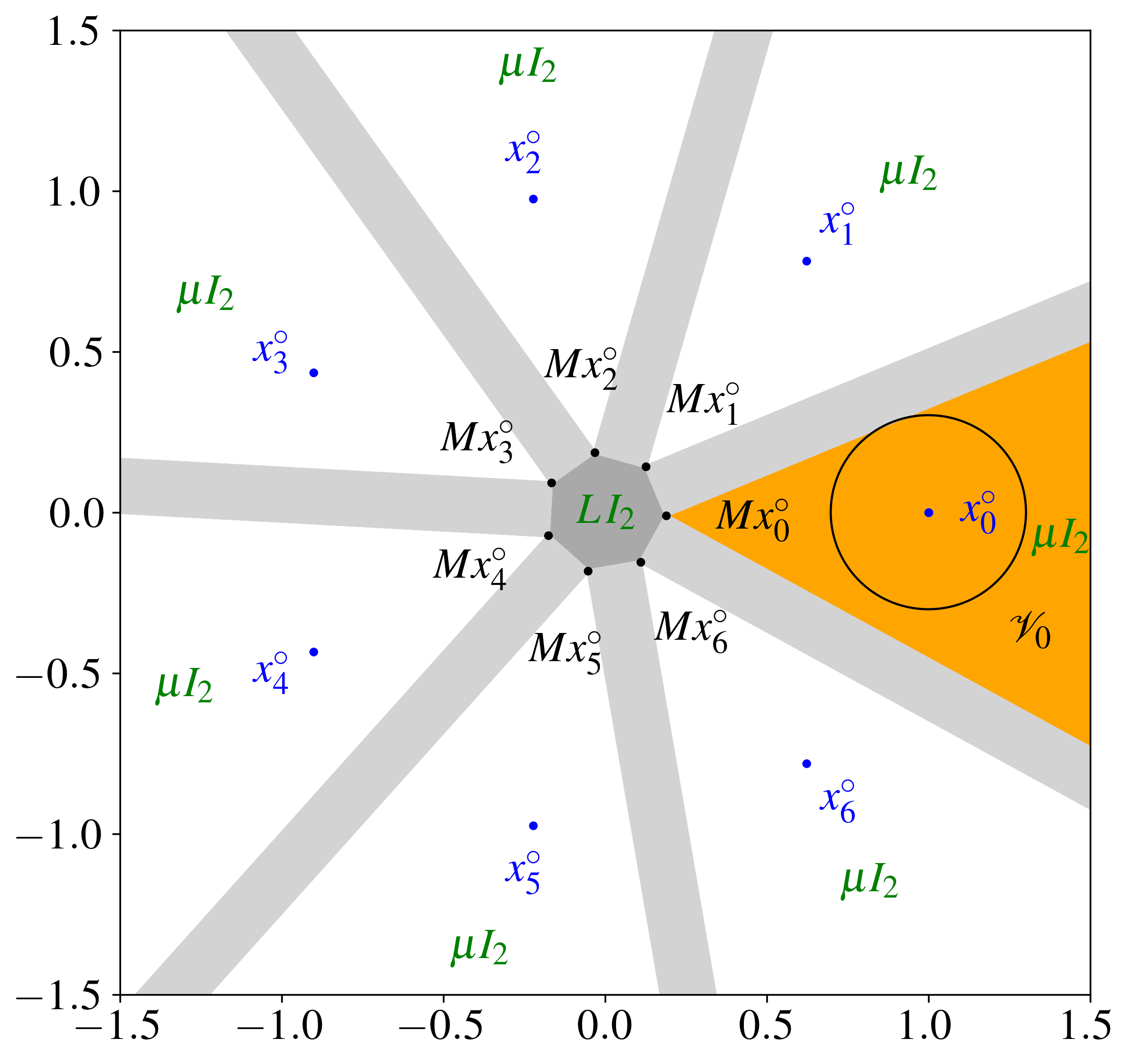}
    \caption{$(\gb) =(3.8, 0.95)$, $(\ml)=(.005, 1)$, $K=7$}
    \label{fig:psi_bad_gammabeta}
  \end{subfigure}
    \caption{Shape of the counter-example function  $\psi^K_{\gb,\ml}$ (See \Cref{fig:psi}), locally quadratic neighborhood $(\Vk)_{k\in \rK}$ (white background, highlighted in orange for $\mathcal V_0$), and ball  $B(x_0^\circ,r_{\max})$.}
    \label{fig:psi_and_neighb}
\end{figure}

In \Cref{thm:analytical_ROU_region} we proved that, for any~$K \geq 2$, for any $(\gb) \in \OKrouml$,~$\eqref{eq:hb}_{\gb}(\psigbml)$ cycles over $\TikCircle_K$, where~$\psigbml$ is defined in~\eqref{eq:def_function_qui_tue} as 
\[\psigbml(x) = \frac{L}{2}\|x\|^2 - \frac{L-\mu}{2}d(x, \conv\left\{Mx^\circ_t, t\in\range{0}{K-1}\right\})^2.\]
A natural concern is the robustness of this result to an initial perturbation of the starting points, or a random or adversarial perturbations of the gradients or the hyperparameters. In this section, building on the properties of~$\psigbml$, we establish that the cycle is indeed robust to an initial perturbation of~$(x_0, x_1)$ in a neighborhood, to small (random or adversarial) variations of the parameters~$\gamma$ and $\beta$ at each iteration, and to a small (random or adversarial) noise on the gradient. We explicitly quantify neighborhoods providing such stability.

To establish stability properties, we leverage the fact that the function $\psigbml$ is locally quadratic around the iterates of~$\TikCircle_K$. In the rest of the section and the proofs, we extend again~$(x^\circ_k)$ to~$k\in \mathbb Z$ by $K$-periodicity, as~$(x^\circ_k)= (x^\circ_{k\mod{K}})$. For~$k\in \range{0}{K-1}$, we introduce the neighborhood~$\mathcal V_k$ of $x_k^\circ$ as follows.
\begin{Def}[Locally quadratic neighborhood $(\mathcal  V_k)_{k\in \rK}$]
    For any $k\in \rK$, we denote $\Vk$ the largest neighborhood of $x_k^\circ$ over which $\psigbml$ is quadratic with Hessian $\mu \Id_2$.
\end{Def}
This neighborhood is represented on \Cref{fig:psi_and_neighb} as the white area surrounding~$x^\circ_k$, highlighted in orange for $\mathcal V_0$. Formally, $\mathcal V_k$ is composed of all points that have the same projection on~$\conv\left\{M x^\circ_t, t\in\range{0}{K-1}\right\}$ as~$x^\circ_k$. Moreover, if  we define~$r_{\max} = -\left<(I-M)x^\circ_0, \frac{M(x^\circ_1 -x^\circ_0)}{\|M(x^\circ_1 - x^\circ_0)\|}\right>$, as the distance between~$x_0^\circ$ and the light gray area, as represented in~\Cref{fig:psi_and_neighb}, we have that for any~$k$,~$B(x_k^\circ, r_{\max}) \subseteq \Vk$. Moreover, as the function is locally quadratic, for any~$z\in \Vk$, 
\begin{alignat}{3}
&\psigbml(z) &&= \frac{L}{2}\|z\|^2 - \frac{L-\mu}{2}\|z - Mx_k^\circ \|^2 &&=  \frac{\mu}{2}\|z\|^2 + (L-\mu)\langle
Mx_k^\circ,  z  \rangle - \frac{L-\mu}{2}\|Mx_k^\circ \|^2 \nonumber\\    
 &   \nabla \psigbml(z) &&= L z - (L-\mu) (z- Mx_k^\circ) &&= \nabla \psigbml(x_k^\circ) + \mu (z-x_k^\circ).
\label{eq:grad}
\end{alignat}
We first consider the case of a perturbation of the initial point, for which we prove the following result:
\begin{Th}\label{thm:stab_init_pert}
Consider $0<\mu<L$,  $K\geq 2$, the roots-of-unity cycle $ (x_k^\circ)_{k\in\range{0}{K-1}} $ and $(\gb)\in \OKrouml$. 
Let   $(z_t)_{t\in \mathbb{N}}$  be the sequence generated by running $\eqref{eq:hb}_{\gb}$ initialized at 
$(z_0,z_1)= (x^\circ_0 + \delta_{0}, x^\circ_1 +  \delta_{1})$.
There exists $\kappa_P$ such that if $\sqrt{\|\delta_{0}\|^2 + \|\delta_{1}\|^2} \leq \kappa_P r_{\max}$, the following properties hold:
\begin{enumerate}
    \item \label{item:1page18} for all $t\in \mathbb N$, $z_t \in \mathcal V_{t\mod{K}}$,
    \item \label{item:2page18} the sequence $\delta_t \eqdef z_t - x^\circ_{t}$ follows the dynamic of $\eqref{eq:hb}_{\gb}(x\mapsto \frac{\mu}{2}\|x\|^2)$, initialized at $(\delta_{0}, \delta_{1})$,
    \item \label{item:3page18} consequently, $\|z_t - x^\circ_{t} \| $ converges to 0 as $t\to \infty$, at rate $\rho_{\gb}(\mathcal Q_{\mu,L=\mu})$.
\end{enumerate}
\end{Th}

The first two points are proved simultaneously, and the third point is a consequence of the second. In words, we show that at all iterations, the iterate $z_t$ remains in the locally quadratic neighborhood of $x_{t}^\circ$, and that remarkably, the dynamic of the residual $(\delta_t)$ is then precisely the one of a $\eqref{eq:hb}$ dynamic \textit{on an isotropic (i.e.,~with~$\kappa=1$) quadratic function}. Indeed, if $z_{t}  \in \mathcal V_{t\mod{K}}$, we have that:
\begin{align*}
    z_{t+1} &\overset{\eqref{eq:hb}}{=} z_t - \gamma \nabla \psigbml (z_t) +  \beta (z_t-z_{t-1}).
\end{align*}
And as $( z_{t})_t=(x_t^\circ+  \delta_t)_t $, using the formula for the gradient \eqref{eq:grad}, we get 
\begin{align*}
    x_{t+1}^\circ+  \delta_{t+1}   &\overset{\eqref{eq:grad}}{=} x_t^\circ + \delta_t - \gamma \nabla \psigbml (x_t^\circ ) - \gamma \mu \delta_t +  \beta (x^\circ_t-x^\circ_{t-1})+\beta (\delta_t-\delta_{t-1}). 
\end{align*}
Moreover, as $\eqref{eq:hb}_{\gb}(\psigbml)$ cycles over $\TikCircle_K$, $x_{t+1}^\circ {=} x_t^\circ - \gamma \nabla \psigbml (x_t^\circ )  +\beta (x^\circ_t-x^\circ_{t-1})$   thus
\begin{align}\label{eq:dyn_res_HB}
 \delta_{t+1}   &{=}   \delta_t - \gamma \mu \delta_t  +\beta (\delta_t-\delta_{t-1}).
\end{align}
Which means that   as long as $z_{t}  \in \mathcal V_{t\mod{K}}$,  $ \delta_{t+1}$ is obtained by the dynamic of \eqref{eq:hb} on the quadratic isotropic function $x\mapsto \frac{\mu}{2}\|x\|^2$. We now give the complete proof.

\begin{proof}
    We introduce matrices $P$ and $D$ verifying $PDP^{-1} = \begin{pmatrix}
        (1+\beta)\Id_2 - \mu\gamma\Id_2 & -\beta\Id_2 \\ \Id_2 & 0
    \end{pmatrix}$, and such that  the operator norm   $\rho_D$ of the matrix $D$, $\rho_D = \|D\|_{\text{op}}$ is smaller than 1, and set $\kappa_P = \frac{1}{\|P\|_{\text{op}}\|P^{-1}\|_{\text{op}}} \leq 1$.  The existence of such matrices is guaranteed as $(\gb)\in \Oqml$.  
To ensure that $z_t\in \mathcal V_{t\mod{K}}$ (point \ref{item:1page18} in \Cref{thm:stab_init_pert}), we prove that  $\forall t\geq 1, \left\|P^{-1}\begin{pmatrix}
        \delta_{t} \\ \delta_{t-1}
    \end{pmatrix}\right\| \leq \frac{r_{\max}}{\|P\|_{\text{op}}}$.
Indeed, this is a stronger statement, as  it implies that $\forall t\geq 1, \left\| \begin{pmatrix}
        \delta_{t} \\ \delta_{t-1}
    \end{pmatrix}\right\| \leq {r_{\max}}$, thus $z_t\in B(x^\circ_{t\mod{K}}, r_{\max}) \subseteq \mathcal V_{t\mod{K}}$.

    \noindent
    We prove simultaneously by induction  point \ref{item:2page18} of \Cref{thm:stab_init_pert} and the condition $\forall t\geq 1, \left\|P^{-1}\begin{pmatrix}
        \delta_{t} \\ \delta_{t-1}
    \end{pmatrix}\right\| \leq \frac{r_{\max}}{\|P\|_{\text{op}}}$. \\

    \noindent
 \textbf{Initialization}: $\left\|\begin{pmatrix}
        \delta_{1} \\ \delta_{0}
    \end{pmatrix}\right\| = \sqrt{\|\delta_{0}\|^2 + \|\delta_{1}\|^2} \leq \kappa_P r_{\max}= \frac{r_{\max}}{\|P\|_{\text{op}}\|P^{-1}\|_{\text{op}}}$   implies $\left\|P^{-1}\begin{pmatrix}
        \delta_{1} \\ \delta_{0}
    \end{pmatrix}\right\| \leq \frac{r_{\max}}{\|P\|_{\text{op}}}$. \\

 \noindent
 \textbf{Induction}:  By induction hypothesis $z_t\in \mathcal V_{t\mod{K}}$, thus by \eqref{eq:dyn_res_HB}, $\delta_{t+1}$ is obtained by $\eqref{eq:hb}$:
 \begin{align*}
        \begin{pmatrix}
            \delta_{t+1} \\ \delta_{t}
        \end{pmatrix}
        & =
        \begin{pmatrix}
        (1+\beta)\Id_2 - \mu\gamma\Id_2 & -\beta\Id_2 \\ \Id_2 & 0
    \end{pmatrix}
        \begin{pmatrix}
            \delta_{t} \\ \delta_{t-1}
        \end{pmatrix}  =
        PDP^{-1}
        \begin{pmatrix}
            \delta_{t} \\ \delta_{t-1}
        \end{pmatrix}\\
 \Rightarrow \ \ \   \left  \|   P^{-1}
        \begin{pmatrix}
            \delta_{t+1} \\ \delta_{t}
        \end{pmatrix} \right\|
        & = 
     \left   \| DP^{-1}
        \begin{pmatrix}
            \delta_{t} \\ \delta_{t-1}
        \end{pmatrix} \right\|
         \overset{\rho_D = \|D\|_{\text{op}}}{\leq}
        \rho_D \left\|
        P^{-1}
        \begin{pmatrix}
            \delta_{t} \\ \delta_{t-1}
        \end{pmatrix}
        \right\| \overset{\text{by induction}}{\leq} \rho_D\frac{r_{\max}}{\|P\|_{\text{op}}} \le \frac{r_{\max}}{\|P\|_{\text{op}}}.
    \end{align*}

\noindent
This concludes the induction and proves  the first two items of \Cref{thm:stab_init_pert}. Finally, Point~\ref{item:3page18} of \Cref{thm:stab_init_pert} is thus a consequence of point~\ref{item:2page18} and \Cref{prop:conv_quad}.
    
\end{proof}  

\vspace{1em} \noindent
We now give a more general theorem, that provides a stability result w.r.t.~the initial points and (potentially adversarial) small perturbations of $\gb$ at each iteration, and of the gradient oracles. Furthermore, we give explicit neighborhoods preserving the non-convergence property of \eqref{eq:hb} on $\psigbml$. As in \Cref{thm:stab_init_pert}, we consider that the process starts at a perturbed point $(x^\circ_0 + \delta_{0}, x^\circ_1 + \delta_{1})$. Moreover, instead of applying parameters $(\gb)$ at each step, we consider that step $t$ is performed with parameters $(\gamma_t, \beta_t)=(\gamma + \delta_{\gamma_t}, \beta + \delta_{\beta_t})$, and rely on  perturbed (or noised) gradients $\hat g_t (z) = \nabla \psigbml(z)+\delta_{g_t}$.

\begin{restatable}{Th}{robutness}
    \label{thm:robustness}
    Consider $0<\mu<L$,  $K\geq 2$, the roots-of-unity cycle $ (x_k^\circ)_{k\in\range{0}{K-1}} $, and $(\gb)\in \OKrouml$. 
    Let $(z_t)_{t\in \mathbb{N}}$  be the sequence generated by running $\eqref{eq:hb}_{\gamma_t,\beta_t}$, with time varying parameters $\gamma_t,\beta_t$, initialized at 
    $(z_0,z_1)= (x^\circ_0 + \delta_{0}, x^\circ_1 + \delta_{1})$, with perturbed  gradients $\hat g_t (z) = \nabla \psigbml(z)+\delta_{g_t}$.
    There exist $\kappa_P$ and $\rho_D<1$ (made explicit in the proof) such that if the following three conditions hold, 
    \begin{enumerate}
        \item $\sqrt{\|\delta_{0}\|^2 + \|\delta_{1}\|^2} \leq \kappa_P r_{\max}$, 
        \item for all $t\in \mathbb N$, 
            $\left(\frac{4}{\gamma} + \mu\kappa_P r_{\max}\right)|\delta_{\gamma_t}|
            +
            \left({2} + 2\kappa_P r_{\max}\right) |\delta_{\beta_t}|
            \leq \frac{1}{2}(1-\rho_D)\kappa_P r_{\max}$,
        \item for all $t\in \mathbb N$,   $\frac{4}{L} \|\delta_{g_t}\|\le \frac{1}{2}  (1-\rho_D)\kappa_P r_{\max}$,
    \end{enumerate}
    then, $(z_t)_{t \geq 0}$ keeps cycling in a neighborhood of $\TikCircle_K$: $\| z_t - x^\circ_{t\mod{K}} \| \leq r_{\max}$  (and thus $z_t\in \mathcal{V}_{t\mod{K}}$).
\end{restatable}

The proof extends the one of  \Cref{thm:stab_init_pert} and is  postponed to \Cref{app:robustess}. The  constants $\kappa_P, \rho_D$ are identical to the ones exhibited in the proof of \Cref{thm:stab_init_pert}.  Overall, we conclude that, as soon as $r_{\max} > 0$, i.e.~as soon as $(\gb)$ belongs to the interior $\mathrm{Int}(\Orouml)$ of $\Orouml$, the cycle $\TikCircle$ is attractive and robust to small variations of the initialization, parameters and gradient oracles. These robustness results ensure that the counter-examples we provide cannot be circumvented by adding small perturbations or stochasticity, and therefore strongly reinforce our conclusions.

In the next section, we explore a different direction, which is the question of acceleration of $\eqref{eq:hb}$ if we restrict the class $\Fml$ to functions that have a Lipschitz-continuous Hessian (i.e.,~under higher-order regularity assumptions), or even any higher order regularity.

\section{No acceleration of \texorpdfstring{\eqref{eq:hb}}{(HB)} under higher-order regularity assumptions}
\label{sec:HL}
A natural way to tentatively extend the proof of \eqref{eq:hb}'s acceleration beyond quadratics, consists in assuming \textit{Hessian Lipschitz continuity}~\citep{wang2022provable}. Indeed, for $0<\mu<L$, $K\geq2$ and $(\gb)\in \OKrouml$ the counter-example function $\psigbml$  used in \Cref{sec:noaccel} is quadratic by part, i.e.,~its hessian is constant by part. As its Hessian is not constant everywhere, $\psigbml$ is not even 3 times differentiable, so our counter-example approach seemingly leaves room for improvement under a restricted class of more regular functions. We first focus on the class of Hessian-Lipschitz functions, and will extend the argument to higher-order regularity.

\begin{Def}
    Let $\Ftml$ be the set of functions $f$ in $\Fml$ that are three times differentiable, with $\tau$-Lipschitz Hessian:
    \[\forall z,w , \|\nabla^2f (z)- \nabla^2f (w)\| \le \tau \|z-w\|. \]
\end{Def}
For any $\ml$ and $\tau$, we have $\Qml\subseteq\Ftml\subset\Fml$.
In the rest of the section we consider  fixed $0<\mu<L$, $K\geq2$ and $(\gb)\in \mathrm{Int}(\OKrouml)$. We thus omit the indices on the notation of $\psi$~(see eq.~\ref{eq:def_function_qui_tue}) and simply use:
\[ \psi := \psigbml.\]
In this section, we prove that \eqref{eq:hb} \textit{does not accelerate} on~$\Ftml$.

\renewcommand{\psigbml}{\psi}
\renewcommand{\phigbmleps}{\varphi_{\epsilon}}
\renewcommand{\phigbml}{\varphi}

\begin{Th}\label{thm:noacc_hess_lip}
For $0<\mu<L$, $K\geq2$ and $(\gb)\in \mathrm{Int}(\OKrouml)$,
\begin{enumerate}
    \item \label{item:point1HL} there exists $\tau$ and a function $\phigbml \in \Ftml$ such that $\eqref{eq:hb}_{\gb}(\phigbml)$ cycles over $\TikCircle_K$,
    \item \label{item:point2HL} moreover, for any $\tau>0$, $\eqref{eq:hb}_{\gb}$ has a cycle on $\Ftml$. 
\end{enumerate}
\end{Th}
The second point means that the set of parameters $(\gb)$ for which $\eqref{eq:hb}_{\gb}$ cycles over $\Ftml$ contains the interior of the set of parameters for which $\eqref{eq:hb}_{\gb}$ cycles on a roots-of-unity cycle over $\Fml$:
\begin{equation}\label{eq:cycleTML_vs_cylce_ML_rou}
    \forall \tau>0, \Ocytml \supseteq \mathrm{Int}(\Orouml), 
\end{equation}
and a direct consequence is that the non-acceleration result \Cref{cor:noaccel} extends to $\Ftml$: $\forall \tau>0, \forall 0<\mu<L, $
\begin{alignat*}{3}
& \rho^\star (\Ftml) &&\eqdef 
    \argmin_{(\gb)\in \R \times \R} \rho_{\gb} (\Ftml)  
     \geq \argmin_{(\gb)\in (\Ocytml)^c}  \rho_{\gb} (\Ftml) && \quad \text{ as } \Omega_{\text{cv}}(\Ftml) \subseteq (\Ocytml)^c \\
    & &&  \geq \argmin_{(\gb)\in (\Ocytml)^c}  \rho_{\gb} (\Qml)&&  \quad \text{ as } \Qml \subseteq \Ftml \\
    & &&\geq \argmin_{(\gb)\in (\Orouml)^c} \rho_{\gb} (\Qml) &&\quad  \text{ by \Cref{eq:cycleTML_vs_cylce_ML_rou} }\\ 
    & &&\geq  \frac{1-C\kappa}{1+C\kappa} && \quad \text{ by \Cref{thm:incompatibility}},
\end{alignat*}
for a constant $C$ (in fact, any $C\geq 50/3$). The rate $\rho^\star (\Ftml)$ is thus lower bounded, independently of $\tau$, by a non~accelerated rate, which proves that Hessian regularity does not help to obtain acceleration.

\paragraph{Interpretation.} This result is surprising as a discontinuity appears in $\tau=0$. Indeed, for $\tau=0$, $\Ftml=\Qml$, and acceleration is obtained. The mapping $\tau\mapsto \rho^\star (\Ftml)$ is thus not continuous in $\tau =0$. A significant nuance to help grasp this phenomenon is that we do \textit{not} prove that for all~$\tau>0$,~$\Oroutml \supset \Orouml$, that is that the existence of a roots-of-unity  cycle on~$\Fml$ implies  the existence of a roots-of-unity cycle on~$\Ftml$ for any $\tau>0$. Such a property is in fact not true. On the other hand,~\Cref{item:point1HL} in \Cref{thm:noacc_hess_lip} means that $\cup_{\tau>0 } \Oroutml \supset \Orouml$. To obtain this result, we show that \textit{there exists} a $\tau$ such that we can ``regularize''  the counter example $\psigbml$ to obtain a function in~$\Ftml$ that cycles over the roots-of-unity cycle. Then~\Cref{item:point2HL} in~\Cref{thm:noacc_hess_lip} establishes that there exist a cycle \textit{for all} $\tau$. But this cycle is not necessarily~$\TikCircle_K$: on the contrary, the cycle corresponds to a dilated version of~$\TikCircle_K$, with  a radius that diverges as~$\tau\to 0$. Such a cycle thus does not have a limit as~$\tau \to 0$, which explains the discontinuity of~$\tau\mapsto \rho^\star (\Ftml)$.

In order to prove \Cref{thm:noacc_hess_lip}, we establish the following two lemmas, that respectively result in \Cref{item:point1HL,item:point2HL} in \Cref{thm:noacc_hess_lip}, and are given in the next two sections.

\subsection{Proof of \texorpdfstring{\Cref{item:point1HL}}{Item 1} of \texorpdfstring{\Cref{thm:noacc_hess_lip}}{Theorem 5.2}}

First, in order to obtain a more regular function from $\psigbml$, we use convolutions. We consider a infinitely differentiable convolution kernel $u_\epsilon$ with bounded support $B(0,\epsilon)$, such that $u_\epsilon$ is the probability density function of a zero centered random variable. For any $\epsilon\geq 0$, let 
\begin{equation*}
\phigbmleps\eqdef u_\epsilon\ast \psigbml.    
\end{equation*}
We obtain the following lemma.
\begin{Lemma}\label{lem:psi_convolee}
    Let $0<\mu<L$, $K\geq2$ and $(\gb)\in \mathrm{Int}(\OKrouml)$. Then we have: 
    \begin{enumerate}
        \item  $\phigbmleps$ is infinitely differentiable and there exists $\tau$  such that $\phigbmleps \in \Ftml$.
        \item For any $\epsilon  \le r_{\max}$  (with $ r_{\max}$ defined in \Cref{sec:robustness}), the gradients of $\phigbmleps $ and $\psigbml$ coincide on $\TikCircle_K$,  that is for all $k\in \rK$, \[\nabla \phigbmleps (x^\circ_k)= \nabla \psigbml (x^\circ_k).\]
    \end{enumerate}
\end{Lemma}
 
\begin{proof}
The first point results from properties of convolution against probability density functions, that are recalled in  \Cref{lem:prop_varphi} in \Cref{app:proofLem1_HL}. The second point  is a consequence of the fact that $\psigbml$ is locally quadratic (thus its gradient $\nabla \psigbml$ is locally linear). Since $\psigbml$ is differentiable, we have for all $z\in \mathbb R^2$
\begin{align*}
    \nabla \phigbmleps (z) =    \nabla \left(\psigbml \ast u_\varepsilon \right)(z) 
        & = \left((\nabla \psigbml) \ast u_\varepsilon\right)(z)  = \int_{y\in B(0,\epsilon)} \nabla \psigbml (z-y) ~ u_\varepsilon(y) ~ \mathrm{dy},
\end{align*}
where  the last step uses the fact that  $u_\varepsilon$ has support $B(0,\epsilon)$. Then, for $k\in \rK$ and $z=x^\circ_k$ a point of $\TikCircle_K$, we have that for any $y\in B(0,\epsilon)$, if $\epsilon  \le r_{\max}$,   $x^\circ_k-y \in \Vk$ and by \eqref{eq:grad}, $\nabla \psigbml(x^\circ_k-y ) = \nabla \psigbml(x_k^\circ) - \mu y$. Thus 
\begin{align*}
    \nabla \phigbmleps (x_k^\circ)  = \nabla \psigbml(x_k^\circ)  - \mu \int_{y\in B(0,\epsilon)} y  u_\varepsilon(y)   \mathrm{dy} = \nabla \psigbml(x_k^\circ).
    \end{align*}
\end{proof}

\paragraph{Proof of~\Cref{thm:noacc_hess_lip} (\Cref{item:point1HL}).}\Cref{item:point1HL} of \Cref{thm:noacc_hess_lip} is a direct consequence of \Cref{lem:psi_convolee}: we use $\phigbml \eqdef \phigbmleps$ for any $\epsilon \le r_{\max}$, as $r_{\max}>0$ for $(\gb)\in \mathrm{Int}(\OKrouml)$.\hfill $\blacksquare$

\subsection{Proof of \texorpdfstring{\Cref{item:point2HL}}{Item 2} of \texorpdfstring{\Cref{thm:noacc_hess_lip}}{Theorem 5.2}}

As for~\Cref{item:point2HL} of~\Cref{thm:noacc_hess_lip}, it can be obtained from a scaling argument provided by the following lemma.
\begin{Lemma}\label{lem:psi_etendue}
      Let $0<\mu<L$, $K\geq2$ and $(\gb)\in \mathrm{Int}(\OKrouml)$. Let $\phigbml \in \Ftml$ be a function that cycles over $\TikCircle_K$ obtained by \Cref{thm:noacc_hess_lip}. Define $\phigbmleps^{(\lambda)}$ as
    \begin{equation*}
        \phigbmleps^{(\lambda)}(x) \eqdef \lambda^2 \phigbmleps\left(\frac{1}{\lambda}x\right).
    \end{equation*}
    Then:
    \begin{enumerate}
        \item $\phigbmleps^{(\lambda)}\in C^{\infty}$ and $\phigbmleps^{(\lambda)}\in \mathcal F_{\mu,L}^{\tau/\lambda}$.
        \item  $\eqref{eq:hb}_{\gb}(\phigbmleps^{(\lambda)})$ cycles on the scaled roots of unity cycle $\lambda \times  \TikCircle_K$.
    \end{enumerate}
\end{Lemma}

The function $\phigbmleps^{(\lambda)}$ is chosen such that the gradients scale proportionally with $\lambda$, (as the cycle $\lambda \times  \TikCircle_K$), the Hessian is un-scaled, and the third-order derivative scales with $1/\lambda$.
\begin{proof}
   First, $\phigbmleps\in C^{\infty}$, $\phigbmleps^{(\lambda)}\in C^{\infty}$. Second, $\phigbmleps^{(\lambda)} \in \Fml$ as $\nabla^2 \phigbmleps^{(\lambda)}(x) = \nabla^2 \phigbmleps\left(\frac{1}{\lambda}x\right)$ and $\phigbmleps\in \Fml$. Third, for any $r\geq 3$, $\nabla^r \phigbmleps^{(\lambda)}(x) = \frac{1}{\lambda^{r-2}}\nabla^r \phigbmleps\left(\frac{1}{\lambda}x\right)$.
   
   Furthermore, $\nabla \phigbmleps^{(\lambda)}(x) = \lambda \nabla \phigbmleps\left(\frac{1}{\lambda}x\right)$, hence $\nabla \phigbmleps^{(\lambda)}(\lambda x) = \lambda \nabla \phigbmleps\left(x\right)$. 
   Therefore, \HBgb$(\phigbmleps^{(\lambda)})$ cycles on $\lambda \TikCircle_K$. 
   Indeed, \HBgb$(\phigbmleps)$ cycles on $\TikCircle_K$, 
   i.e.~$\forall t\in\range{0}{K-1}, x_{t+1}^\circ = x_{t}^\circ - \gamma \nabla \phigbmleps (x_{t}^\circ) + \beta (x_{t}^\circ - x_{t-1}^\circ)$. 
   Multiplying by $\lambda$, we get $\forall t\in\range{0}{K-1}, \lambda x_{t+1}^\circ = \lambda x_{t}^\circ - \gamma \lambda \nabla \phigbmleps (x_{t}^\circ) + \beta (\lambda x_{t}^\circ - \lambda x_{t-1}^\circ) = \lambda x_{t}^\circ - \gamma \nabla \phigbmleps^{(\lambda)} (\lambda x_{t}^\circ) + \beta (\lambda x_{t}^\circ - \lambda x_{t-1}^\circ)$, that is \HBgb$(\phigbmleps^{(\lambda)})$ cycles on $\lambda \TikCircle_K$.
\end{proof}

\paragraph{Proof of~\Cref{thm:noacc_hess_lip} (\Cref{item:point2HL}).} It is a direct consequence of \Cref{lem:psi_etendue}, by choosing $\lambda $ large enough, it shows that for all $\tau> 0$, $\eqref{eq:hb}_{\gb}$ has a cycle on $\Ftml$. \hfill$\blacksquare$

\subsection{Beyond third-order regularity}

As a conclusion of this section, we showed that Hessian-Lipschitz continuity does not help to obtain acceleration of~\eqref{eq:hb}. It turns out that the arguments work beyond third-order regularity, and that no Lipschitz argument on higher-order derivative can help improving the situation.

In short, the arguments used for controlling the third-order derivative in this section can be extended to impose any bound on higher-order derivatives of the function. In particular, the scaling argument made in \Cref{lem:psi_etendue} enables to arbitrarily reduce all derivatives of order higher than  2 of the $\mathcal C^{\infty}$ function obtained in \Cref{lem:psi_convolee}, whose derivatives are all uniformly bounded. This extension is thus for free: we choose to focus on the third-order derivative in this section for simplicity of exposition, but the generalization follows naturally.

In the last section, we provide an in-depth analysis of the structure of the cycles, beyond dimension 2, that supports the seemingly arbitrary choice of roots-of-unity cycles made in \Cref{sec:noaccel}.

\section{General study of cycles for stationary first-order methods}
\label{sec:cycles}
In this section, we investigate the set $\Ocyml$ of $(\gb)$ for which $\eqref{eq:hb}_{\gb}$ has  a cycle (in the sense of \Cref{def:cycle}), without specifically focusing on roots-of-unity cycles $\TikCircle_K$. We explain why this particular cycle shape, that led to the main result in  \Cref{sec:noaccel}, arises as a natural candidate when studying cycles for stationary first-order methods. Informally, we show that: \textit{if $\eqref{eq:hb}_{\gb}$ has a cycle, then $\eqref{eq:hb}_{\gb}$ has a symmetric cycle}.

\begin{Rem}
    The arguments made in this section directly apply to any \emph{stationary} first-order method, that is, to any method whose dynamic does not change along the iterations \citep[see~Definition 2.1,][]{goujaud2023counter}. For simplicity and coherence, we only instantiate the construction on~\eqref{eq:hb}.
\end{Rem}

Our approach is decomposed into three steps. First, in \Cref{subsec:sdplift}, we show that the existence of a cycle can be cast as a SDP: to that end, we follow the classical  \textit{performance estimation} approach~\citep{drori2014performance,taylor2017smooth}.
Second, in \Cref{subsec:cyc_sym}, we leverage the convexity of the problem in its SDP form and  the structure of the cycle-existence problem to obtain a solution that admits particular symmetries: first on the space of Gram matrices in $\mathcal S_K^+$ in \Cref{subsubsec:circulant}, then to  decompose the cycle onto low-dimensional subspaces in \Cref{subsubsec:cycles_decomp}, and ultimately, to rewrite the problem of finding a cycle as a linear problem in \Cref{subsubsec:LP}. Finally, in \Cref{sec:num}, we establish numerically that, for \eqref{eq:hb}, the set of parameters for which there exists a cycle in dimension 2 is the same as the one for which there exists any cycle.

\subsection{Casting the existence of a cycle as a convex feasibility problem}
\label{subsec:sdplift}
Let $(\gb) \in \Oqml$. In this section we approach cycles from an optimization point of view, similar in spirit with~\citet{goujaud2023counter}, see~\eqref{eq:cycle_search_problem}. However, we here choose to directly write the existence of a cycle with period~$K$ as the following feasibility problem:
\begin{equation}
~ \exists (x_0, \dots, x_{K-1})\neq(\bar x, \dots, \bar x),  ~ \exists f\in\mathcal{F}_{\mu, L} | ~ \HB_{\gb}(f) \text{ cycles on } (x_{t})_{0\leq t \leq K-1},  \tag{$\mathcal{P}'_K$} \label{eq:cycle_feasibility}
\end{equation}
where $\bar x \eqdef \sum_{i=0}^{K-1} x_i /K$ is used in place of $x_0$ to avoid constant cycles,  while preserving symmetry. By \Cref{def:cycle} and \Cref{def:cycle_region}, $(\gb)\in \Ocyml$ if and only if there exists $K\geq 2$ such that $\eqref{eq:cycle_feasibility}$ holds.

We now fix $(\gb) \in \Oqml$, and $K \geq 2$ in the rest of the section.
 We denote $\mathcal S_K^+(\mathbb R)$ the cone of symmetric positive semi-definite (p.s.d.) matrices. We prove the following result.
\begin{Th}[Cycle as an SDP]\label{thm:cycle_as_SDP}
The feasibility problem \eqref{eq:cycle_feasibility} is equivalent to the following feasibility problem:
\begin{equation}
   \exists G \in \mathcal{S}_K^+(\mathbb R) , G \neq \mathbf{0}_K,  G\mathbf{1}_K = \mathbf{0}_K ,~ \exists F\in \mathbb R^K ~|~ \forall i, j  \in \llbracket0, K-1\rrbracket, \left<F, e_i - e_j\right> \geq \left<G, M_{i, j}\right> \label{eq:feasibility_problem_lifted} \tag{$\mathcal{P}_{K-\mathrm{SDP}}$}
\end{equation}
where $(e_i)$ corresponds to the $(i+1)^{\text{th}}$ canonical vector in $\mathbb R^K$ and the matrices $M_{i, j}$ are fixed.
\end{Th}
In words, the feasibility problem is equivalent  to the existence of a p.s.d.~matrix $G$ and  a vector $F$ satisfying a list of linear inequalities, which is usually referred to as a semi-definite program (SDP) \citep{vandenberghe1996semidefinite}.
The proof, which is given below,  is decomposed into three steps. In short, first, we  give a necessary and sufficient condition on the gradients of $f$ for $\eqref{eq:hb}_{\gb}(f)$ to cycle over a $(x_0, \dots x_{K-1})$. Second, we characterize by a list of inequalities the \textit{existence} of a function $f$ in the class $\Fml$ that admits those specific gradients. Those inequalities are referred to as \textit{interpolation conditions}. Third, we rewrite the feasibility problem in terms of the p.s.d.~Gram matrix of the translated iterates $(x_0- \bar x, \dots x_{K-1}-\bar x)$ and the vector of function values $(f(x_0), \dots f(x_{K-1}))$: all constraints are then linear, and the problem writes as an~SDP.

\begin{Rem}[Link with performance estimation problems (PEPs)]
This approach is essentially the one systematically used in \textit{performance estimation} (see, e.g.~\citep{drori2014performance,taylor2017exact,taylor2017smooth} for details), to cast the derivation of worst-case guarantees of first-order optimization methods as SDPs. As proposed by~\citep{goujaud2023counter}, formulating and solving~\eqref{eq:cycle_search_problem} through SDP formulations can be done numerically with appropriate PEP software~\citep{goujaud2022pepit,taylor2017performance} and SDP software~\citep{mosek}.
\end{Rem}

\begin{proof}
We fix $(\gb) \in \Oqml$, and $K \geq 2$ in the proof.
We first observe that for any cycle under consideration $  (x_0, \dots, x_{K-1})$, and any function $ f\in\mathcal{F}_{\mu, L} $,  $ ~ \HB_{\gb}(f) $ cycles on $ (x_{t})_{0\leq t \leq K-1}$ if an only if,  for any $t\in \llbracket0; K-1\rrbracket$
\begin{equation}
    \nabla f(x_t) = \frac{(1 + \beta)x_t - x_{t+1} - \beta x_{t-1}}{\gamma}, \label{eq:hb_inverse}
\end{equation}
where the sequence $(x_t)_{t}$ is extended to $t\in \mathbb Z$  in this proof by $K$-periodicity as $x_t\eqdef  x_{t\mod{K}}$ (in fact, \eqref{eq:hb_inverse} only requires to introduce $x_{-1}$ and $x_{K}$). This directly follows from inverting \eqref{eq:hb} recursion to obtain the unique value of the gradients that result in a particular cycle.\footnote{This derivation generalizes the one of \eqref{eq:hb_inv} in the proof of \Cref{lem:gt_def} to any cycle.}

Therefore, for $(x_0, \dots, x_{K-1})\neq (\bar x, \dots, \bar x)$ there exists a function $f\in\Fml$ such that $\eqref{eq:hb}_{\gb}(f)$ cycles over $(x_0, \dots, x_{K-1})$ if and only if  there exists a function $f\in\Fml$ verifying \eqref{eq:hb_inverse}. This problem is known as \emph{interpolation problem} of the class $\Fml$ and a necessary and sufficient condition is given by~\citep[Theorem 4]{taylor2017smooth}, which is recalled below.

\begin{Lemma}\label{lem:interpol}
[$\Fml$-interpolation, see~\citep{taylor2017smooth}]
    Let $\mathcal{I}$ a set of indices and $(x_i, g_i, f_i)_{i\in\mathcal{I}}$ a family of triplets. There exists a function $f\in\Fml$ verifying $\forall i \in \mathcal{I}, ~ f(x_i) = f_i, $ and $ \nabla f(x_i) = g_i$,
    if and only
    \begin{equation*}
        \forall i, j \in \mathcal{I}, ~ f_i - f_j \geq  \left<g_j, x_i - x_j\right> + \frac{1}{2L}\|g_i - g_j\|^2 + \frac{\mu}{2(1 - \kappa)}\|x_i - \tfrac{1}{L}g_i - x_j + \tfrac{1}{L}g_j\|^2.
    \end{equation*}
\end{Lemma}
Consequently, using \eqref{eq:hb_inverse} in \Cref{lem:interpol}, we obtain that \eqref{eq:cycle_feasibility} is equivalent to
\begin{align}
\exists (x_0, \dots, x_{K-1}) & \neq (\bar x, \dots, \bar x), ~ \exists (f_0, \dots, f_{K-1}) ~|~ \forall i, j \in \llbracket 0, K-1\rrbracket, \nonumber \\
    f_i - f_j \geq  &   \left<\frac{(1 + \beta)x_{j} - x_{j+1\; } - \beta x_{j-1\; }}{\gamma}, x_i - x_j\right> \label{eq:ineg_ij} \tag{$\mathrm{IC}_{i,j}$} \\
    & + \frac{1}{2L}\left\|\frac{(1 + \beta)x_{i} - x_{i+1\; } - \beta x_{i-1\; }}{\gamma} - \frac{(1 + \beta)x_{j} - x_{j+1\; } - \beta x_{j-1\; }}{\gamma}\right\|^2 \nonumber \\
    & + \frac{\mu}{2(1 - \kappa)}\left\|\frac{(1 + \beta - L\gamma)x_{i} - x_{i+1\; } - \beta x_{i-1\; }}{L\gamma} - \frac{(1 + \beta - L\gamma)x_{j} - x_{j+1\; } - \beta x_{j-1\; }}{L\gamma}\right\|^2. \nonumber
\end{align}
Under this form, this feasibility problem is not convex due to quadratic terms in $(x_t)_{0\leq t\leq K-1}$. However, all terms involving $(x_t)_{0\leq t\leq K-1}$ are exactly quadratic. We therefore introduce the Gram matrix $G$ of vectors $(x_t - \bar x)_{0\leq t\leq K-1}$, and the vector~$F = (f_0, \dots, f_{K-1})^T$:
\begin{equation*}
    G = \begin{pmatrix}
    x_0 - \bar x, \dots, x_{K-1} - \bar x
    \end{pmatrix}^T \begin{pmatrix}
    x_0 - \bar x, \dots, x_{K-1} - \bar x
    \end{pmatrix} = \big(\left<x_i - \bar x, x_j - \bar x\right>\big)_{0 \leq i, j \leq K-1}.
\end{equation*}
The matrix $G$ is symmetric positive semi-definite. Moreover, for any $(i,j)$, (1) the right hand side of \eqref{eq:ineg_ij} can  be written as~$\left<G, M_{i, j}\right>$, that is, as a linear combination of the coefficients of~$G$, for a matrix~$M_{ij}$ obtained from \eqref{eq:ineg_ij}; and (2)  the  left hand side of  \eqref{eq:ineg_ij} can  be written as~$\left<F, e_{i}- e_{ j}\right>$, where $e_i$ denotes the $(i+1)^{\text{th}}$ vector of the canonical basis. This method, referred to as \emph{SDP lifting} thus linearizes the above problem as
\begin{equation*}
    \exists G \in \mathcal{S}_K^+(\mathbb R) , G \neq \mathbf{0}_K,  G\mathbf{1}_K = \mathbf{0}_K, ~ \exists F\in \mathbb R^K |~ \forall i, j \in \llbracket0, K-1\rrbracket, \left<F, e_i - e_j\right> \geq \left<G, M_{i, j}\right>
\end{equation*}
for some matrices $(M_{i, j})_{i, j}$ independent of the variables of the problem (in particular independent of the cycle itself).  Note that the constraint $ G \neq \mathbf{0}_K$ comes from the condition $(x_0, \dots, x_{K-1})  \neq (\bar x, \dots, \bar x)$ and the constraint $G\mathbf{1}_K = \mathbf{0}_K$ comes from the fact that $\frac{1}{K}\sum_{k=0}^{K-1} (x_k - \bar{x}) = 0$. Overall, this corresponds to the feasibility problem given as \eqref{eq:feasibility_problem_lifted}.

\end{proof}

\vspace{1em}
\noindent
This proof also provides a slightly stronger result, that is leveraged in the following.
\begin{Th}\label{thm:from_gram_to_cycle}
    The feasibility problem \eqref{eq:cycle_feasibility} is equivalent to~\eqref{eq:feasibility_problem_lifted}. Moreover,
    \begin{enumerate}[leftmargin=*]

        \item \label{item:1page12} \begin{enumerate}
            \item  \label{item:1a} For any solution $(F,G)\in \R^K \times \mathcal S_K^+(\R)$ of \eqref{eq:feasibility_problem_lifted}, there exist points $(x_0, \dots , x_{K-1})$ in dimension at most $K-1$ such that $G$ is the Gram matrix of $(x_0, \dots , x_{K-1})$.
            \item \label{item:1b} Moreover for all such $(x_0, \dots , x_{K-1})$, there exists a function $f\in \Fml$ such that $\eqref{eq:hb}_{\gb}(f)$ cycles over $(x_0, \dots , x_{K-1})$, $F=(f(x_k))_{k\in \range{0}{K-1}}$ and $G$ is the Gram matrix of the vectors $(x_0, \dots , x_{K-1})$.  \end{enumerate}
        \item For any points $(x_0,  \dots , x_{K-1})$ and  function $f\in \Fml$ solution of \eqref{eq:cycle_feasibility}
        then  with $F=(f(x_k))_{k\in \range{0}{K-1}}$ and $G$  the Gram matrix of the vectors $(x_0, \dots , x_{K-1})$, we have that $(F,G)$ is a solution of \eqref{eq:feasibility_problem_lifted},
    \end{enumerate}
\end{Th}
This results thus links the solution of the two problems.

\subsection{Building a symmetric feasible point from a given feasible point}
\label{subsec:cyc_sym}

\subsubsection{Circulant solution to \texorpdfstring{\eqref{eq:feasibility_problem_lifted}}{PKSDP}}
\label{subsubsec:circulant}
In this section, we leverage simultaneously the existence of a cycle under the  SDP form given by \Cref{thm:cycle_as_SDP} and the initial form of the problem as the existence of a cycle for a \textit{stationary} first-order method -- which is the case for \eqref{eq:hb}. We consider $K, \gb$ to be fixed in what follows.

In short, the proof builds upon the intuition that all iterates within the  cycle play a symmetric role. As a consequence, from a given cycle $\mathcal C_0 = (x_0, x_1, \dots, x_{K-1})$ with $G_0$ the Gram matrix of $(x_k)_{k\in \range{0}{K-1}}$, we have access to $K-1$ other cycles $\mathcal C_1 \eqdef (x_1, x_2, \dots, x_{K-1}, x_0)$, $\dots, \mathcal C_s = (x_s, x_{s+1} \dots, x_{s-1}), \dots$, for which the Gram matrix $(G_s)_{s\in \range{0}{K-1}}$ is obtained by  applying a circular permutation to the rows and  columns of $G_0$. We then average $G_0, \dots, G_{K-1}$: $\bar G=\frac{1}{K} \sum_{s=0}^{K-1} G_s$ is a solution to the problem, and a circulant matrix.

\begin{Def}[Circulant matrix]
    We denote $J_K\eqdef (\delta_{i+1-j \mod{K}})_{1\leq i, j\leq K} =
  { \footnotesize \begin{pmatrix}
        0 & 1 & 0 & \dots & 0 \\
        0 & 0 & 1 & \dots & 0 \\
        \vdots  &     &     & \ddots & \vdots  \\
        0 & & &        & 1 \\
        1 & 0 & 0 & \dots  & 0
    \end{pmatrix}}$. A~matrix $M$ of dimension $(K \times K)$ is said to be \textit{circulant} if it is equal to a polynomial in $J_K$, i.e.~there exist $ (c_0, c_1, \dots, c_{d-1})$ such that
    \begin{equation*}
         M = c_0 I_d + c_1 J + \dots + c_{K-1} J^{K-1} =
      { \footnotesize  \begin{pmatrix}
            c_0     & c_1 & c_2 & \dots  & c_{K-1}     \\
            c_{K-1}     & c_0 & c_1 &        & c_{K-2} \\
            c_{K-2} & c_{K-1} & c_0 &        & c_{K-3} \\
            \vdots  &     &     & \ddots & \vdots  \\
            c_1     & c_2 & c_3 & \dots & c_0
        \end{pmatrix}}.
    \end{equation*}
\end{Def}
As Gram matrices are also symmetric, we will have an additional constraint that $c_1=c_{K-1}$, $c_2 =c_{K-2}$, etc. We establish the following result describing symmetric solutions to~\eqref{eq:feasibility_problem_lifted}.
\begin{Th}[Symmetries of the cycle]\label{thm:circulant_sol}
    If \eqref{eq:feasibility_problem_lifted} admits a solution, then \eqref{eq:feasibility_problem_lifted} admits a solution $(\bar F, \bar G)$ with $\bar F=\mathbf{0}_K$ and $\bar G$ a (symmetric PSD) circulant matrix.
\end{Th}

In short, from a solution to \eqref{eq:feasibility_problem_lifted}, we build $K-1$ other solutions by performing a circular permutation of the elements of the cycle. This corresponds to applying a circular permutation matrix to the Gram matrix of the iterates. We can then \textit{average} the $K$ solutions of the problem: by convexity of the set of solutions of the SDP, the resulting Gram matrix is still a solution to the problem. We make this argument precise in the following proof.

\begin{proof}{}
    Assume there exists a solution $(F_0, G_0)\in\mathbb{R}^K \times \mathcal{S}_K^+(\mathbb R)$ of \eqref{eq:feasibility_problem_lifted}, i.e.,~that~$(F_0, G_0)$ are such that~$G_0 \neq 0_K$, $G\mathbf{1}_K = 0_K$ and for all $ i, j \in \llbracket0, K-1\rrbracket, \left<F_0, e_i - e_j\right> \geq \left<G_0, M_{i, j}\right>$. From \Cref{thm:from_gram_to_cycle}-\ref{item:1page12}
    there exists a cycle $ \mathcal C_0 = (x_0, \dots, x_{K-1})$ and a function $f\in\mathcal{F}_{\mu, L} $ such that $\HB_{\gb}(f) $ cycles on $ \mathcal C_0 $, and $G_0 = (\left<x_i ,  x_j \right>)_{0\leq i, j \leq K-1}$   is the Gram matrix of the iterates and $F_0 = (f(x_i))_{0\leq  i\leq K-1}$ the vector of function values.

    Furthermore, for a stationary method, cycling over $\mathcal C_0 = (x_0, \dots, x_{K-1})$ is equivalent to cycling over $\mathcal C_s = (x_s, x_{s+1} \dots, x_{s-1})$ for any $s\in\range{0}{K-1}$. This  leads to $K-1$ other solutions to \eqref{eq:feasibility_problem_lifted}: \[F_s \eqdef  (f(x_{i+s}))_{0\leq i \leq K-1},  \quad G_s\eqdef (\left<x_{i+s},  x_{j+s} \right>)_{0\leq i, j \leq K-1}.\]
    Remarkably, the solution $(F_s,G_s)$ is obtained from $F_0,G_0$ by cyclically permuting the elements of $F_0$ as well as the rows and columns of $G_0$.
    Indeed,  for $F_s$, we start with the $(s+1)^{\mathrm{th}}$ element of $F$, and end with its $s^{\mathrm{th}}$ element. Similarly, we start in $G_s$ with the $(s+1)^{\mathrm{th}}$ row and column of $G$ and end with its $s^{\mathrm{th}}$ row and column. Mathematically, for all $ s\in \range{0}{K-1}$, we have that  $F_s = J_K^{-s} F_0$ and $G_s= J_K^{-s} G_0 J_K^s$. And $(F_s ,G_s)$ is a solution to \eqref{eq:feasibility_problem_lifted}.
    By convexity of the set of solutions to~\eqref{eq:feasibility_problem_lifted},
    \[(\bar F, \bar G)\eqdef \left(\frac{1}{K} \sum_{s=0}^{K-1} F_s, \frac{1}{K} \sum_{s=0}^{K-1} G_s\right)=\left(\frac{1}{K} \sum_{s=0}^{K-1} J_K^{-s} F_0, \frac{1}{K} \sum_{s=0}^{K-1} J_K^{-s} G_0 J_K^s\right)\] is also solution to~\eqref{eq:feasibility_problem_lifted}.
    Moreover, note that the vector $\bar F$ is colinear with $\mathbf{1}_K$ and that it only appears in~\eqref{eq:feasibility_problem_lifted} via inner products with vectors orthogonal to $\mathbf{1}_K$ (only differences between 2 components matters).
    Therefore, $(\mathbf{0}_K, \tfrac{1}{K} \sum_{s=0}^{K-1} J_K^{-s} G J_K^s)$ is also solution to~\eqref{eq:feasibility_problem_lifted}. We use the following fact to conclude.
    \begin{Fact}
        A matrix $M$ is circulant if and only if $M=J_K^{-1} M J_K$.
    \end{Fact}
 Thus $\bar G= \tfrac{1}{K} \sum_{s=0}^{K-1} J_K^{-s} G J_K^s$ is a circulant matrix, as \[J_K^{-1}\bar G J_K= J_K^{-1} \left(\frac{1}{K} \sum_{s=0}^{K-1} J_K^{-s} G J_K^s \right) J_K =  \frac{1}{K} \sum_{s=1}^{K} J_K^{-s} G J_K^s=  \frac{1}{K} \sum_{s=0}^{K-1} J_K^{-s} G J_K^s= \bar G,\]
thereby arriving to the desired claim.
\end{proof}

\vspace{1em}
\noindent
In the next section, we rely on the symmetries of the Gram matrix to gain insights on the shape of a corresponding cycle $(x_k)_{k\in \range{0}{K-1}}$.

\subsubsection{From symmetries on the Gram matrix to symmetric cycle shapes}
\label{subsubsec:cycles_decomp}
Leveraging \Cref{thm:from_gram_to_cycle}, we obtain that any cycle $(x_0, \dots, x_{K-1})$ obtained from a circulant matrix $\bar G$ (i.e.,~such that $\bar  G$ is its Gram matrix), has multiple symmetries. This is formalized in the following  corollary.

\begin{Cor}\label{cor:symmetric_cycles}
If $\eqref{eq:hb}_{\gb}$  has a cycle\footnote{in the sense of  \Cref{def:cycle}, items~\ref{item:2} and \ref{item:3}.} on $\Fml$ then $\eqref{eq:hb}_{\gb}$ cycles on a \emph{symmetric cycle} $(x_i)_{i\in \range{0}{K-1}}$ on $\Fml$. A symmetric cycle   is such that its Gram matrix is symmetric circulant, i.e.:
\begin{enumerate}
    \item for all $t\in \range{0}{K-1}$,  $\|x_t\|^2 = c_0$, i.e.~all iterates are on a sphere,
    \item for all $t\in \range{0}{K-1}$, $\langle x_t, x_{t+1}\rangle = c_1$, i.e.~the inner product between two consecutive iterates is constant along the cycle,
    \item more generally, there exist $(c_s)_{s\in \range{0}{K-1}}$, such that for all $s, t\in \range{0}{K-1}$,  $\langle x_t, x_{t+s}\rangle = c_s$, i.e.~the inner product between two $s$-separated iterates is constant along the cycle.
\end{enumerate}
\end{Cor}

\begin{proof}
    If $\eqref{eq:hb}_{\gb}$  has a cycle on $\Fml$ then by \Cref{thm:cycle_as_SDP} then \Cref{thm:circulant_sol}, there exists a circulant solution to \eqref{eq:feasibility_problem_lifted}. A symmetric cycle is obtained from \Cref{thm:from_gram_to_cycle}-\ref{item:1a} on the circulant solution.
\end{proof}

\vspace{1em}

\paragraph{Motivating the roots-of-unity cyclic structure.} An example of such a symmetric cycle is the roots-of-unity cycle in dimension~2 (see \Cref{def:rou_cycle}), that was the focus of \Cref{sec:noaccel}. It corresponds to the arguably simplest solution to obtain the symmetries mentioned above. \Cref{cor:symmetric_cycles} thus supports the idea of looking for simple two-dimensional roots-of-unity cycles: the study of those particular cycles, that are sufficient to demonstrate the main non-acceleration result~\Cref{thm:incompatibility},  takes its roots in this  higher-level analysis of the cycles as an SDP and the inherent symmetries of the problem.

\begin{Ex}
A straightforward application of the symmetrization process given  in the proof of \Cref{thm:circulant_sol} is the symmetrization of the cycle provided by \citet{lessard2016analysis}, which is a one-dimensional cycle over the three iterates: $(x_0, x_1, x_2)= \frac{1}{1225} (792 , -2208 , 2592)$. The Gram matrix of the centered iterates $(x_0-\bar x, x_1-\bar x, x_2-\bar x)$ is $$G_0= \left(\frac{8}{49}\right)^2{\footnotesize\begin{pmatrix}
    4&  -26&   22 \\
    -26&  169& -143\\
    22& -143&  121
\end{pmatrix}}$$
After the circulation process described in the proof on \Cref{thm:circulant_sol}, we obtain that $\bar G = \frac{8^2}{49} {\footnotesize\begin{pmatrix}
    2&  -1&   -1 \\
    -1&  2& -1\\
    -1& -1&  2
\end{pmatrix}}$.
This Gram matrix is (proportional to) the one of the $3^\text{rd}$-roots-of-unity cycle $\TikCircle_3$.
\end{Ex}

Although we {cannot} prove that in full generality, if $\eqref{eq:hb}_{\gb}$  has a cycle on $\Fml$ then $\eqref{eq:hb}_{\gb}$ cycles over a roots-of-unity cycle, the next result provides a generic decomposition, beyond dimension 2, of the symmetric cycles.

\begin{Prop}\label{prop:shape_sym_cycle_generic}
  If $\eqref{eq:hb}_{\gb}$  has a cycle on $\Fml$ then $\eqref{eq:hb}_{\gb}$ cycles on $K$ points $(x_k)_{k\in \range{0}{K-1}}$ on $\Fml$, for which there exist $\tilde \nu_1, \dots, \tilde \nu_{\lfloor K/2 \rfloor} \geq 0$, such  that for all $k\in\range{0}{K-1}$
    \begin{equation}x_k = {\footnotesize \left(\begin{array}{rll}
        \tilde\nu_1 &\begin{pmatrix}
    \cos\left(k\times\frac{2\pi \times 1}{K}\right) \\
    \sin\left(k\times\frac{2\pi \times 1}{K}\right)
\end{pmatrix}& \textcolor{gray}{[\text{block $1$}]} \\
& \qquad \vdots &
\\
\tilde \nu_\ell &\begin{pmatrix}
    \cos\left(k\times\frac{2\pi \times \ell}{K}\right) \\
    \sin\left(k\times\frac{2\pi \times \ell}{K}\right)
\end{pmatrix}& \textcolor{gray}{[\text{block $\ell$}]}\\
& \qquad \vdots &
\\
\tilde \nu_{\left\lfloor \frac{K-1}{2} \right\rfloor} &\begin{pmatrix}
    \cos\left(k\times\frac{2\pi \times \left\lfloor \frac{K-1}{2} \right\rfloor}{K}\right) \\
    \sin\left(k\times\frac{2\pi \times \left\lfloor \frac{K-1}{2} \right\rfloor}{K}\right)
\end{pmatrix}& \textcolor{gray}{[\text{block $\floor {\frac{K-1}{2}}$}]} \\
\textcolor{gray}{\tilde \nu_{\frac{K}{2}}} & \textcolor{gray}{\begin{pmatrix}
  (-1)^k )
\end{pmatrix}}& \textcolor{gray}{[\text{block $ {\frac{K}{2}}$}, \text{only if $K$ is even}]}
    \end{array}\right)} \in \R^{K-1}     \label{eq:sym_cycle}
\tag{Symmetric Cycle}\end{equation}
\end{Prop}
All points $(x_k)_{k\in \range{0}{K-1}}$ are in dimension $K-1$, as mentioned in \Cref{thm:from_gram_to_cycle}-\ref{item:1page12}. The  points $(x_k)_{k\in \range{0}{K-1}}$  are decomposed over a Cartesian product of $\floor{\frac{K-1}{2}}$ independent two-dimensional spaces: on each block,  the cycle is perfectly regular. For example, on block 1, one recognizes the roots-of-unity shape. If $K-1 $ is even, there are exactly $(K-1)/2$ blocks of dimension 2, and if $K $ is even there are exactly $(K-2)/2$ blocks of dimension 2, and one block  of dimension 1.

\begin{proof}
First, recall that  by \Cref{thm:cycle_as_SDP} then \Cref{thm:circulant_sol},  if $\eqref{eq:hb}_{\gb}$  has a cycle on $\Fml$, then  \eqref{eq:feasibility_problem_lifted} admits a solution $(0, \bar G)$ with $\bar G$ a (symmetric) circulant matrix.

Second, circulant matrices constitute a long standing object of interest in linear algebra, and their reduction properties are well understood~\citep{gray2006toeplitz}. We use the following lemma.
\begin{restatable}{Lemma}{circulantreduction}\label{lem:circulant_dec}
   A matrix $\bar  G$ is symmetric and circulant such that $\bar   G\mathbf{1}_K=0$, if and only is there exist non-negative $ \nu_1, \dots, \nu_{\floor{K/2}}$ such that $\bar  G = \sum_{\ell=1}^{\floor{K/2}}\nu_\ell H_\ell$, with $H_\ell \eqdef \left(\cos\left(\frac{2\pi \ell}{K}|i-j|\right)\right)_{i, j}$.
\end{restatable}
This is a classical result, whose proof is recalled for completeness in \Cref{app:circulant_proof}.
For each $\ell\in\range{0}{\floor{\frac{K}{2}}}$, the matrix $H_\ell$ is a rank 2 matrix and is the Gram matrix of the family of vectors $\begin{pmatrix}
    \cos\left(k\times\frac{2\pi \ell}{K}\right) \\
    \sin\left(k\times\frac{2\pi \ell}{K}\right)
\end{pmatrix}_{k\in\range{0}{K-1}}$ that corresponds to the $\ell^{\mathrm{th}}$ block.

Overall, considering $\nu_1, \dots, \nu_{\floor{K/2}}$  that provide a  decomposition $\bar G$ as in   \Cref{lem:circulant_dec},  and defining $(x_k)_{k\in \range{0}{K-1}}$ by \eqref{eq:sym_cycle} with $\tilde \nu_\ell  = \sqrt{\nu_\ell}$, we obtain that the Gram matrix of $(x_k)_{k\in \range{0}{K-1}}$ is~$\bar G$. Finally,~\Cref{thm:from_gram_to_cycle}-\ref{item:1b} provides the desired claim.
\end{proof}

\vspace{1em} \noindent
Overall, this provides  a complete picture of shape of all cycles of period $K$.  There is no apparent reason for the~\eqref{eq:sym_cycle} to further reduce to dimension 2, and the proof of such a result is left as an open question. In~\Cref{sec:num}, we display a numerical comparison of the cycles obtained analytically as roots-of-unity cycles in dimension 2 and the ones obtained numerically by solving directly~\eqref{eq:feasibility_problem_lifted}. Before turning to this numerical study, we underline that our analysis enables to rewrite the existence of a cycle as a linear feasibility problem.

\subsubsection{Casting the problem of finding cycles as a linear feasibility problem}
\label{subsubsec:LP}
A notable byproduct of \Cref{lem:circulant_dec} is the decomposition of any circulant matrix as a positive linear combination of elementary matrices. This enables to reparametrize \eqref{eq:feasibility_problem_lifted} as a linear problem:
\begin{Th}\label{thm:LP}
  \eqref{eq:cycle_feasibility} and \eqref{eq:feasibility_problem_lifted} are equivalent to the following linear problem:
    \begin{equation}
   \exists \nu \in \mathbb{R}_{\geq 0}^{\floor{K/2}}, \nu\neq \mathbf{0}_{\floor{K/2}}, ~  \forall  i  \in \llbracket0, K-1\rrbracket, 0 \geq  \sum_{\ell=1}^{\floor{K/2}}\nu_\ell \left< H_\ell, M_{i, j=0}\right>  \label{eq:feasibility_problem_as_LP} \tag{$\mathcal{P}_{K-\mathrm{LP}}$}
\end{equation}
Equivalently, introducing the matrix $ P =\left(\left< M_{i, j=0}, H_\ell\right>_{i \in \range{0}{K-1}, \ell \in \range{1}{\floor{K/2}} }\right)  \in \R^{(K-1) \times \floor{K/2}}$, \eqref{eq:feasibility_problem_as_LP} writes, for $P \nu\in \R^{K-1}$, as:
\[ \exists \nu \in \mathbb{R}_{\geq 0}^{\floor{K/2}},  \quad P \nu \leq  0. \]
\end{Th}
\begin{proof}
We use \Cref{thm:circulant_sol} to obtain a circulant symmetric solution from \eqref{eq:feasibility_problem_lifted}, then by \Cref{lem:circulant_dec} all circulant symmetric matrices $\bar G$ are written as a linear combination   $\sum_{\ell=1}^{\floor{K/2}}\nu_\ell H_\ell$, with $(\nu_\ell)_{\ell\in \range{1}{\floor{K/2}}} \in \mathbb{R}_{\geq 0}^{\floor{K/2}} $. We obtain \eqref{eq:feasibility_problem_as_LP} by parametrizing the problem by $\nu$, as for any $(i,j)$,  $ \left<\bar G, M_{i, j}\right>  = \sum_{\ell=1}^{\floor{K/2}}\nu_\ell \left< H_\ell, M_{i, j}\right>  $, and by observing that for any $(i,j)$,  $\left<\bar G, M_{i, j}\right> \le 0$ if  and only if $\left<\bar G, M_{i-j, 0}\right> \le 0$, as $\bar G$ is circulant.
\end{proof}  \vspace{1em} \noindent

In conclusion, we observe that a cycle can either be parametrized by
\begin{enumerate}[itemsep=0.3pt,leftmargin=*]
    \item Initially, in \Cref{def:cycle} and \eqref{eq:cycle_feasibility}, by $(x_k)_{k\in \range {0}{K-1}} \in (\R^d)^K$ and $f\in \Fml$.
    \item Second, by $G\in \mathcal S_K^+(\R)$  in \eqref{eq:feasibility_problem_lifted} and \Cref{thm:cycle_as_SDP}, with $(K-1)^2$ constraints.
    \item Then, using  \Cref{thm:circulant_sol}, by $(c_k)_{k\in \range{0}{K-1}}$ the first row of the circulant matrix $\bar G$ . But the constraints on $c_0, \dots c_{K-1}$ correspond to an SDP  type constraint.
    \item Finally,  in \Cref{prop:shape_sym_cycle_generic} and \eqref{eq:feasibility_problem_as_LP} ,  by $(\tilde \nu_\ell)_{\ell\in \range{0}{\floor{K/2}}}$, onto which only $K-$ \textit{linear} constraints hold.
\end{enumerate}
The latest parametrization, as~\eqref{eq:feasibility_problem_as_LP}, naturally  provides the best numerical results, that are given in the next section.

\subsection{Numerical results on \texorpdfstring{\eqref{eq:hb}}{(HB)}}
\label{sec:num}

\begin{figure}[t]
   \begin{subfigure}[t]{.49\linewidth}
    \centering
    \includegraphics[width=\linewidth]
       {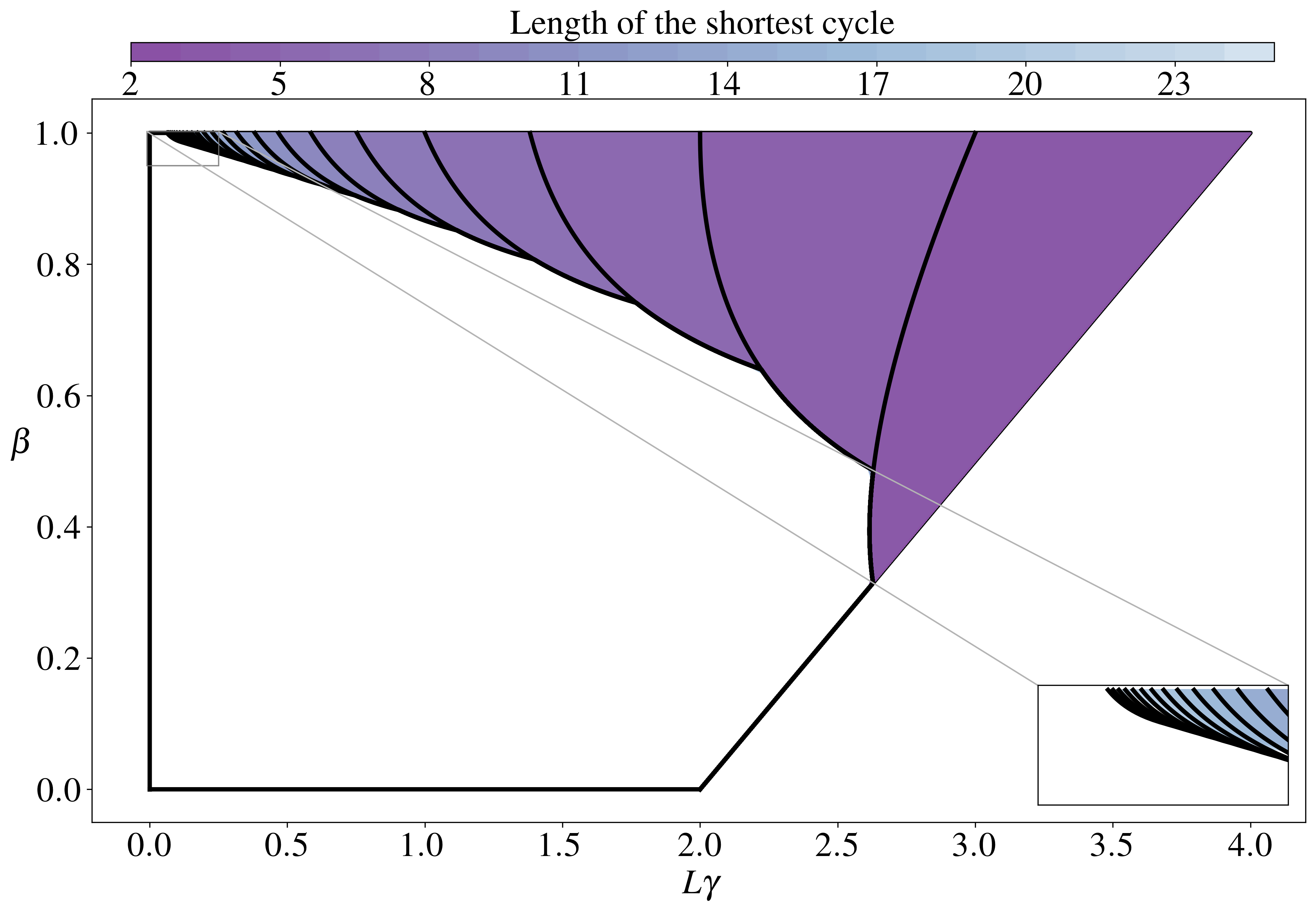}
   \caption{\label{fig:numeric_cycle} For $K\in \range{2}{25}$, comparison between $\OKrouml$ (analytically obtained in \Cref{sec:noaccel}), which borders are represented as black lines,  and $\Ocyml$, represented as the set of purple points, obtained by solving~\Cref{eq:feasibility_problem_lifted}.}
  \end{subfigure}
\hfill
\begin{subfigure}[t]{.49\linewidth}
    \centering
         \includegraphics[width=\linewidth]{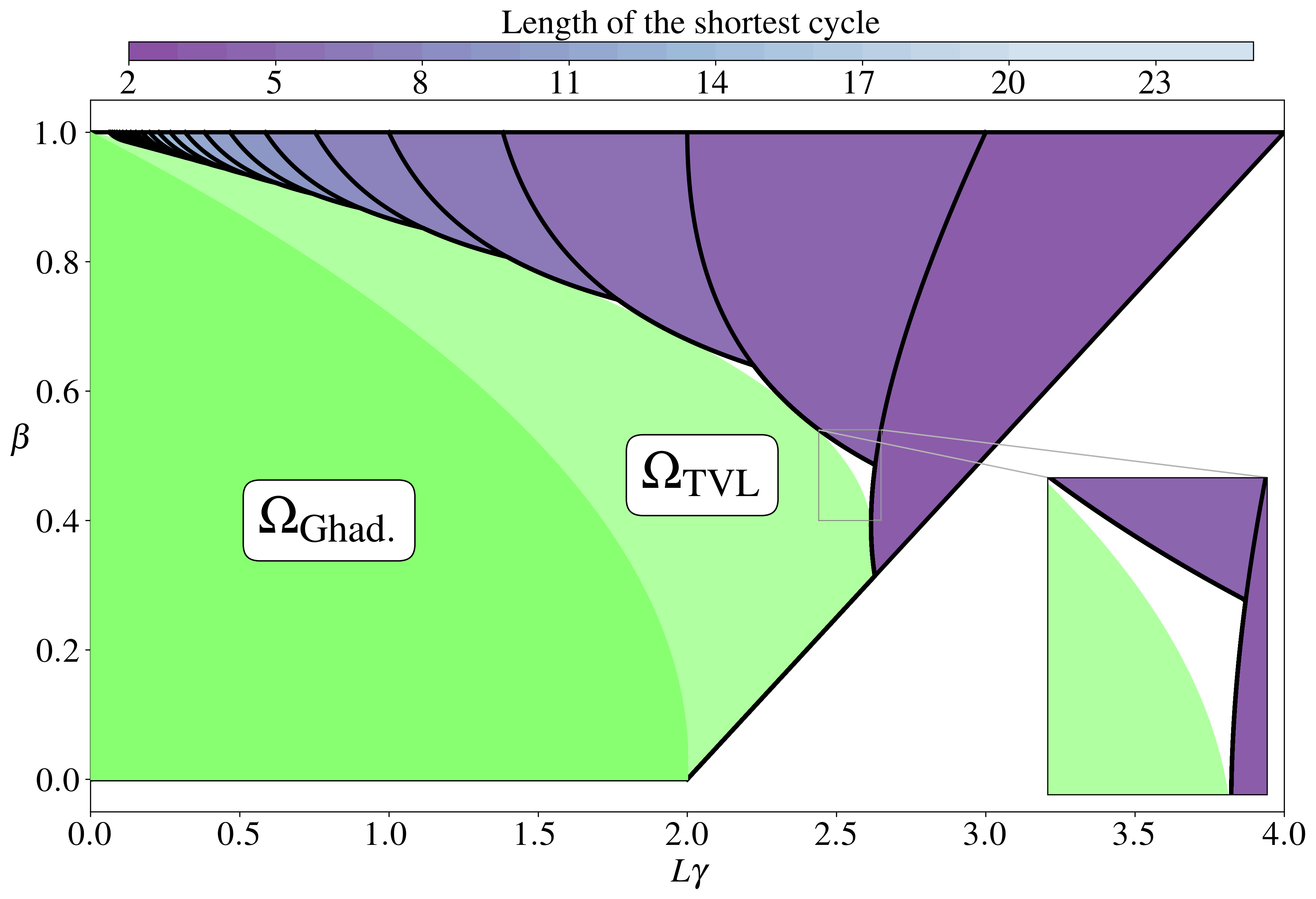}
         \caption{\label{fig:all_regions} Comparison between the hyper-parameter regions $\Ocyml$ for which we obtain cycles, $\Otaml$ for which a Lyapunov is found numerically, and $\Oghml$ for which a Lyapunov is known analytically}
\end{subfigure}
    \caption{\label{fig:numerical_results} Numerical results on the behavior of \eqref{eq:hb} as a function of $(\gb)$.}
\end{figure}

In this section, we provide a comparison between the roots-of-unity cycling region $\Orouml$ obtained analytically in \Cref{sec:noaccel} and $\Ocyml$ obtained numerically by solving \eqref{eq:feasibility_problem_lifted}. On \Cref{fig:numeric_cycle}, we observe that the two sets appear to be identical, i.e.,~numerically  $\Orouml = \Ocyml$.

Secondly, we compare in \Cref{fig:all_regions} the hyper-parameter regions $\Ocyml$ for which we obtain cycles, $\Otaml$ for which a Lyapunov is found numerically in \cite{goujaud2023counter} using the approach of \cite{taylor2018lyapunov}, and $\Oghml$ for which a Lyapunov is known analytically (see \Cref{lem:ghad}). We observe that numerically,  $(\Ocyml)^c $ and $ \Otaml$ are nearly similar, apart from some small neighborhoods (left white on \Cref{fig:all_regions}). These observations can be summarized as follows.
\begin{Fact}For any $0<\mu\le  L$:
\begin{equation*}
     \Oghml \subseteq
     \Otaml \subseteq
     \Ocvml \subseteq
     (\Ocyml)^c \subseteq
     (\Orouml)^c
     \subseteq\Oqml.
\end{equation*}
\end{Fact}
\begin{Conjecture}\label{conj:1}
For any $0<\mu\le L$:
    \begin{equation*}
     \Oghml \subsetneq \Otaml \left \lbrace\begin{array}{cc}
        \subset \Ocvml =   \\
       \emph{\footnotesize{or}} \\
        = \Ocvml \subset
     \end{array}\right \rbrace (\Ocyml)^c  = (\Orouml)^c \subsetneq \Oqml.
\end{equation*}
\end{Conjecture}
Although the proof of this conjecture is left open, it is strongly supported by the numerical experiments given in~\Cref{fig:numerical_results}.

\renewcommand{\psigbml}{\psi_{\gb,\ml}^K}
\renewcommand{\phigbmleps}{\varphi_{\epsilon,\gb,\ml}^K}
\renewcommand{\phigbml}{\varphi_{\gb,\ml}^K}

\section{Concluding remarks}\label{sec:conclusion}

As a brief summary, this work provides a definitive and negative answer to the question of obtaining accelerated convergence rate by using the heavy-ball ($\HB$) method on all smooth strongly convex functions beyond quadratics. In other words, for smooth strongly convex minimization, the complexity of $\HB$ is the same as the one of GD, up to at most a constant $50/3\simeq 17$. Further, we show that this result is stable to reasonable additional assumptions, including that of Lipschitz conditions on the Hessian of the problem, or to functions in the set~$\mathcal{C}^\infty\cap\Fml$, and also robust to perturbations to the initial conditions, as well as to parameter and gradient noise. 

Furthermore, we propose two constructive approaches, based on the construction of cyclic trajectories, for disproving convergence of (stationary) optimization methods. The first one consists in constructing a two-dimensional cycles (namely \emph{roots-of-unity cycles}), and the second one provides a linear program for testing the existence of higher-dimensional cycles.

\paragraph{Future work and open questions.} There remain a few open questions related to the convergence of $\HB$, or that are raised by our work. In particular, we highlight the following points.

First, we now have an upper bound and a lower bound on~\eqref{eq:hb}'s convergence rate, both in $1 - \Theta(\kappa)$. However, those bounds do not match perfectly, as they differ by a constant $\simeq 17$. The theoretical question of the exact rate of $\eqref{eq:hb}$ thus remains open.

Second, our non-acceleration result of~\eqref{eq:hb} holds in any dimension larger than or equal to 2, as our roots-of-unity cycles and counter examples construction only hold in dimension 2 at least. The potential acceleration of \eqref{eq:hb} in a one-dimensional space is thus left open.

Lastly, although we proved that acceleration cannot be achieved by assuming  higher-order Lipschitz-type regularity, it  remains an open question to see if acceleration  can be obtained (a)  on another intermediary  functional class between $\Qml$ and $\Fml$, (b) under additional information on $f\in\Fml$ --  such refined information should go beyond the knowledge of $\mu$ and $L$ and could potentially be based on adaptive tunings (e.g.,~based on online information) of the method.

\paragraph{Conjecture.}
Finally, we mention the following open conjecture, echoing~\Cref{conj:1}.
\begin{Conjecture}
    For any stationary first-order method, if there exists a cyclic trajectory, there exists a roots-of-unity two-dimensional cycle on the function~\eqref{eq:def_function_qui_tue} (where $M$ depends on the stationary first-order method under consideration).
\end{Conjecture}

Proving such a conjecture only requires to show that the~\eqref{eq:sym_cycle} shape obtained in~\Cref{prop:shape_sym_cycle_generic}, can always be reduced to roots-of-unity cycles for any first-order method, as this is the case numerically for $\HB$.
While open, this conjecture, if proven, constitutes a promising direction. Indeed, a practical consequence of this conjecture is that testing existence of a cyclic trajectory of a stationary algorithm would only require testing it on one function of the form~\eqref{eq:def_function_qui_tue} for each cycle length, making the exploration of cyclical behavior of stationary first-order algorithms straightforward and making the cycles stable.


\section*{Declarations}
    
    \textbf{Acknowledgements.}
    The authors thank Margaux Zaffran for her feedback and her insights on the plots.

    \noindent\textbf{Fundings.}
    A.~Taylor acknowledges support from the European Research Council (grant SEQUOIA 724063). This work was partly funded by the French government under management of Agence Nationale de la Recherche as part of the ``Investissements d’avenir'' program, reference ANR-19-P3IA-0001 (PRAIRIE 3IA Institute).
    The work of B. Goujaud and A. Dieuleveut is partly supported by ANR-19-CHIA-0002-01/chaire SCAI, and Hi!Paris FLAG project.

    \noindent\textbf{Conflict of interest.} The authors declare that they have no conflict of interest.
    
    

\bibliographystyle{spbasic}
\bibliography{references}

\newpage
\setcounter{section}{0}
\renewcommand\thesection{\Alph{section}}
\begin{center}
     Appendix
\end{center}

\noindent
This appendix is organized in six sections: \Cref{app:proofquad} to \ref{app:proofLem1_HL} respectively contain complementary proof results to \Cref{sec:preliminary-materials} to \ref{sec:cycles}, and \Cref{app:bonus_all_rates} a summary table of convergence rates on $\Qml$ and $\Fml$  for several first-order methods of interest.

\renewcommand{\cfttoctitlefont}{}
\renewcommand{\cftaftertoctitle}{}
\renewcommand{\contentsname}{}
\hypersetup{linkcolor = black}
	\setlength\cftparskip{2pt}
	\setlength\cftbeforesecskip{2pt}
	\setlength\cftaftertoctitleskip{3pt}
	\addtocontents{toc}{\protect\setcounter{tocdepth}{2}}
	\setcounter{tocdepth}{1}
	\tableofcontents
	\hypersetup{linkcolor={red!50!black}}

\section{Auxiliary proofs from \texorpdfstring{\Cref{sec:preliminary-materials}}{Section 2.1}: Proof of \texorpdfstring{\Cref{prop:conv_quad}}{Proposition 2.1}}\label{app:proofquad}
In this section, we give a complete proof of the asymptotic convergence rate of \eqref{eq:hb} on sets of quadratic functions.

\quadrates*

\begin{proof}
Since, this recursion is of second-order, a classical and convenient trick is to consider the variable \\ $X_t~\eqdef~\begin{pmatrix}
    x_t - x_\star \\
    x_{t-1} - x_\star
\end{pmatrix}~\in~\left(\mathbb{R}^d\right)^2$.
It follows the simplified recursion

\begin{equation}\label{eq:linearalgebraapproach}
    X_{t+1} =
    \begin{pmatrix}
        (1 + \beta)\Id - \gamma H & - \beta \Id \\
        \Id & 0
    \end{pmatrix} X_t,
\end{equation}
that can be enrolled to obtain
\begin{equation*}
    X_{T} =
    \begin{pmatrix}
        (1 + \beta)\Id - \gamma H & - \beta \Id \\
        \Id & 0
    \end{pmatrix}^T X_0,
\end{equation*}
hence the result
\begin{equation*}
    \|x_t - x_\star\| \leq \|X_t\| \leq \left\|
    \begin{pmatrix}
        (1 + \beta)\Id - \gamma H & - \beta \Id \\
        \Id & 0
    \end{pmatrix}^T
    \right\|_{\op}
    \|X_0\|.
\end{equation*}
Note that the last inequality is reached from some $X_0$ by definition of the operator norm.
Finally,
\begin{equation*}
    \|x_t - x_\star\|^{1/T} \leq
    \left\|
    \begin{pmatrix}
        (1 + \beta)\Id - \gamma H & - \beta \Id \\
        \Id & 0
    \end{pmatrix}^T
    \right\|_{\op}^{1/T}
    \|X_0\|^{1/T}
    \underset{T\rightarrow\infty}{\longrightarrow} \rho 
    \begin{pmatrix}
        (1 + \beta)\Id - \gamma H & - \beta \Id \\
        \Id & 0
    \end{pmatrix}
\end{equation*}
where $\rho$ here denotes the spectral radius of the matrix, i.e.~the largest complex module of its eigenvalues.
Since $H$ is diagonalizable (as self-adjoint operator, or symmetric matrix), we can block-diagonalize the previous matrix as
\begin{equation*}
    \begin{pmatrix}
        (1 + \beta)\Id - \gamma H & - \beta \Id \\
        \Id & 0
    \end{pmatrix}
    \sim
    \left(\begin{pmatrix}
        1 + \beta - \gamma \lambda & - \beta \\
        1 & 0
        \end{pmatrix}_{\lambda\in\Sp(H)}\right),
\end{equation*}
and then the worst-case asymptotic convergence rate is upper bounded by
\begin{equation*}
    \rho 
    \begin{pmatrix}
        (1 + \beta)\Id - \gamma H & - \beta \Id \\
        \Id & 0
    \end{pmatrix}
    =
    \max_{\lambda \in [\mu, L]} \rho \begin{pmatrix}
        1 + \beta - \gamma \lambda & - \beta \\
        1 & 0
    \end{pmatrix}.
\end{equation*}
The first thing to notice is that the determinant of
$\begin{pmatrix}
    1 + \beta - \gamma \lambda & - \beta \\
    1 & 0
\end{pmatrix}$
is $\beta$.
Therefore, when $\beta\geq1$, at least one eigenvalue has a module larger than or equal to 1 and \eqref{eq:hb} provably diverges on some function of $\mathcal{Q}_{\mu, L}$.
Then, from now, we only consider $\beta\in(-1, 1)$.

Note that, when $\beta=0$, we recover \eqref{eq:gd} and its convergence rate $\max_{\lambda \in [\mu, L]}|1 - \gamma\lambda|$.
As for \eqref{eq:gd}, \eqref{eq:hb}'s asymptotic convergence rate is given by the extreme eigenvalue $\mu$ or $L$, depending on the value of $\tilde{\gamma} \eqdef \tfrac{\gamma}{1 + \beta}$.
Indeed,
\begin{itemize}
    \item when $\tilde{\gamma} = \tfrac{\gamma}{1 + \beta} \leq \frac{2}{L+\mu}$, then
    $\max_{\lambda \in [\mu, L]} \rho \begin{pmatrix}
        1 + \beta - \gamma \lambda & - \beta \\
        1 & 0
    \end{pmatrix}
    =
    \rho \begin{pmatrix}
        1 + \beta - \gamma \mu & - \beta \\
        1 & 0
    \end{pmatrix}$,
    \item and when $\tilde{\gamma} = \tfrac{\gamma}{1 + \beta} \geq \frac{2}{L+\mu}$, then
    $\max_{\lambda \in [\mu, L]} \rho \begin{pmatrix}
        1 + \beta - \gamma \lambda & - \beta \\
        1 & 0
    \end{pmatrix}
    =
    \rho \begin{pmatrix}
        1 + \beta - \gamma L & - \beta \\
        1 & 0
    \end{pmatrix}$.
    
    Moreover, when $\tilde{\gamma} = \tfrac{\gamma}{1 + \beta} \geq \frac{2}{L}$, then
    $\rho \begin{pmatrix}
        1 + \beta - \gamma L & - \beta \\
        1 & 0
    \end{pmatrix}
    \geq 1$, which cannot guarantee convergence.
\end{itemize}

\begin{table*}[ht]
{
\begin{center}
\caption{Classification of~\eqref{eq:hb}'s behavior in three regions, see~\Cref{fig:1} for a graphical description.}\label{tab:hb_regions}
{\renewcommand{\arraystretch}{1.8}
\begin{tabular}{@{}lcc@{}}
\toprule
Region's name & Range of $\gamma$'s values & Asymptotic convergence rate \\
\midrule
Lazy region & $0 < \gamma \leq \min\left(\frac{2(1+\beta)}{L+\mu}, \frac{\left(1 - \sqrt{\beta}\right)^2}{\mu}\right) $ & $\frac{1 + \beta - \mu\gamma}{2} + \sqrt{\left(\frac{1 + \beta - \mu\gamma}{2}\right)^2-\beta}$ \\ \midrule
Robust region & $\frac{\left(1 - \sqrt{\beta}\right)^2}{\mu}\leq\gamma\leq\frac{\left(1 + \sqrt{\beta}\right)^2}{L}$ & $\sqrt{\beta}$ \\ \midrule
Knife's edge & $\max\left(\frac{2(1+\beta)}{L+\mu}, \frac{\left(1 + \sqrt{\beta}\right)^2}{L}\right) \leq \gamma < \frac{2(1+\beta)}{L}$ & $\frac{L\gamma - (1 + \beta)}{2} + \sqrt{\left(\frac{L\gamma - (1 + \beta)}{2}\right)^2-\beta}$ \\
\bottomrule
\end{tabular}
}
\end{center}}
\end{table*}

Last but not least, all the presented upper bounds are reached on either $x\mapsto \tfrac{\mu}{2} x^2$ or $x\mapsto \tfrac{L}{2} x^2$.
Indeed, for those 1D functions, the system is directly written with a single 2x2 block as
\begin{equation}
    x_T - x_\star = \begin{pmatrix} 1 & 0 \end{pmatrix} X_T = \begin{pmatrix} 1 & 0 \end{pmatrix}
    \begin{pmatrix}
        1 + \beta - \gamma \lambda & - \beta \\
        1 & 0
    \end{pmatrix}^T X_0,
\end{equation}
where $\lambda$ is either $\mu$ or $L$.
Let us triangularize the 2x2 matrix above as $P \begin{pmatrix}
                                                  \rho & \alpha \\
                                                  0 & \rho'
    \end{pmatrix} P^{-1}$ where $\rho$ and $\rho'$ are its two eigenvalues with $|\rho'|\leq|\rho|\neq 0$ ($\rho$ is not 0 and $\rho'$ is if and only if $\beta$ is).
    The upper right coefficient $\alpha$ is 0 if the matrix is diagonalizable and can be anything else otherwise (in this last case, we would have $\rho = \rho'$) but its value will have no impact in the following.
    Note that $|\rho|$ is the quantity of interest since it exactly corresponds to the asymptotic rates given in~\Cref{tab:hb_regions} for $\lambda=\mu$ in the \emph{lazy region}, $\lambda=L$ in the \emph{knife's edge} and any $\lambda\in[\mu, L]$ in the \emph{robust region}.
    Our goal is to prove that the above quantity's norm, for some well-chosen $X_0$, can be lower bounded by $C |\rho|^T$ for some positive constant $C$.
By choosing $X_0 = P\begin{pmatrix}
                          1 \\ 0
    \end{pmatrix}$, we obtain
\begin{equation}
    x_t - x_\star = \begin{pmatrix} 1 & 0 \end{pmatrix}
    P \begin{pmatrix}
          \rho & \alpha \\
          0 & \rho'
    \end{pmatrix}^T \begin{pmatrix}
                          1 \\ 0
    \end{pmatrix} = \begin{pmatrix} 1 & 0 \end{pmatrix}
    P \begin{pmatrix}
                          \rho^T \\ 0
    \end{pmatrix}
    = C \rho^T, \qquad \text{with} ~ C \triangleq \begin{pmatrix} 1 & 0 \end{pmatrix}
    P \begin{pmatrix}
                          1 \\ 0
    \end{pmatrix}.
\end{equation}
Finally, it remains to show that $C \neq 0$ to conclude that $\|x_T - x_\star\|^{1/T} \geq C^{1/T} |\rho| \underset{T\rightarrow\infty}{\longrightarrow} |\rho|$.
We now conclude this proof by proving by contradiction that $C$ cannot be 0.
    Note that $C$ corresponds to the first coordinate of the eigenvector of $\begin{pmatrix}
        1 + \beta - \gamma \lambda & - \beta \\
        1 & 0
    \end{pmatrix}$ associated with $\rho\neq 0$.
    It this coordinate was 0, then this eigenvector would be colinear to $\begin{pmatrix} 0 \\ 1
    \end{pmatrix}$ and we would have that
    $\begin{pmatrix}
        - \beta \\
        0
    \end{pmatrix} = \begin{pmatrix}
        1 + \beta - \gamma \lambda & - \beta \\
        1 & 0
    \end{pmatrix} \begin{pmatrix} 0 \\ 1
    \end{pmatrix} = \rho \begin{pmatrix} 0 \\ 1
    \end{pmatrix}$, implying $\rho=0$ which is excluded.

\end{proof}

\begin{Rem}
This proof is based on linear algebra, relying on writing the system as \eqref{eq:linearalgebraapproach}.
Another classical approach to the analysis of~\eqref{eq:hb} on $\mathcal{Q}_{\mu, L}$ consists in exploiting links between first-order methods and polynomials (see e.g.~\citep{fischer2011polynomial,nemirovskinotes1995} or~\citep[Chapter 2 of][]{d2021acceleration} for a recent introduction.
\end{Rem}

\section{Auxiliary proofs from \texorpdfstring{\Cref{sec:noaccel}}{Section 3}}
\subsection{Proof of \texorpdfstring{\Cref{thm:analytical_ROU_region}}{Theorem 3.5}}
\label{app:proof_of_analytical_ROU_region}

\roucyclesregion*

\begin{proof}
Let $(\gb) \in \Oqml$. 
First we prove the expression of $\OKrouml$. Then we will prove that for all $(\gb) \in \OKrouml$, \eqref{eq:hb}$\,_{\gb}(\psi)$ cycles.

\paragraph{Expression of $\OKrouml$:} We have $(\gb) \in \OKrouml$
\begin{eqnarray} 
    & \overset{(\text{\Cref{lem:gt_def}})}{\Longleftrightarrow} &
    \exists f\in\Fml ~|~ \forall t \in \llbracket 0, K-1 \rrbracket, \nabla f(x_t^\circ) = \frac{(1+\beta) I_2 - R - \beta R^{-1}}{\gamma} x_t^\circ, \nonumber \\
    & \overset{\left(\bar{f}(x) \eqdef \frac{f(x) - \tfrac{\mu}{2}\|x\|^2}{L-\mu}\right)}{\Longleftrightarrow} &
    \exists \bar{f}\in\mathcal{F}_{0, 1} ~|~ \forall t \in \llbracket 0, K-1 \rrbracket, \nabla \bar{f}(x_t^\circ) = \frac{(1+\beta -\mu\gamma) I_2 - R - \beta R^{-1}}{(L - \mu)\gamma} x_t^\circ, \nonumber \\
    & \overset{(\text{By definition of } M)}{\Longleftrightarrow} &
    \exists \bar{f}\in\mathcal{F}_{0, 1} ~|~ \forall t \in \llbracket 0, K-1 \rrbracket, \nabla \bar{f}(x_t^\circ) = M x_t^\circ, \label{eq:f-normalization} \\
    & \overset{\begin{subarray}{c}
                \text{(By properties of the Fenchel transform}\\\text{\citep[Theorem 23.5]{rockafellar1997convex}})
            \end{subarray}}{\Longleftrightarrow} &
    \exists \bar{f}^*\in\mathcal{F}_{1, \infty} ~|~ \forall t \in \llbracket 0, K-1 \rrbracket, \nabla \bar{f}^*(M x_t^\circ) = x_t^\circ, \nonumber \\
    & \overset{\left(\hat{f}(x) = \bar{f}^*(x) - \tfrac{1}{2}\|x\|^2\right)}{\Longleftrightarrow} &
    \exists \hat{f}\in\mathcal{F}_{0, \infty} ~|~ \forall t \in \llbracket 0, K-1 \rrbracket, \nabla \hat{f}(M x_t^\circ) = (I-M) x_t^\circ. \nonumber
\end{eqnarray}
Applying the interpolation theorem~\citep[theorem 1]{taylor2017smooth}, the latest assertion is equivalent to
\begin{equation*}
    \exists (\hat{f}_t)_{t\in\range{0}{K-1}} ~|~ \forall i \neq j \in \llbracket 0, K-1 \rrbracket, \hat{f}_i \geq \hat{f}_j + \left<(I-M) x_j^\circ, M (x_i^\circ - x_j^\circ)\right>.
\end{equation*}
Note the inner product $\left<(I-M) x_j^\circ, M (x_i^\circ - x_j^\circ)\right>$ is also equal to $\left<(I-M) x_0^\circ, M (x_{i-j}^\circ - x_0^\circ)\right>$. Hence we can conclude $(\gb) \in \OKrouml $
\begin{eqnarray*}
    & \Longleftrightarrow & \exists (\hat{f}_t)_{t\in\range{0}{K-1}} ~|~ \forall i \neq j \in \llbracket 0, K-1 \rrbracket, \hat{f}_i \geq \hat{f}_j + \left<(I-M) x_0^\circ, M (x_{i-j}^\circ - x_0^\circ)\right>, \\
    & \Longleftrightarrow & \exists (\hat{f}_t)_{t\in\range{0}{K-1}} ~|~ \forall j\in \llbracket 0, K-1 \rrbracket, \forall \Delta\in \llbracket 1, K-1 \rrbracket, \hat{f}_{j + \Delta} \geq \hat{f}_j + \left<(I-M) x_0^\circ, M (x_{\Delta}^\circ - x_0^\circ)\right>.
\end{eqnarray*}
Summing up the latest over $j$ implies $\forall \Delta\in \llbracket 1, K-1 \rrbracket, 0 \geq \left<(I-M) x_0^\circ, M (x_{\Delta}^\circ - x_0^\circ)\right>.$

Reciprocally, this assertion implies the previous one with $\hat{f}_t = 0$ for all $t$.
Hence\footnote{Note that this result might be surprising (no function value appears) to the readers familiar with \emph{cyclic monotonicity} (see, e.g.,~\cite{rockafellar1997convex}). In our case, function values naturally disappear as we can arbitrarily set them to zero, which is different than simply not taking function values into account (see, e.g.,~discussion in~\cite[Remark 1]{taylor2017smooth}).},
\begin{center}
    \fbox{\parbox{\linewidth}{
        \begin{equation}
            (\gb) \in \OKrouml \Longleftrightarrow \forall \Delta\in \llbracket 1, K-1 \rrbracket, \left<M^T(I-M) x_0^\circ, x_{\Delta}^\circ - x_0^\circ\right> \leq 0.
            \label{eq:cycle_main_ineq_M}
        \end{equation}
    }}
\end{center}
This can be written
    \begin{align*}
        \left(\frac{1 + \beta - \gamma\mu - (1 + \beta)\cos{\theta_K}}{(L-\mu)\gamma} - \left(\frac{1 + \beta - \gamma\mu - (1 + \beta)\cos{\theta_K}}{(L-\mu)\gamma}\right)^2 - \left(\frac{(1 - \beta)\sin{\theta_K}}{(L-\mu)\gamma}\right)^2\right)\left(\cos{\Delta\theta_K}-1\right) \\
        + \frac{(1 - \beta)\sin{\theta_K}}{(L-\mu)\gamma}\sin{\Delta\theta_K} \leq 0,
    \end{align*}
    or dividing by $1 - \cos{\Delta\theta_K}$,
    \begin{align*}
        - \left(\frac{1 + \beta - \gamma\mu - (1 + \beta)\cos{\theta_K}}{(L-\mu)\gamma} - \left(\frac{1 + \beta - \gamma\mu - (1 + \beta)\cos{\theta_K}}{(L-\mu)\gamma}\right)^2 - \left(\frac{(1 - \beta)\sin{\theta_K}}{(L-\mu)\gamma}\right)^2\right) \\
        + \frac{(1 - \beta)\sin{\theta_K}}{(L-\mu)\gamma}\frac{\sin{\Delta\theta_K}}{1 - \cos{\Delta\theta_K}} \leq 0.
    \end{align*}
    This inequality must hold for any $\Delta\in\range{1}{K-1}$ and the LHS expression is maximized for $\Delta = 1$.  We conclude
    \begin{center}
    \fbox{\parbox{\linewidth}{
        \begin{equation*}
            (\gb) \in \OKrouml \Longleftrightarrow \left<M^T(I-M) x_0^\circ, x_{1}^\circ - x_0^\circ\right> \leq 0.
        \end{equation*}
    }}
\end{center}
 Or equivalently, $(\gb) \in \OKrouml$ if and only if
    \begin{align*}
        \left(\frac{1 + \beta - \gamma\mu - (1 + \beta)\cos{\theta_K}}{(L-\mu)\gamma} - \left(\frac{1 + \beta - \gamma\mu - (1 + \beta)\cos{\theta_K}}{(L-\mu)\gamma}\right)^2 - \left(\frac{(1 - \beta)\sin{\theta_K}}{(L-\mu)\gamma}\right)^2\right)\left(\cos{\theta_K}-1\right) \\
        + \frac{(1 - \beta)\sin{\theta_K}}{(L-\mu)\gamma}\sin{\theta_K} \leq 0.
    \end{align*}
    After multiplying the above inequality by $\kappa \tfrac{((L-\mu)\gamma)^2}{1 - \cos\theta_K}$ and rearranging the terms, we obtain equivalence with
    \begin{equation*}
        (\mu\gamma)^2 - 2\left[\beta - \cos\theta_K + \kappa(1 - \beta\cos\theta_K)\right](\mu\gamma) + 2\kappa(1 - \cos\theta_K)(1 + \beta^2 - 2\beta\cos\theta_K) \leq 0.
    \end{equation*}

\paragraph{Cycle on $\psi$:} 
Recall the function $\psi$ is defined as
\begin{equation*}
    \psi: x \mapsto \frac{L}{2}\|x\|^2 - \frac{L-\mu}{2}d(x, \conv \left\{Mx_t, t\in\range{0}{K-1}\right\})^2.
\end{equation*}
We define $\bar{\psi}$ as
\begin{equation*}
    \bar{\psi}: x \mapsto \frac{\psi(x) - \frac{\mu}{2}\|x\|^2}{L-\mu} = \frac{1}{2}\|x\|^2 - \frac{1}{2}d(x, \conv \left\{Mx_t, t\in\range{0}{K-1}\right\})^2,
\end{equation*}
and then we have
\begin{equation*}
    \nabla \bar{\psi}(x) = \mathrm{proj}_{\conv \left\{Mx_t, t\in\range{0}{K-1}\right\}}(x).
\end{equation*}
Using the same normalization as \eqref{eq:f-normalization}, we know that \HBgb$(f)$ cycles on $\TikCircle_K$ if and only if
$\forall t\in\range{0}{K-1}, \nabla \bar{\psi}(x_t) = Mx_t$, i.e.
\begin{equation*}
    \forall t\in\range{0}{K-1}, \mathrm{proj}_{\conv \left\{Mx_t, t\in\range{0}{K-1}\right\}}(x_t) = Mx_t.
\end{equation*}
The projection on a convex set $\mathrm{proj}_{\conv \left\{Mx_t, t\in\range{0}{K-1}\right\}}(x_t) = Mx_t$ can be characterized by the following set of inequalities: $\left\{\left<x_t - Mx_t, Mx_s - Mx_t\right> \leq 0, s\neq t\right\}$.
We conclude with \eqref{eq:cycle_main_ineq_M} that
\begin{center}
    \fbox{\parbox{\textwidth}{
        \begin{equation*}
            (\gb) \in \OKrouml \Longleftrightarrow \text{\HBgb}(\psi) \text{ cycles on } \TikCircle_K
        \end{equation*}
    }}
\end{center}

\paragraph{Construction of $\psi$:}
Using \eqref{eq:f-normalization}, we search for a function $\bar{f}\in\mathcal{F}_{0, 1} ~|~ \forall t \in \llbracket 0, K-1 \rrbracket, \nabla \bar{f}(x_t^\circ) = M x_t^\circ.$
The functions $h: x \mapsto \max_{t \in \llbracket 0, K-1 \rrbracket} \left< M x_t, x\right>$ is a natural convex function verifying $\forall t \in \llbracket 0, K-1 \rrbracket, \nabla \bar{f}(x_t^\circ) = M x_t^\circ$. However, the latter is not smooth.
We therefore compute $M_h$ its Moreau envelope with smoothing parameter 1, defined as
\begin{align*}
    M_h(x) & \eqdef \min_y f(y) + \frac{1}{2} \|x-y\|^2 \\
    & = \min_y \max_{t \in \llbracket 0, K-1 \rrbracket} \left< M x_t, y\right> + \frac{1}{2} \|x-y\|^2 \\
    & = \min_y \max_{(\lambda_t)_{t \in \llbracket 0, K-1 \rrbracket} \geq 0 ~ | ~ \sum_{t=0}^{K-1}\lambda_t = 1} \left< \sum_{t=0}^{K-1} \lambda_t M x_t, y\right> + \frac{1}{2} \|x-y\|^2 \\
    & = \max_{(\lambda_t)_{t \in \llbracket 0, K-1 \rrbracket} \geq 0 ~ | ~ \sum_{t=0}^{K-1}\lambda_t = 1} \left< \sum_{t=0}^{K-1} \lambda_t M x_t, x\right> - \frac{1}{2} \|\sum_{t=0}^{K-1} \lambda_t M x_t\|^2 \\
    & = \max_{z\in\conv\left\{Mx_t\right\}} \left< z, x\right> - \frac{1}{2} \|z\|^2 \\
    & = \frac{1}{2}\|x\|^2 - \frac{1}{2}d(x, \conv\left\{Mx_t\right\})^2.
\end{align*}
This function is $\bar{\psi}$ and by renormalization we obtain $\psi$.

\end{proof}

\subsection{Analysis of \texorpdfstring{$\Orouml$}{Omega cycle Fml}}\label{apx:convenient_expression_of_omega_cycle}

In this section, we aim at proving that the roots-of-unity cycling region $\Orouml$ can be written as $\Orouml = \left\{(\gb)\in\Oqml ~|~ \gamma\geq\gamma_{\min}(\beta,\ml)\right\}$ for some values of $\gamma_{\min}(\beta, \ml)$ to be determined.
\Cref{thm:analytical_ROU_region} states that, for any $(\gb)\in\Oqml$, $(\gb)\in\OKrouml$ if and only if 
\begin{align}\label{eq:Pgle0}
P_{\beta, K,\ml}(\gamma)\leq 0
\end{align}
with 
\begin{align}\label{eq:defPbml}
   P_{\beta, K,\ml}: \gamma \mapsto (\mu\gamma)^2 - 2\left[\beta - \cos\theta_K + \kappa(1 - \beta\cos\theta_K)\right](\mu\gamma) + 2\kappa(1 - \cos\theta_K)(1 + \beta^2 - 2\beta\cos\theta_K) .
\end{align}
The polynomial $P_{\beta, K,\ml}$ has two roots for $\beta$ larger than a value $\beta_{-}(K,\ml)$,  we thus use the following notations.
\begin{Not}\label{not:ab}
    For any $\ml$, $K \geq 2$, we denote 
    \begin{align*}
         \beta_{-}(K,\ml) & \eqdef \frac{\kappa\cos{\theta_K}^2 + (1-\kappa)^2\cos{\theta_K} - \kappa + (1-\kappa)(1-\cos{\theta_K})\sqrt{2\kappa(1 + \cos{\theta_K})}}{1 - 2\kappa + \kappa^2\cos{\theta_K}^2}, 
    \end{align*}
    For any $\beta\geq \beta_{-}(K,\ml)$, we denote
    \begin{align}
        A_K(\beta, \ml) & \eqdef \left[\beta - \cos\theta_K + \kappa(1 - \beta\cos\theta_K)\right], \label{eq:a_def}\\
        B_K(\beta, \ml) & \eqdef \sqrt{\left[\beta - \cos\theta_K + \kappa(1 - \beta\cos\theta_K)\right]^2 - 2\kappa(1 - \cos\theta_K)(1 + \beta^2 - 2\beta\cos\theta_K)} \label{eq:b_def},
    \end{align}
    \noindent Finally, we introduce the roots $(  \gamma_{-}(\beta, K,\ml), \gamma_{\max}(\beta,K,\ml) )$ of $P_{\beta, K,\ml}$:
    \begin{align*}
        \gamma_{-}(\beta, K,\ml) & \eqdef \frac{A_K(\beta,\ml) - B_K(\beta,\ml)}{\mu}, \\
        \gamma_{+}(\beta, K,\ml) & \eqdef \frac{A_K(\beta,\ml) + B_K(\beta,\ml)}{\mu}.
    \end{align*}
\end{Not}
We underline that we can obtain an alternative expression of $  \beta_{-}(K,\ml)  $, that intuitively provides an approximation of $\beta_{-}(K,\ml)$ as $\kappa\to 0$
\begin{Lemma} \label{lem:otherexprbeta}For any $K$, $\ml$, it holds that:
    \begin{align*}
         \beta_{-}(K,\ml) -  \cos{\theta_{K+1}} = & \sqrt{\kappa}(1 - \cos{\theta_{K}}) \frac{(1 + \cos{\theta_{K}})\sqrt{\kappa} + \sqrt{2(1 + \cos{\theta_{K}})}}{1 + \kappa\cos{\theta_{K}} + \sqrt{2\kappa(1 + \cos{\theta_{K}})}} \label{eq:beta_minus_cos} .
    \end{align*}
\end{Lemma}
\begin{proof} The proof consists in algebraic manipulations.
     \begin{align}
         \beta_{-}(K,\ml) - \cos{\theta_{K}}  &= \frac{\kappa\cos{\theta_{K}}^2 + (1-\kappa)^2\cos{\theta_{K}} - \kappa + (1-\kappa)(1-\cos{\theta_{K}})\sqrt{2\kappa(1 + \cos{\theta_{K}})}}{1 - 2\kappa + \kappa^2\cos{\theta_{K}}^2} \nonumber\\
        & - \frac{(1 - 2\kappa)\cos{\theta_{K}} + \kappa^2\cos{\theta_{K}}^3}{1 - 2\kappa + \kappa^2\cos{\theta_{K}}^2} \nonumber\\
        = & \frac{\kappa(\kappa\cos{\theta_{K}}-1)(1 - \cos^2{\theta_{K}}) + (1-\kappa)(1-\cos{\theta_{K}})\sqrt{2\kappa(1 + \cos{\theta_{K}})}}{1 - 2\kappa + \kappa^2\cos{\theta_{K}}^2} \nonumber\\
        = & (1 - \cos{\theta_{K}})\sqrt{\kappa (1 + \cos{\theta_{K})}} \cdot \frac{\sqrt{2}(1 - \kappa) + (\kappa\cos{\theta_{K}}-1)\sqrt{\kappa(1 + \cos{\theta_{K}})}}{1 - 2\kappa + \kappa^2\cos{\theta_{K}}^2} \nonumber\\
        = & \sqrt{\kappa}(1 - \cos{\theta_{K}})\sqrt{1 + \cos{\theta_{K}}} \nonumber\\
        & \times \frac{(1 + \kappa\cos{\theta_{K}} - \sqrt{2\kappa(1 + \cos{\theta_{K}})})(\sqrt{2} + \sqrt{\kappa(1 + \cos{\theta_{K}})})}{(1 + \kappa\cos{\theta_{K}} - \sqrt{2\kappa(1 + \cos{\theta_{K}})})(1 + \kappa\cos{\theta_{K}} + \sqrt{2\kappa(1 + \cos{\theta_{K}})})} \nonumber\\
        = & \sqrt{\kappa}(1 - \cos{\theta_{K}}) \frac{(1 + \cos{\theta_{K}})\sqrt{\kappa} + \sqrt{2(1 + \cos{\theta_{K}})}}{1 + \kappa\cos{\theta_{K}} + \sqrt{2\kappa(1 + \cos{\theta_{K}})}}  \nonumber .
    \end{align}
\end{proof}

\noindent
Using those notations, we can restate the condition \eqref{eq:Pgle0} as follows.
\begin{Fact}
$P_{\beta, K,\ml}(\gamma)\leq 0$  if and only if  $\beta \geq \beta_{-}(K, \ml)$ and  \begin{equation}
\label{eq:gamma_bound2}    
\gamma_{-}(\beta, K,\ml) \leq \gamma \leq \gamma_{+}(\beta, K, \mu, L), 
\end{equation}
i.e.:
\begin{align}
    \Orouml = & \Bigl\{ (\gb)\in\Oqml ~|~ \nonumber \\ & \exists K \geq 3 \text{ such that } \beta \geq \beta_{-}(K, \ml) \text{ and }   \gamma_{-}(\beta, K,\ml) \leq \gamma \leq \gamma_{+}(\beta, K, \mu, L)
    \Bigl\}.\label{eq:gamma_bound3}
\end{align}
\end{Fact}
In words, for any $\beta\geq 0$, the set of all $\gamma$ such that $(\gb)\in\Orouml$ is a union of intervals given by \eqref{eq:gamma_bound3}, non-necessarily connected. The next theorem states that this set of $\gamma$ is actually a single interval as soon as $\kappa$ is sufficiently small.
\begin{Th}[Analytical form of Roots-of-unity cycle region]
\label{thm:analytical_ROU_region_bis}
If $\kappa\leq \left(\frac{3 - \sqrt{5}}{4}\right)^2$, the roots-of-unity cycling region is:
\begin{equation*}
    \Orouml = \Bigl\{ (\gb)\in\Oqml ~|~ \exists K \geq 3 \text{ such that } \beta \geq \beta_{-}(K,\ml) \text{ and } \gamma \geq \gamma_{-}(\beta, K, 
    \ml)
    \Bigl\}.
\end{equation*}
\end{Th}
This theorem means that we can ignore, the condition $\gamma \leq \gamma_{+}(\beta, K, \mu, L)$ in the parametric description of $\Orouml$ in \eqref{eq:gamma_bound2}. Note that $\OKrouml \neq \bigl\{ (\gb)\in\Oqml ~|~ \beta \geq \beta_{-}(K,\ml) \text{ and } \gamma \geq \gamma_{-}(\beta, K, \ml) \bigl\}$: the theorem states that the union (over $K\geq  2$) of those sets can be written without an upper bound on $\gamma$, not each set individually. That is, when $(\gb) \notin \OKrouml$ because $\gamma \geq \gamma_{+}(\beta, K, \ml)$, $(\gb)\in\Omega_{K'\text{-}\circ\text{-}\cycle}(\Fml)$ for some $K'\leq K$.
\begin{Rem}[$K = 2$]
    \label{rem:k2}
    Plugging $K=2$ into \eqref{eq:gamma_bound2} leads to $\gamma \in \left[\frac{2(1+\beta)}{L}, \frac{2(1+\beta)}{\mu}\right]$ whose intersection with $\Oqml$ is empty. Hence, $\Omega_{2\text{-}\circ\text{-}\cycle}(\Fml) = \emptyset$, which explains why \Cref{thm:analytical_ROU_region_bis} does not consider cycles of length 2.
\end{Rem}
In order to prove \Cref{thm:analytical_ROU_region_bis}, we first state \Cref{lem:AB_sort}.
\begin{Lemma}\label{lem:AB_sort}
    We assume $\kappa\leq \left(\frac{3 - \sqrt{5}}{4}\right)^2$.
    For any $K \geq 2$, and any $\beta\geq \beta_{-}(K+1,\ml)$, we have
    \begin{equation*}
        \gamma_{-}(\beta, K, \ml) \le \gamma_{+}(\beta, K+1, \ml).
    \end{equation*} 
\end{Lemma}
First, we show that thanks to  \Cref{lem:AB_sort}, we can establish \Cref{thm:analytical_ROU_region_bis}.

\begin{proof}[\Cref{thm:analytical_ROU_region_bis}]
We denote $\OmaxKrouml = \bigcup_{3\le \bar K\le K}\OKrouml$ the set of $(\gb)\in\Oqml$ such that there exists $\bar{K} \leq K$ with $(\gb)\in\OKrouml$. We prove by induction over $K$ that 
\begin{equation}
    \OmaxKrouml = \Bigl\{ (\gb)\in\Oqml ~|~ \exists \bar{K} \leq K \text{ such that } \beta \geq \beta_{-}(\bar{K},\ml) \text{ and } \gamma \geq \gamma_{-}(\beta, \bar{K}, \ml)
\Bigl\}. \label{eq:inductionb3}
\end{equation}

\noindent
\textbf{Initialization $(K=3)$:} From \Cref{lem:AB_sort}, $ \gamma_{+}(\beta, 3, \ml) \geq \gamma_{-}(\beta, 2, \ml)  $, and from~\Cref{rem:k2},  $ \gamma_{-}(\beta, 2, \ml) = \frac{2(1+\beta)}{L}$. Thus
\[\Omega_{\le 3\text{-}\circ\text{-}\cycle}(\Fml) = \Omega_{3\text{-}\circ\text{-}\cycle}(\Fml) = \left\{(\gb)\in\Oqml ~|~ \beta \geq \beta_{-}(3,\ml) \text{ and } \mu\gamma \geq  \mu \gamma_{-}(\beta, 3, \ml)\right\}.\]
Indeed, the additional assumption that $\gamma \leq \gamma_{+}(\beta, 3, \ml) $ is useless as $\gamma_{+}(\beta, 3, \ml)$ is larger than $\frac{2(1+\beta)}{L}$, which corresponds to  the right-hand side border of $\Ocvml$.\\ 
 
\noindent
\textbf{Induction:} Let us assume that \eqref{eq:inductionb3} holds for some $K$. We prove it still holds for $K+1$.
    Indeed, $(\gb) \in \Omega_{\leq K+1\text{-}\circ\text{-}\cycle}(\Fml)$ if and only if $(\gb) \in \Omega_{\leq K\text{-}\circ\text{-}\cycle}(\Fml)$ or $(\gb) \in \Omega_{K+1\text{-}\circ\text{-}\cycle}(\Fml)$, i.e.~if, by our induction hypothesis $\gamma\geq \gamma_{-}(\beta, \bar{K},\ml)$ for some $\bar{K} \leq K$, or $\gamma_{-}(\beta, K+1,\ml) \leq \gamma \leq \gamma_{+}(\beta, K+1, \mu, L)$.

    Then by \Cref{lem:AB_sort},  if $\gamma \geq \gamma_{+}(\beta, K+1, \mu, L)$, we have  $\gamma \geq \gamma_{-}(\beta, K,\ml)$, thus $(\gb) \in \Omega_{\leq K\text{-}\circ\text{-}\cycle}(\Fml)$. Finally, $(\gb) \in \Omega_{\leq K+1\text{-}\circ\text{-}\cycle}(\Fml)$ if and only if $\gamma\geq \gamma_{-}(\beta, \bar{K},\ml)$ for some $\bar{K} \leq K$, or $\gamma \geq \gamma_{-}(\beta, K+1,\ml)$, i.e.~if and only if $\gamma\geq \gamma_{-}(\beta, \bar{K},\ml)$ for some $\bar{K} \leq K+1$, which concludes the proof.
\end{proof}
\vspace{1em}

\noindent
We now prove \Cref{lem:AB_sort}.
\begin{proof}[\Cref{lem:AB_sort}]
    Since for any $K$, $\gamma_{-}(\beta, K, \ml)$ is defined as $\frac{A_K(\beta,\ml) - B_K(\beta,\ml)}{\mu}$, we show the equivalent formulation
    \begin{equation}
        A_K(\beta,\ml) - B_K(\beta,\ml) \leq A_{K+1}(\beta,\ml) + B_{K+1}(\beta,\ml). \label{eq:AB_sort}
    \end{equation}
First, we rely  on the following  sequence of implication, that shows that proving \eqref{eq:AB_relaxation_pp} hereafter is sufficient to establish \eqref{eq:AB_sort}.
    \begin{align}
        & \left(A_K(\beta,\ml) - A_{K+1}(\beta,\ml)\right)^2 \leq B_K(\beta,\ml)^2 - B_{K+1}(\beta,\ml)^2 \label{eq:AB_relaxation_pp} \\
        & \quad \overset{B_{K+1}(\beta,\ml)^2 \geq 0}{\Longrightarrow}
        \left(A_K(\beta,\ml) - A_{K+1}(\beta,\ml)\right)^2 \leq B_K(\beta,\ml)^2 \label{eq:AB_relaxation} \\
        & \quad  \overset{B_K(\beta,\ml) \geq 0}{\Longrightarrow}
        A_K(\beta,\ml) - A_{K+1}(\beta,\ml) \leq B_K(\beta,\ml) \nonumber \\
        & \quad \overset{\text{Reordering terms}}{\Longleftrightarrow}
        A_K(\beta,\ml) - B_K(\beta,\ml) \leq A_{K+1}(\beta,\ml) \nonumber \\
        & \quad \overset{B_{K+1}(\beta,\ml) \geq 0}{\Longrightarrow}
        A_K(\beta,\ml) - B_K(\beta,\ml) \leq A_{K+1}(\beta,\ml) + B_{K+1}(\beta,\ml). \nonumber
    \end{align}
    \Cref{eq:AB_relaxation} has the advantage of isolating $B_K(\beta,\ml)$ so that we get rid of the square-root appearing in its definition. \Cref{eq:AB_relaxation_pp} has 
    the additional advantage of getting rid of the terms that are independent of the cosines and of making a factor $\cos{\theta_{K+1}} - \cos{\theta_{K}}$ appear on both sides of the inequality.  Indeed, first we have
    \begin{align}
        A_K(\beta,\ml) - A_{K+1}(\beta,\ml)
        & \overset{\eqref{eq:a_def}}{=} \left[\beta - \cos\theta_K + \kappa(1 - \beta\cos\theta_K)\right] - \left[\beta - \cos\theta_{K+1} + \kappa(1 - \beta\cos\theta_{K+1})\right] \nonumber \\
        & = (1 + \beta\kappa)(\cos{\theta_{K+1}} - \cos{\theta_{K}}) \nonumber \\
        (A_K(\beta,\ml) - A_{K+1}(\beta,\ml))^2
        & = (1 + \beta\kappa)^2(\cos{\theta_{K+1}} - \cos{\theta_{K}})^2. \label{eq:diffdesA}
    \end{align}
    Also, $B_K(\beta,\ml)^2$ can be simplified, expanding and reordering terms, as 
    \begin{align*}
        B_K(\beta,\ml)^2
        \overset{\eqref{eq:b_def}}{=} &
        \left[\beta - \cos\theta_K + \kappa(1 - \beta\cos\theta_K)\right]^2 - 2\kappa(1 - \cos\theta_K)(1 + \beta^2 - 2\beta\cos\theta_K) \\
        = &
        \left[ \beta + \kappa - (1 + \beta\kappa) \cos\theta_K \right]^2 - 2\kappa(1 - \cos\theta_K)(1 + \beta^2 - 2\beta\cos\theta_K) \\
        = &
        (\beta + \kappa)^2 - 2(\beta + \kappa)(1 + \beta\kappa) \cos\theta_K + (1 + \beta\kappa)^2 \cos^2\theta_K \\
        & - 2\kappa (1 + \beta^2 - 2\beta\cos\theta_K - (1 + \beta^2)\cos\theta_K + 2\beta \cos^2\theta_K) \\
        = & 
        \left[(\beta + \kappa)^2 - 2\kappa(1 + \beta^2)\right]  -2 \left[(\beta + \kappa)(1 + \beta\kappa) - \kappa(1+\beta)^2\right]\cos\theta_K  \\
        &+ \left[(1 + \beta\kappa)^2 - 4\beta\kappa \right] \cos^2\theta_K \\
        = & 
       \underbrace{ \left[(\beta + \kappa)^2 - 2\kappa(1 + \beta^2)\right] }_{\text{independent of $K$}} -2 \beta (1-\kappa)^2 \cos\theta_K + (1 - \beta\kappa)^2 \cos^2\theta_K.
    \end{align*}
    While inequality \eqref{eq:AB_relaxation} thus writes
    \begin{equation*}
        (1 + \beta\kappa)^2(\cos{\theta_{K+1}} - \cos{\theta_{K}})^2 \leq \left[(\beta + \kappa)^2 - 2\kappa(1 + \beta^2)\right] -2 \beta (1-\kappa)^2 \cos\theta_K + (1 - \beta\kappa)^2 \cos^2\theta_K,
    \end{equation*}
    and is thus not easy to  handle, $B_K(\beta,\ml)^2 - B_{K+1}(\beta,\ml)^2$ simplifies as
    \begin{align}
        B_K(\beta,\ml)^2 - B_{K+1}(\beta,\ml)^2 =
        & -2 \beta (1-\kappa)^2 (\cos\theta_K - \cos\theta_{K+1})  + (1 - \beta\kappa)^2 (\cos^2\theta_K - \cos^2\theta_{K+1}), \label{eq:diffdesB}
    \end{align}
    and using \eqref{eq:diffdesA} and \eqref{eq:diffdesB},  inequality \eqref{eq:AB_relaxation_pp} thus writes
    \begin{align*}
        (1 + \beta\kappa)^2(\cos{\theta_{K+1}} - \cos{\theta_{K}})^2 & \leq 2 \beta (1-\kappa)^2 (\cos{\theta_{K+1}} - \cos{\theta_{K}}) - (1 - \beta\kappa)^2 (\cos^2\theta_{K+1} - \cos^2\theta_K) \\
        & = \left[2 \beta (1-\kappa)^2  - (1 - \beta\kappa)^2 (\cos\theta_{K+1} + \cos\theta_K)\right] (\cos{\theta_{K+1}} - \cos{\theta_{K}}).
    \end{align*}
    Dividing  by the positive term $\cos{\theta_{K+1}} - \cos{\theta_{K}}$\footnote{This simplification was the motivation to prove the relaxation \eqref{eq:AB_relaxation_pp} instead of \eqref{eq:AB_relaxation}. Interestingly, this relaxation is tight when $B_{K+1}(\beta,\ml)=0$, i.e.~for $\beta = \beta_{-}(K+1, \ml)$. As we will see later,  this value of $\beta$ is the only one that actually matters (if $K\neq 3$).}, we get that \eqref{eq:AB_relaxation_pp} is equivalent to
    \begin{equation*}
        (1 + \beta\kappa)^2(\cos{\theta_{K+1}} - \cos{\theta_{K}}) \leq \left[2 \beta (1-\kappa)^2  - (1 - \beta\kappa)^2 (\cos\theta_{K+1} + \cos\theta_K)\right].
    \end{equation*}
    By ordering terms, this is also equivalent to
    \begin{equation*}
        (1 + \beta\kappa)^2(\cos{\theta_{K+1}} - \cos{\theta_{K}}) + (1 - \beta\kappa)^2 (\cos\theta_{K+1} + \cos\theta_K) \leq 2 \beta (1-\kappa)^2,
    \end{equation*}
    or again, dividing by $4\beta\kappa$,
    \begin{equation}
        \frac{1}{2}\left(\frac{1}{\beta\kappa} + \beta\kappa\right)\cos{\theta_{K+1}} \leq \frac{(1-\kappa)^2}{2\kappa} + \cos{\theta_K}. \label{eq:super_reduced}
    \end{equation}
    Note that we need this last equation to hold for all $\beta \geq \beta_{-}(K+1,\ml)$.
    Moreover, while the RHS does not depend on $\beta$, the LHS increases or decreases with respect to $\beta$, according to the sign of $\cos{\theta_{K+1}}$. We thus have to distinguish several  cases.
    \begin{itemize}[leftmargin=* ]
        \item When $K=2$, $\cos{\theta_{K+1}} = \cos{\theta_{3}} = -1/2 <0$. Then the LHS increases and it is sufficient to verify \eqref{eq:super_reduced} for $\beta=1$. We have: if $-\frac{1}{4}\left(\frac{1}{\kappa} + \kappa\right) \leq \frac{(1-\kappa)^2}{2\kappa} - 1$, then \eqref{eq:AB_sort} holds for $K=2$. A simple technical computation gives that this condition holds if and only if $\kappa \leq \frac{4-\sqrt{7}}{3}$.
        \item When $K=3$, $\cos{\theta_{K+1}} = \cos{\theta_{4}} = 0$. Then the LHS does not depend on $\beta$ neither. \Cref{eq:super_reduced} is independent of $\beta$ and is written $0 \leq \frac{(1-\kappa)^2}{2\kappa} -\frac{1}{2}$. We conclude that, if $\kappa \leq \frac{3-\sqrt{5}}{2}$, then \eqref{lem:AB_sort} holds for $K=3$.
        \item When $K\geq 4$, $\cos{\theta_{K+1}} \geq \cos{\theta_{5}} > 0$, hence the LHS decreases with respect to $\beta$. It is sufficient to verify~\eqref{eq:super_reduced} for $\beta = \beta_{-}(K+1,\ml)$, or even for any  value smaller  than $\beta_{-}(K+1,\ml)$  to prove the lemma. 
    \end{itemize}
    We use the following lower bound on $\beta_{-}(K+1,\ml)$, obtained from  \Cref{lem:otherexprbeta}
      \begin{align}
        \beta_{-}(K+1,\ml) -  \cos{\theta_{K+1}}  
        \overset{\text{\ref{lem:otherexprbeta}}}{=} & \sqrt{\kappa}(1 - \cos{\theta_{K+1}}) \frac{(1 + \cos{\theta_{K+1}})\sqrt{\kappa} + \sqrt{2(1 + \cos{\theta_{K+1}})}}{1 + \kappa\cos{\theta_{K+1}} + \sqrt{2\kappa(1 + \cos{\theta_{K+1}})}}\nonumber \\
        \geq & \sqrt{\kappa}(1 - \cos{\theta_{K+1}})\nonumber. \\
        = & \cos{\theta_{K+1}}(1 + \sqrt{\kappa}\xi_{K+1}), \nonumber
    \end{align}
  with $\xi_{K+1} \eqdef \frac{1}{\cos{\theta_{K+1}}}-1$. Overall, for $K\geq 4$, \eqref{eq:super_reduced} is valid  for all $\beta \geq \beta_{-}(K+1,\ml)$ if it is valid at $\cos{\theta_{K+1}}(1 + \sqrt{\kappa}\xi_{K+1})$. Plugging this value into \eqref{eq:super_reduced} it is thus sufficient to prove that
    \begin{equation*}
        \frac{1}{2}\left(\frac{1}{\kappa(1 + \sqrt{\kappa}\xi_{K+1})} + \kappa\cos^2{\theta_{K+1}}(1 + \sqrt{\kappa}\xi_{K+1})\right) \leq \frac{(1-\kappa)^2}{2\kappa} + \cos{\theta_K}.
    \end{equation*}
    Multiplying by $2\kappa$ and reordering terms, the latter is equivalent to
    \begin{equation}
        \frac{1}{1 + \sqrt{\kappa}\xi_{K+1}} - (1-\kappa)^2 \leq 2\kappa\cos{\theta_K} - \kappa^2 \underbrace{\cos^2{\theta_{K+1}}}_{\leq \cos{\theta_{K+1}}} \underbrace{(1 + \sqrt{\kappa}\xi_{K+1})}_{\leq 1 + \xi_{K+1}}. \label{eq:auxaux}
    \end{equation}
    First, 
    as for any $\tau \geq 0$, $\frac{1}{1 + \tau} \leq 1 - \tau+ \tau^2$, the LHS of \eqref{eq:auxaux} is upper bounded by $1 - \sqrt{\kappa}\xi_{K+1} + \kappa \xi_{K+1}^2 - (1-\kappa)^2 $. Second, the RHS of \eqref{eq:auxaux} is lower bounded by $2\kappa\cos{\theta_K} - \kappa^2$.  It  is  thus sufficient to prove
    \begin{equation*}
        1 - \sqrt{\kappa}\xi_{K+1} + \kappa \xi_{K+1}^2 - (1 - 2\kappa + \kappa^2) \leq 2\kappa\cos{\theta_K} - \kappa^2.
    \end{equation*}
    Simplifying and reordering terms, this is equivalent to
    \begin{equation*}
         \sqrt{\kappa}\xi_{K+1} \geq \kappa ( \xi_{K+1}^2 + 2 - 2\cos{\theta_K}),
    \end{equation*}
    or again
    \begin{equation*}
         \frac{1}{\sqrt{\kappa}} \geq \left( \xi_{K+1} + \frac{2(1 - \cos{\theta_K})}{\xi_{K+1}}\right) = \left( \frac{1}{\cos{\theta_{K+1}}}-1 + \frac{2\cos{\theta_{K+1}}(1 - \cos{\theta_K})}{(1 - \cos{\theta_{K+1}})}\right).
    \end{equation*}
    Finally,
    \begin{eqnarray*}
        \frac{1}{\cos{\theta_{K+1}}}-1 + \frac{2\cos{\theta_{K+1}}(1 - \cos{\theta_K})}{(1 - \cos{\theta_{K+1}})} & \overset{\text{\Cref{lem:techn1}}}{\leq} & \frac{1}{\cos{\theta_{K+1}}} - 1 + 2 \times 1 \times \frac{3}{2} = \frac{1}{\cos{\theta_{K+1}}} + 2 \\
        & \overset{K \geq 4}{\leq} & \frac{1}{\cos{\theta_{5}}} + 2 = \frac{1}{\cos{\frac{2\pi}{5}}} + 2 \\
        & = & 3 + \sqrt{5}.
    \end{eqnarray*}
    Ultimately,  $\kappa \leq \left(\frac{1}{3 + \sqrt{5}}\right)^2 = \left(\frac{3 - \sqrt{5}}{4}\right)^2$ is a sufficient condition for \Cref{eq:AB_sort} to hold for any $K \geq 4$.
    As a summary, we proved the following sufficient conditions so that \eqref{eq:AB_sort} holds:
    \begin{itemize}
        \item For $K=2$, \eqref{eq:AB_sort} holds as soon as $\kappa\leq \frac{4 - \sqrt{7}}{3}$.
        \item For $K=3$, \eqref{eq:AB_sort} holds as soon as $\kappa\leq \frac{3 - \sqrt{5}}{2}$.
        \item For $K\geq 4$, \eqref{eq:AB_sort} holds as soon as $\kappa\leq \left(\frac{3 - \sqrt{5}}{4}\right)^2$.
    \end{itemize}
    Among those 3 values, $\left(\frac{3 - \sqrt{5}}{4}\right)^2$ is the smallest. Hence, we conclude that \eqref{eq:AB_sort} holds for any $K\geq 2$ if $\kappa\leq \left(\frac{3 - \sqrt{5}}{4}\right)^2$.
    
\end{proof}

\subsection{Proof of \texorpdfstring{\Cref{thm:incompatibility}}{Theorem 3.6}}
\label{apx:main-result}
In order to prove \Cref{thm:incompatibility}, we first establish technical results in \Cref{subsubapp:tech}, then prove the result in \Cref{subsubapp:proof}

\subsubsection{Technical lemmas}\label{subsubapp:tech}
We start by proving the following technical lemmas.
\begin{Lemma}\label{lem:techn1}
 For any integer $K \geq 2, ~ 1 \leq \frac{1 - \cos\theta_K}{1 - \cos\theta_{K+1}} \leq \frac{3}{2}$.
\end{Lemma}

\begin{proof}
    First of all, let's note that
    \begin{equation*}
        \frac{1 - \cos\frac{2\pi}{K}}{1 - \cos\frac{2\pi}{K+1}} = \frac{\sin^2\frac{\pi}{K}}{\sin^2\frac{\pi}{K+1}}.
    \end{equation*}
    Since $\sin$ is positive and increasing on $[0, \pi/2]$, $\sin^2\frac{\pi}{K} \geq \sin^2\frac{\pi}{K+1}$, hence for all $K \geq 2,$  $
         ~ 1 \leq \frac{1 - \cos\frac{2\pi}{K}}{1 - \cos\frac{2\pi}{K+1}}$.
    Note moreover this bound is tight as $\frac{1 - \cos\frac{2\pi}{K}}{1 - \cos\frac{2\pi}{K+1}} \underset{K\rightarrow \infty}{\rightarrow} 1$.
It remains to show that for all integer $K \geq 2, \frac{\sin\frac{\pi}{K}}{\sin\frac{\pi}{K+1}} \leq \sqrt{\frac{3}{2}}$. We study the four following cases:
    \begin{itemize}
        \item for $K=2$, $\frac{\sin\frac{\pi}{K}}{\sin\frac{\pi}{K+1}} = \frac{1}{\sqrt{3}/2} < \sqrt{\frac{3}{2}}$,
        \item for $K=3$, $\frac{\sin\frac{\pi}{K}}{\sin\frac{\pi}{K+1}} = \frac{\sqrt{3}/2}{\sqrt{2}/2} = \sqrt{\frac{3}{2}}$,
        \item for $K=4$, $\frac{\sin\frac{\pi}{K}}{\sin\frac{\pi}{K+1}} = \frac{\sqrt{2}/2}{\sqrt{10-2\sqrt{5}}/4} < \sqrt{\frac{3}{2}}$,
        \item for $K\geq 5$, 
        \begin{align*}
            \frac{\sin\frac{\pi}{K}}{\sin\frac{\pi}{K+1}} & ~ = \frac{\sin(\frac{\pi}{K+1} + \frac{\pi}{K(K+1)})}{\sin\frac{\pi}{K+1}} = \frac{\sin\frac{\pi}{K+1} \cos\frac{\pi}{K(K+1)} + \cos\frac{\pi}{K+1} \sin\frac{\pi}{K(K+1)}}{\sin\frac{\pi}{K+1}} \\
            & ~ = \cos\frac{\pi}{K(K+1)} + \rm{cotan}\frac{\pi}{K+1} \sin\frac{\pi}{K(K+1)} \\
            & ~ \leq 1 + \frac{K+1}{\pi}\frac{\pi}{K(K+1)}  = 1 + \frac{1}{K} \leq \frac{6}{5} < \sqrt{\frac{3}{2}}.
        \end{align*}
        This shows that 
        \begin{equation*}
            \forall \text{ integer } K \geq 2, ~ \frac{1 - \cos\frac{2\pi}{K}}{1 - \cos\frac{2\pi}{K+1}} \leq \frac{3}{2}.
        \end{equation*}
        Note furthermore that this bound is also tight as reached for $K=3$.
    \end{itemize}
\end{proof}  \vspace{1em} 

\noindent
Next, we establish a second technical result.
\begin{Lemma}\label{lem:techn2}
    For any $ \beta \in [0, 1],$ there exists $ K\geq 2 $ such that $ \frac{2}{3} \leq \frac{\beta - \cos\theta{K}}{1-\beta} \leq \frac{3}{2}$.
\end{Lemma}
\begin{proof}
    Let $\beta \in [0, 1]$.
    And let $Z \geq 2$ a real number such that $\beta = \frac{1 + \cos\frac{2\pi}{Z}}{2}$.
    We note that $\frac{1 + \cos\frac{2\pi}{\left\lfloor{Z}\right\rfloor}}{2} \leq \beta < \frac{1 + \cos\frac{2\pi}{\left\lfloor{Z}\right\rfloor+1}}{2}$. Besides, from \Cref{lem:techn1},
    $1 \leq \frac{1 - \cos\frac{2\pi}{\left\lfloor{Z}\right\rfloor}}{1 - \cos\frac{2\pi}{\left\lfloor{Z}+1\right\rfloor}} \leq \frac{3}{2}$, which implies 
    $$1 \leq \frac{1 - \cos\frac{2\pi}{\left\lfloor{Z}\right\rfloor}}{1 - \cos\frac{2\pi}{Z}} \leq \frac{5}{4} 
    \quad\text{ or }\quad
    1 \leq \frac{1 - \cos\frac{2\pi}{Z}}{1 - \cos\frac{2\pi}{\left\lfloor{Z}+1\right\rfloor}} \leq \frac{6}{5}.$$
    In the first case, we define $K = \left\lfloor{Z}\right\rfloor$ while in the second case, we define $K = \left\lfloor{Z}+1\right\rfloor$. In any case, $\frac{5}{6} \leq \frac{1 - \cos\frac{2\pi}{K}}{1 - \cos\frac{2\pi}{Z}} \leq \frac{5}{4}$.
    Moreover, we have $\frac{\beta - \cos\frac{2\pi}{K}}{1-\beta} = -1 + 2 \frac{1 - \cos\frac{2\pi}{K}}{1 - \cos\frac{2\pi}{Z}}$, and we conclude
    $\frac{2}{3} \leq \frac{\beta - \cos\frac{2\pi}{K}}{1-\beta} \leq \frac{3}{2}$.
\end{proof}
\vspace{1em}

\noindent
Third, we establish a final technical result.
\begin{Lemma}\label{lem:techn3}
    Assume that $\kappa \leq \frac{1}{16}$.
    For any $ \beta\in [0, 1], $ for any $ K\geq 2 $ such that $ \frac{2}{3} \leq \frac{\beta - \cos\theta_{K}}{1-\beta} \leq \frac{3}{2}, $ we have $ \beta \geq \beta_{-}(K,\ml)$.
\end{Lemma}

\begin{proof}
    The proof consists in proving that
    \begin{equation*}
        \frac{1 - \beta_{-}(K,\ml)}{1 - \cos{\theta_K}} \geq \frac{3}{5}.
    \end{equation*}
    Then, since $\frac{\beta - \cos{\theta_K}}{1-\beta}\geq \frac{2}{3}$, we have $\frac{1 - \cos{\theta_K}}{1-\beta}\geq \frac{5}{3}$, and finally $\frac{1-\beta}{1 - \cos{\theta_K}} \leq \frac{3}{5} \leq \frac{1 - \beta_{-}(K,\ml)}{1 - \cos{\theta_K}}$, concluding that $\beta \geq \beta_{-}(K,\ml)$.

    By \Cref{lem:otherexprbeta}, we have
    \begin{equation*}
        \frac{\beta_{-}(K,\ml) - \cos{\theta_{K}}}{1 - \cos{\theta_{K}}} = \frac{\kappa + \kappa\cos{\theta_{K}} + \sqrt{2\kappa(1 + \cos{\theta_{K}})}}{1 + \kappa\cos{\theta_{K}} + \sqrt{2\kappa(1 + \cos{\theta_{K}})}}.
    \end{equation*}
    Applying the mapping $x\mapsto 1-x$ to the previous equation, we get
    \begin{equation*}
        \frac{1-\beta_{-}(K,\ml)}{1 - \cos{\theta_{K}}} = \frac{1-\kappa}{1 + \kappa\cos{\theta_{K}} + \sqrt{2\kappa(1 + \cos{\theta_{K}})}}.
    \end{equation*}
    We now need to prove that whatever $\theta_K$ is, $\frac{1-\kappa}{1 + \kappa\cos{\theta_{K}} + \sqrt{2\kappa(1 + \cos{\theta_{K}})}} \geq \frac{3}{5}$.
    We note that the LHS converges to 1 when $\kappa$ goes to 0, hence we know that this inequality holds for $\kappa$ small enough. We make the precise computation below.
    \begin{align*}
        \frac{1-\kappa}{1 + \kappa\cos{\theta_{K}} + \sqrt{2\kappa(1 + \cos{\theta_{K}})}} \geq \frac{3}{5} 
        & \Longleftrightarrow &
        3(1 + \kappa\cos{\theta_{K}} + \sqrt{2\kappa(1 + \cos{\theta_{K}})}) \leq 5(1-\kappa) \\
        & \Longleftrightarrow &
        3 \sqrt{2\kappa(1 + \cos{\theta_{K}})} + \kappa (5 + 3\cos{\theta_K}) \leq 2 \\
        & \Longleftarrow &
        6 \sqrt{\kappa} + 8\kappa \leq 2 \\
        & \Longleftrightarrow &
        1 - 3\sqrt{\kappa} - 4\kappa \geq 0 \\
        & \Longleftrightarrow &
        (1 - 4\sqrt{\kappa})(1 + \sqrt{\kappa}) \geq 0 \\
        & \Longleftrightarrow &
        \sqrt{\kappa} \leq \frac{1}{4}.
    \end{align*}
    Hence, the desired result.
    
\end{proof}

\vspace{1em} \noindent
\noindent
\subsubsection{Proof of \texorpdfstring{\Cref{thm:incompatibility}}{Theorem 3.6}}
\label{subsubapp:proof}
Finally, we prove \Cref{thm:incompatibility} :
\incompatibility*

\begin{proof}
    First, we assume that $\kappa\leq \left(\frac{3 - \sqrt{5}}{4}\right)^2$.
    Note that in the opposite case, $\sqrt{\kappa} \leq (3 + \sqrt{5})\kappa \leq \frac{50}{3}\kappa$, hence the result would still hold.
    
    Let $(\gb)\in (\Orouml)^c = (\bigcup_{K=2}^\infty \OKrouml)^c = \bigcap_{K=2}^\infty (\OKrouml)^c  $. 
    By \Cref{thm:analytical_ROU_region_bis} (since $\kappa\leq\left(\frac{3 - \sqrt{5}}{4}\right)^2$),  for all $K\geq 2$, if $\beta \geq \beta_{-}(K,\ml)$ (see \Cref{not:ab}),
    \begin{align*}
        \mu\gamma & \leq \mu \gamma_{-}(\beta,K,\ml)\\
        & =
        \left[\beta - \cos{\theta_{K}} + \kappa(1 - \beta\cos{\theta_{K}})\right]\left( 1 - \sqrt{1 - \frac{2\kappa(1 - \cos{\theta_{K}})(1 + \beta^2 - 2\beta\cos{\theta_{K}})}{\left[\beta - \cos{\theta_{K}} + \kappa(1 - \beta\cos{\theta_{K}})\right]^2}}\right) \\
        & \leq
        \left[\beta - \cos{\theta_{K}} + \kappa(1 - \beta\cos{\theta_{K}})\right]\left( 1 - \left(1 - \frac{2\kappa(1 - \cos{\theta_{K}})(1 + \beta^2 - 2\beta\cos{\theta_{K}})}{\left[\beta - \cos{\theta_{K}} + \kappa(1 - \beta\cos{\theta_{K}})\right]^2}\right)\right) \\
        & =
        \left(\frac{2\kappa(1 - \cos{\theta_{K}})(1 + \beta^2 - 2\beta\cos{\theta_{K}})}{\beta - \cos{\theta_{K}} + \kappa(1 - \beta\cos{\theta_{K}})}\right) \\
        & \leq
        \left(\frac{2\kappa(1 - \cos{\theta_{K}})(1 + \beta^2 - 2\beta\cos{\theta_{K}})}{\beta - \cos{\theta_{K}}}\right).
    \end{align*}
    Finally,
    \begin{equation*}
        \mu\gamma \leq \min_{K\geq 2} \left(\frac{2\kappa(1 - \cos{\theta_{K}})(1 + \beta^2 - 2\beta\cos{\theta_{K}})}{\beta - \cos{\theta_{K}}}\right).
    \end{equation*}
    We introduce $C_\beta \eqdef \frac{\beta - \cos \tfrac{2\pi}{K_{\beta}}}{1 - \beta}$ with $K$ chosen so that \Cref{lem:techn2} holds.
    First we know by \Cref{lem:techn3} (since $\kappa\leq\left(\frac{3 - \sqrt{5}}{4}\right)^2\leq\frac{1}{16}$) that $\beta\geq\beta_{-}(K_{\beta}, \ml)$ and then that the previous calculus is valid.
    Second, we have
    \begin{align}
        \mu\gamma & \leq \left(2\left(2\beta C_\beta + 1 + 3\beta + \tfrac{1+\beta}{C_\beta}\right)\right)\kappa(1 - \beta) \nonumber \\
        & \leq 4\left(C_\beta + \frac{1}{C_\beta} + 2\right)\kappa(1 - \beta) \nonumber\\
        & \leq 4\left(\frac{3}{2} + \frac{2}{3} + 2\right)\kappa(1 - \beta) \nonumber \\
        & = \frac{50}{3}\kappa(1 - \beta) \label{eq:toto}
    \end{align}
    which proves that $ \mu\gamma \leq \frac{50}{3}\kappa(1 - \beta)$. Set $C=50/3$.
    By contradiction, we assume that $\rho < \frac{1 - C\kappa}{1 + C\kappa}$, and additionally $(\gb) \in \SLS_{\ml}\left(\rho\right)$,
    From \Cref{lem:level_sets} we know
    \begin{align*}
        \frac{\tfrac{1-\kappa}{1+\kappa}-\rho}{\tfrac{1}{\rho} - \tfrac{1-\kappa}{1+\kappa}} \leq \beta \leq \rho^2 \\
        \mu\gamma \geq (1 - \rho)(1 - \tfrac{\beta}{\rho}).
    \end{align*}
    And from \eqref{eq:toto}, we know that
    \begin{equation*}
        \mu\gamma \leq C\kappa (1 - \beta).
    \end{equation*}
    Combining those inequalities, we get
    $(1 - \rho)(1 - \tfrac{\beta}{\rho}) \leq C\kappa (1 - \beta)$.  Rearanging the terms leads to
    \begin{equation*}
        \left(\frac{1-\rho}{\rho} - C\kappa\right)\beta \geq 1 - \rho - C\kappa.
    \end{equation*}
    Besides, $\frac{1-\rho}{\rho} - C\kappa \geq 0$ (since $\rho < \tfrac{1 - C\kappa}{1 + C\kappa} \leq \tfrac{1}{1 + C\kappa}$) and moreover $\beta \in \left[\frac{\tfrac{1-\kappa}{1+\kappa}-\rho}{\tfrac{1}{\rho} - \tfrac{1-\kappa}{1+\kappa}}, \rho^2\right]$, thus $\beta < \rho^2$ we get:
    \begin{equation*}
        \left(\frac{1-\rho}{\rho} - C\kappa\right)\rho^2 \geq 1 - \rho - C\kappa.
    \end{equation*} 
    Equivalently $(1 - \rho)^2 \leq C\kappa(1 - \rho^2)$, that is $\rho \geq \tfrac{1 - C\kappa}{1 + C\kappa}$. Which is in contradiction with our initial assumption $\rho < \frac{1 - C\kappa}{1 + C\kappa}$. As a conclusion,
    \begin{equation*}
        \rho_{\HB_{\gb}} \leq \rho \Longrightarrow \rho \geq \tfrac{1 - C\kappa}{1 + C\kappa}.
    \end{equation*}
\end{proof}

\section{Auxiliary proofs  from \texorpdfstring{\Cref{sec:robustness}}{Section 4}}
\subsection{Proof of \texorpdfstring{\Cref{thm:robustness}}{Theorem 4.3}}
\label{app:robustess}

\robutness*

\begin{proof}
    We introduce matrices $P$ and $D$ verifying $PDP^{-1} = \begin{pmatrix}
        (1+\beta)\Id_2 - \mu\gamma\Id_2 & -\beta\Id_2 \\ \Id_2 & 0
    \end{pmatrix}$, and such that  the operator norm   $\rho_D$ of the matrix $D$, $\rho_D = \|D\|_{\text{op}}$ is smaller than 1, and set $\kappa_P = \frac{1}{\|P\|_{\text{op}}\|P^{-1}\|_{\text{op}}} \leq 1$.  The existence of such matrices is guaranteed as $(\gb)\in \Oqml$, and  we discuss the choice of this reduction  in \Cref{apx:pdp-1_discussion}.
    
We prove by induction that  $\forall t\geq 1, \left\|P^{-1}\begin{pmatrix}
        \delta_{t} \\ \delta_{t-1}
    \end{pmatrix}\right\| \leq \frac{r_{\max}}{\|P\|_{\text{op}}}$.
which implies that  $\forall t\geq 1, \left\| \begin{pmatrix}
        \delta_{t} \\ \delta_{t-1}
    \end{pmatrix}\right\| \leq {r_{\max}}$, thus the  result of the theorem, i.e.,~$z_t\in B(x^\circ_{t\mod{K}}, r_{\max}) \subseteq \mathcal V_{t\mod{K}}$. 

In the  proof, we extend again $(x^\circ_k)$ to $k\in \mathbb N$ by periodicity, i.e.,~$(x^\circ_k)=(x^\circ_{k\mod{K}})$.

    \noindent
 \paragraph{Initialization:} $\left\|\begin{pmatrix}
        \delta_{1} \\ \delta_{0}
    \end{pmatrix}\right\| = \sqrt{\|\delta_{0}\|^2 + \|\delta_{1}\|^2} \leq \kappa_P r_{\max}= \frac{r_{\max}}{\|P\|_{\text{op}}\|P^{-1}\|_{\text{op}}}$   implies $\left\|P^{-1}\begin{pmatrix}
        \delta_{1} \\ \delta_{0}
    \end{pmatrix}\right\| \leq \frac{r_{\max}}{\|P\|_{\text{op}}}$. \\

 \noindent

 \noindent
 \paragraph{Induction:}  By induction hypothesis $\left\|P^{-1}\begin{pmatrix}
        \delta_{t} \\ \delta_{t-1}
    \end{pmatrix}\right\| \leq \frac{r_{\max}}{\|P\|_{\text{op}}}$, thus $z_t\in \mathcal V_{t\mod{K}}$. 
    We now write \eqref{eq:hb}'s $(t+1)^{\mathrm{th}}$ step:
    \begin{align*}
        z_{t+1} & = z_{t}  - \gamma_t \hat g_t(x^\circ_t + \delta_{t}) +  \beta_t  (z_{t} -z_{t-1} ) = z_{t}  - \gamma_t (\nabla \psi(x^\circ_t + \delta_{t}) +\delta_{g_t}) +  \beta_t  (z_{t} -z_{t-1} ).
    \end{align*}
    Using that for all $t\geq 0$, $z_{t+1}=  x^\circ_{t+1} + \delta_{t+1}$.
    \begin{align*}
        x^\circ_{t+1} + \delta_{t+1}& = x^\circ_{t} + \delta_{t} + (\beta + \delta_{\beta_t}) (x^\circ_{t} + \delta_{t} - x^\circ_{t-1} - \delta_{t-1}) - (\gamma + \delta_{\gamma_t}) \nabla \psi(x^\circ_t + \delta_{t}) - \gamma_t \delta_{g_t}.
    \end{align*}
And as $\eqref{eq:hb}_{\gb}(\psi)$ cycles on $\TikCircle_K$, we have that $x^\circ_{t+1} = {x^\circ_{t}} + {\beta}  (x^\circ_{t} - x^\circ_{t-1}) - {\gamma} \nabla \psi(x^\circ_t)$, thus
    \begin{align*}
        \delta_{t+1} & =  \delta_{t} + \delta_{\beta_t} (x^\circ_{t} - x^\circ_{t-1}) + (\beta + \delta_{\beta_t}) (\delta_{t} - \delta_{t-1}) - (  \delta_{\gamma_t}) \nabla \psi(x^\circ_t) - (\gamma + \delta_{\gamma_t}) \mu \delta_{t}- \gamma_t \delta_{g_t}, \\
        \delta_{t+1}
        & = \delta_{t} + \beta (\delta_{t} - \delta_{t-1}) - \gamma \mu \delta_{t} + \delta_{\beta_t} (x^\circ_{t} - x^\circ_{t-1} + \delta_{t} - \delta_{t-1}) - \delta_{\gamma_t} (\nabla \psi(x^\circ_t) + \mu \delta_{t})- \gamma_t \delta_{g_t}.
    \end{align*}
    This can be written in an augmented space as
    \begin{equation*}
        \begin{pmatrix}
            \delta_{t+1} \\ \delta_{t}
        \end{pmatrix}
        =
        \begin{pmatrix}
            1 + \beta - \mu\gamma & -\beta \\ 1 & 0
        \end{pmatrix}
        \begin{pmatrix}
            \delta_{t} \\ \delta_{t-1}
        \end{pmatrix}
        +
        \begin{pmatrix}
            \delta_{\beta_t} (x^\circ_{t} - x^\circ_{t-1} + \delta_{t} - \delta_{t-1}) - \delta_{\gamma_t} (\nabla \psi(x^\circ_t) + \mu \delta_{t}) - \gamma_t \delta_{g_t} \\ 0
        \end{pmatrix}.
    \end{equation*}
    The second and third assumptions corresponds on the perturbations enable to write that
    \begin{align*}
     &   \left\|\begin{pmatrix}
            \delta_{\beta_t} (x^\circ_{t} - x^\circ_{t-1} + \delta_{t} - \delta_{t-1}) - \delta_{\gamma_t} (\nabla \psi(x^\circ_t) + \mu \delta_{t} - \gamma_t \delta_{g_t}) \\ 0
        \end{pmatrix}\right\| \\
        & \qquad = 
        \|\delta_{\beta_t} (x^\circ_{t} - x^\circ_{t-1} + \delta_{t} - \delta_{t-1}) - \delta_{\gamma_t} (\nabla \psi(x^\circ_t) + \mu \delta_{t})- \gamma_t \delta_{g_t}\| \\
        & \qquad \leq |\delta_{\beta_t}|\|x^\circ_{t} - x^\circ_{t-1} + \delta_{t} - \delta_{t-1}\| + |\delta_{\gamma_t}| \|\nabla \psi(x^\circ_t) + \mu \delta_{t}\| + \gamma_t \|\delta_{g_t}\|\\
        & \qquad = \left(\frac{\sqrt{(1+\beta)^2(1 - \cos(\theta_K))^2 + (1-\beta)^2\sin(\theta_K)^2}}{\gamma} + \mu\kappa_P r_{\max}\right)|\delta_{\gamma_t}|  \\ 
        & \qquad\qquad +
        \left(\sqrt{(1 - \cos(\theta_K))^2 + \sin(\theta_K)^2} + 2\kappa_P r_{\max}\right) |\delta_{\beta_t}| + \frac{4}{L} \|\delta_{g_t}\|\\
        & \qquad \le \left(\frac{4}{\gamma} + \mu\kappa_P r_{\max}\right)|\delta_{\gamma_t}| +
        \left({2} + 2\kappa_P r_{\max}\right) |\delta_{\beta_t}| + \frac{4}{L} \|\delta_{g_t}\|\\
        & \qquad \leq \frac{1}{2}(1-\rho_D)\kappa_P r_{\max}+ \frac{1}{2}(1-\rho_D)\kappa_P r_{\max} = (1-\rho_D)\kappa_P r_{\max}.
    \end{align*}
    We now have
    \begin{align*}
        \begin{pmatrix}
            \delta_{t+1} \\ \delta_{t}
        \end{pmatrix}
        & =
        \begin{pmatrix}
            1 + \beta - \mu\gamma & -\beta \\ 1 & 0
        \end{pmatrix}
        \begin{pmatrix}
            \delta_{t} \\ \delta_{t-1}
        \end{pmatrix}
        +
        \begin{pmatrix}
            \delta_{\beta_t} (x^\circ_{t} - x^\circ_{t-1} + \delta_{t} - \delta_{t-1}) - \delta_{\gamma_t} (\nabla \psi(x^\circ_t) + \mu \delta_{t}) \\ 0
        \end{pmatrix} \\
        & =
        PDP^{-1}
        \begin{pmatrix}
            \delta_{t} \\ \delta_{t-1}
        \end{pmatrix}
        +
        \begin{pmatrix}
            \delta_{\beta_t} (x^\circ_{t} - x^\circ_{t-1} + \delta_{t} - \delta_{t-1}) - \delta_{\gamma_t} (\nabla \psi(x^\circ_t) + \mu \delta_{t}) \\ 0
        \end{pmatrix} \\
        P^{-1}
        \begin{pmatrix}
            \delta_{t+1} \\ \delta_{t}
        \end{pmatrix}
        & = 
        DP^{-1}
        \begin{pmatrix}
            \delta_{t} \\ \delta_{t-1}
        \end{pmatrix}
        +
        P^{-1}
        \begin{pmatrix}
            \delta_{\beta_t} (x^\circ_{t} - x^\circ_{t-1} + \delta_{t} - \delta_{t-1}) - \delta_{\gamma_t} (\nabla \psi(x^\circ_t) + \mu \delta_{t}) \\ 0
        \end{pmatrix} \\
        \left\|
        P^{-1}
        \begin{pmatrix}
            \delta_{t+1} \\ \delta_{t}
        \end{pmatrix}
        \right\|
        & \leq 
        \|D\| \left\|
        P^{-1}
        \begin{pmatrix}
            \delta_{t} \\ \delta_{t-1}
        \end{pmatrix}
        \right\|
        +
        \left\|P^{-1}
        \begin{pmatrix}
            \delta_{\beta_t} (x^\circ_{t} - x^\circ_{t-1} + \delta_{t} - \delta_{t-1}) - \delta_{\gamma_t} (\nabla \psi(x^\circ_t) + \mu \delta_{t}) \\ 0
        \end{pmatrix} \right\| \\
        & \leq 
        \rho_D \left\|
        P^{-1}
        \begin{pmatrix}
            \delta_{t} \\ \delta_{t-1}
        \end{pmatrix}
        \right\|
        +
        \|P^{-1}\|_{\text{op}}
        (1-\rho_D)\kappa_P r_{\max} \\
        & \overset{\text{by induction}}{\leq} \rho_D\frac{r_{\max}}{\|P\|_{\text{op}}} + (1-\rho_D)\frac{r_{\max}}{\|P\|_{\text{op}}} \\
        & = \frac{r_{\max}}{\|P\|_{\text{op}}},
    \end{align*}
    thereby reaching the desired claim.
\end{proof}

\subsection{Discussion about the reduction made in the proof of \texorpdfstring{\Cref{thm:robustness}}{Theorem 4.3}}
\label{apx:pdp-1_discussion}

In \Cref{thm:robustness}, we decompose the matrix $\begin{pmatrix}
    1+\beta - \mu\gamma & -\beta \\ 1 & 0
\end{pmatrix}$ into the form $PDP^{-1}$ with $\|D\| < 1$.
In this section, we discuss a possible way to do it  depending on the region in which $(\gb)$ lies (see \Cref{prop:conv_quad}).
In the 3 possible cases, we provide a decomposition $PDP^{-1}$ and $\rho_D = \|D\|<1$.
To compute $\kappa_P = \frac{1}{\|P\|_{\text{op}} \|P^{-1}\|_{\text{op}}}$, we can use the fact that $\kappa_P^2$ is the ratio of the 2 eigenvalues of the matrix $P^T P$:
{\footnotesize\begin{align*}
    \kappa_P
    & = 
    \left(
    \frac{
    \frac{\mathrm{Tr}(P^T P)}{2} - \sqrt{\left(\frac{\mathrm{Tr}(P^T P)}{2}\right)^2 - \mathrm{Det}(P^T P)}
    }
    {
    \frac{\mathrm{Tr}(P^T P)}{2} + \sqrt{\left(\frac{\mathrm{Tr}(P^T P)}{2}\right)^2 - \mathrm{Det}(P^T P)}
    }
    \right)^{1/2}  = 
    \left(
    \frac{
    \left(\frac{\mathrm{Tr}(P^T P)}{2} - \sqrt{\left(\frac{\mathrm{Tr}(P^T P)}{2}\right)^2 - \mathrm{Det}(P^T P)}\right)^2
    }
    {
    \mathrm{Det}(P^T P)
    }
    \right)^{1/2} \\
    & = 
    \frac{\frac{\mathrm{Tr}(P^T P)}{2} - \sqrt{\left(\frac{\mathrm{Tr}(P^T P)}{2}\right)^2 - \mathrm{Det}(P)^2}
    }
    {
    |\mathrm{Det}(P)|
    } = 
    \frac{\mathrm{Tr}(P^T P)}{2|\mathrm{Det}(P)|} - \sqrt{\left(\frac{\mathrm{Tr}(P^T P)}{2|\mathrm{Det}(P)|}\right)^2 - 1}.
\end{align*}}

\noindent\paragraph{Lazy region ($\gamma < \frac{(1-\sqrt{\beta})^2}{\mu}$).}

In this region, 
$$\begin{pmatrix}
    1+\beta - \mu\gamma & -\beta \\ 1 & 0
\end{pmatrix}
\sim
D = \begin{pmatrix}
    \frac{1+\beta - \mu\gamma}{2} + \sqrt{\left(\frac{1+\beta - \mu\gamma}{2}\right)^2 - \beta} & 0 \\ 0 & \frac{1+\beta - \mu\gamma}{2} - \sqrt{\left(\frac{1+\beta - \mu\gamma}{2}\right)^2 - \beta}
\end{pmatrix}$$

with the transition matrix
$P = \begin{pmatrix}
    \frac{1+\beta - \mu\gamma}{2} + \sqrt{\left(\frac{1+\beta - \mu\gamma}{2}\right)^2 - \beta} \quad & \frac{1+\beta - \mu\gamma}{2} - \sqrt{\left(\frac{1+\beta - \mu\gamma}{2}\right)^2 - \beta} \\ 1 & 1
\end{pmatrix}$.

We then obtain $\rho_D = \|D\|_{\op} = \frac{1+\beta - \mu\gamma}{2} + \sqrt{\left(\frac{1+\beta - \mu\gamma}{2}\right)^2 - \beta}< 1$.

\noindent\paragraph{Robust region ($\gamma > \frac{(1-\sqrt{\beta})^2}{\mu}$).}

In this region, 
$$\begin{pmatrix}
    1+\beta - \mu\gamma & -\beta \\ 1 & 0
\end{pmatrix}
\sim
D = \begin{pmatrix}
    \frac{1+\beta - \mu\gamma}{2} + i \sqrt{\beta - \left(\frac{1+\beta - \mu\gamma}{2}\right)^2} & 0 \\ 0 & \frac{1+\beta - \mu\gamma}{2} - i \sqrt{\beta - \left(\frac{1+\beta - \mu\gamma}{2}\right)^2}
\end{pmatrix}$$

with the transition matrix
$P = \begin{pmatrix}
    \frac{1+\beta - \mu\gamma}{2} + i \sqrt{\beta - \left(\frac{1+\beta - \mu\gamma}{2}\right)^2} \quad & \frac{1+\beta - \mu\gamma}{2} - i \sqrt{\beta - \left(\frac{1+\beta - \mu\gamma}{2}\right)^2} \\ 1 & 1
\end{pmatrix}$.

We then obtain $\rho_D = \|D\|_{\op} = \sqrt{\beta} < 1$.

\noindent\paragraph{Boundary ($\gamma = \frac{(1-\sqrt{\beta})^2}{\mu}$).}

On the boundary between the lazy and the robust regions, 
$\begin{pmatrix}
    1+\beta - \mu\gamma & -\beta \\ 1 & 0
\end{pmatrix}$
is not diagonalisable.

However, we can write
$\begin{pmatrix}
    1+\beta - \mu\gamma & -\beta \\ 1 & 0
\end{pmatrix}
\sim
D = \sqrt{\beta}\begin{pmatrix}
    1 & \varepsilon \\ 0 & 1
\end{pmatrix}$, with any $\varepsilon>0$
using the transition matrix
$P = \begin{pmatrix}
    \sqrt{\beta} & \frac{\varepsilon\sqrt{\beta}}{1 + \beta} \\ 1 & -\frac{\beta\varepsilon}{1 + \beta}
\end{pmatrix}$.
We then obtain $\rho_D = \|D\| = \sqrt{\beta}\sqrt{1 + \frac{\varepsilon^2}{2} + \sqrt{\left(1 + \frac{\varepsilon^2}{2}\right)^2 - 1}} = \sqrt{\beta}\left(\frac{\varepsilon}{2} + \sqrt{1 + \frac{\varepsilon^2}{4}}\right)$.
Note $\rho_D < 1$ if and only if $\varepsilon < \frac{1-\beta}{\sqrt{\beta}}$.

\section{Auxiliary proofs from \texorpdfstring{\Cref{sec:HL}}{Section 5}}
\label{app:proofLem1_HL}

\begin{Def}[Mollifier $u_\epsilon$, $\epsilon>0$]
    We define, for any  $\varepsilon>0.$
\begin{align*}
    u(x) & \eqdef \frac{1}{Z} e^{-\frac{1}{1 - \|x\|^2}}\mathbf{1}_{\|x\| < 1}, & \text{with } Z \eqdef \int_{\|x\|\leq 1} e^{-\frac{1}{1 - \|x\|^2}}\mathrm{dx},  \quad u_{\varepsilon}(x) \eqdef \frac{1}{\varepsilon^2}u\left(\frac{1}{\varepsilon}x\right).
\end{align*}
\end{Def}
We first recall some classical properties of $u_\varepsilon$, see for example \citep[Section 4.4 of ][on \textit{mollifiers}]{brezis2011functional}.
\begin{Lemma}
    \label{lem:proba}
    $\forall \varepsilon>0, u_\varepsilon$ is the probability density function (pdf) of an $L^\infty$ and centered random variable.
\end{Lemma}

\begin{Lemma}[$C^{\infty}$ with compact support]
    \label{lem:c_inf_with_compact_support}
    For any $\varepsilon > 0$, $u_\varepsilon \in C^{\infty}$ and its support is $B(0, \varepsilon)$.
\end{Lemma}

\begin{Lemma}
[Bounded derivatives]
    \label{cor:bounded_derivatives}
    For any $\varepsilon > 0$ and any $r \geq 0$,
    \begin{equation*}
        M^{(r)}_{\varepsilon} \eqdef \sup_{x\in\mathbb{R}^2} \left\|\nabla^r u_{\varepsilon}(x)\right\| = \sup_{x\in B(0, \varepsilon)} \left\|\nabla^r u_{\varepsilon}(x)\right\| < \infty.
    \end{equation*}
\end{Lemma}

Let $\psi\in \Fml$. We consider $\varphi_\varepsilon = \psi\ast u_\epsilon$.  We recall the following properties of $\varphi_\varepsilon$.

\begin{Lemma}[Higher-order derivatives are bounded]
    \label{lem:prop_varphi}
    Let $\psi\in \Fml$ and $\varphi_\varepsilon = \psi\ast u_\epsilon$. Then for any $\varepsilon>0$,  
    \begin{enumerate}
        \item  $\varphi_{\varepsilon} \in C^{\infty}$
    \item for any $r\geq 2$  $\left\|\nabla^r (\varphi_{\varepsilon})\right\|$ is bounded.
    \item  $\varphi_{\varepsilon} \in \Fml$.
    \end{enumerate}
\end{Lemma}
The proof of point 1 is standard, the one of point 2 uses that $\forall x, \|\nabla^2 \psi(x)\|\le L$ and the properties of the convolution, and finally, the proof of point 3 relies on the fact that $\forall \ x, \mu \le \|\nabla^2 \psi(x)\|\le L$ and that $u_\epsilon$ is a probability density function.

\section{Auxiliary proofs from \texorpdfstring{\Cref{sec:cycles}}{Section 6}: Proof of \texorpdfstring{\Cref{lem:circulant_dec}}{Lemma 6.12}}
\label{app:circulant_proof}

\circulantreduction*

\begin{proof}{\Cref{lem:circulant_dec}}
    The minimal polynomial of $J_K$ is $X^K - 1$, a split polynomial with simple roots in $\mathbb{C}$.
    Therefore, $J_K$ diagonalizes on $\mathbb{C}$ with eigenvalues $\omega^\ell = e^{2i\pi \ell/K}$, $\ell\in\llbracket 0, K-1 \rrbracket$, and $G$ diagonalizes in the same basis with eigenvalues $\nu_\ell \eqdef \sum_{j=0}^{K-1} c_j \omega^{j \ell}$.
    The sequence $(\nu_\ell)_{\ell\in\llbracket 0, K-1 \rrbracket}$ is the discrete Fourier transform of $(c_j)_{j\in\llbracket 0, K-1 \rrbracket}$.
    Therefore, $(c_j)_{j\in\llbracket 0, K-1 \rrbracket}$ is the inverse discrete Fourier transform of $(\nu_\ell)_{\ell\in\llbracket 0, K-1 \rrbracket}$:
    \begin{equation}
        \forall j\in\llbracket 0, K-1 \rrbracket, c_j = \frac{1}{K}\sum_{\ell=0}^{K-1} \nu_\ell \omega^{-j \ell}. \label{eq:dft}
    \end{equation}
    Moreover,
    \begin{equation*}
        \nu_{K-\ell} = \sum_{j=0}^{K-1} c_j \omega^{j (K-\ell)} = \sum_{j=0}^{K-1} c_j \omega^{- j \ell} = \bar{\nu_\ell},
    \end{equation*}
    and by symmetry of $G$, $\bar{\nu_\ell} = \nu_\ell$.
    Hence, the eigenvalues almost all have an even multiplicity.
    This excludes $\nu_0$ and $\nu_{K/2}$ if $K$ is even.
    We can therefore reorder and group terms in \eqref{eq:dft} as
    \begin{align*}
        \forall j\in\llbracket 0, K-1 \rrbracket, c_j & = \frac{1}{K}\left[\nu_0 + \sum_{\ell=1}^{\lfloor \frac{K-1}{2} \rfloor} \nu_\ell (\omega^{-j \ell} + \omega^{j \ell}) + \nu_{\frac{K}{2}} \omega^{-j \frac{K}{2}} \right] \tag{where the last term drops if $K=1\mod{2}$} \\
        & = \frac{1}{K}\left[\nu_0 + \sum_{\ell=1}^{\lfloor \frac{K-1}{2} \rfloor} 2 \nu_\ell \cos\left(\frac{2\pi jl}{K}\right) + \nu_{\frac{K}{2}} (-1)^{j} \right]
        \label{eq:dft_grouped}
    \end{align*}
    Finally, $G$ is symmetric positive semi-definite and circulant if and only if there exists non-negative $(\nu_l)_{l \in \range{0}{\lfloor \frac{K}{2} \rfloor}}$ such that
    \begin{align*}
        G = & \sum_{j=0}^{K-1} \frac{1}{K}\left[\nu_0 + \sum_{\ell=1}^{\lfloor \frac{K-1}{2} \rfloor} 2 \nu_\ell \cos\left(\frac{2\pi jl}{K}\right) + \nu_{\frac{K}{2}} (-1)^{j} \right] J_K^j \\
        = & \frac{\nu_0}{K} \underbrace{\sum_{j=0}^{K-1} J_K^j}_{\text{Matrix full of 1s}} + \frac{\nu_{\frac{K}{2}}}{K} \underbrace{\sum_{j=0}^{K-1} (-1)^{j} J_K^j}_{\text{Checkerboard $\left((-1)^{|i-j|}\right)_{i, j}$}} + \sum_{\ell=1}^{\lfloor \frac{K-1}{2} \rfloor} \frac{2 \nu_\ell}{K} \underbrace{\sum_{j=0}^{K-1} \cos\left(\frac{2\pi jl}{K}\right) J_K^j}_{\left(\cos\left(\frac{2\pi \ell}{K}|i-j|\right)\right)_{i, j}}.
    \end{align*}
    The condition $G\mathbf{1}_K=0$ implies $\nu_0=0$.
\end{proof}

\section{A summary of convergence rates on \texorpdfstring{$\Fml$}{Fml} and \texorpdfstring{$\Qml$}{Qml}}
\label{app:bonus_all_rates}
For completeness and reference purposes, we summarize the convergence rates of the algorithms under consideration of this work, on the functional classes $\Fml$ and $\Qml$. \Cref{tab:rates_fml} provides the algorithms and their respective worst-case converge rates in the sense of~\Cref{def:awcc} expressed uniformly over $\kappa$, and approximately as $\kappa\to 0$.
\begin{table*}[ht]
{
\begin{center}
{\renewcommand{\arraystretch}{1.8}
\caption{Asymptotic convergence rates (on $\|x_t-x_\star\|$) for standard algorithms. \framebox{Optimal convergence rates} (lower and upper complexity bounds) are highlighted in {boxes}.}
\begin{tabular}{lp{2.3cm}p{2.3cm}p{2.3cm}p{2.3cm}}
\toprule
Algorithm    & \multicolumn{2}{c}{Known convergence rate} & \multicolumn{2}{c}{\parbox{5cm}{\centering Approximate rate as $\kappa\rightarrow 0$}}  \\
\cmidrule(l){2-3} \cmidrule(l){4-5}
   &  on $\Fml$ &on $\Qml$&  on $\Fml$ & on $\Qml$\\
\cmidrule(l){2-2} \cmidrule(l){3-3}\cmidrule(l){4-4} \cmidrule(l){5-5}
GD$(\gamma = 1/L)$ & $1-\kappa$ \newline(see~\cite{Nest03a})& $1-\kappa$ \newline(see~\cite{Nest03a})& $1-\kappa$ & $1-\kappa$ \\
\midrule
GD$(\gamma = 2/(L+\mu))$ &$\frac{1-\kappa}{1+\kappa}$ \newline(see~\cite{Nest03a}) & $\frac{1-\kappa}{1+\kappa}$ \newline(see~\cite{Nest03a})& $1-2\kappa$ & $1-2\kappa$\\\midrule
Chebyshev's method & ? &  \framebox{${\frac{1-\sqrt{\kappa}}{1+\sqrt{\kappa}}}$} (see~\cite{nemirovskinotes1995})&  ? &  \framebox{${1-2\sqrt{\kappa}}$}\\
\midrule
$\HB(\gamma^\star(\Qml), \beta^\star(\Qml))$ & none (cycles) \newline(see~\cite{lessard2016analysis})  &  \framebox{${\frac{1-\sqrt{\kappa}}{1+\sqrt{\kappa}}}$} (see~\cite{nemirovskinotes1995})&  none (cycles) & \framebox{${1-2\sqrt{\kappa}}$}\\
\midrule
$\HB(\gamma^\star(\Fml), \beta^\star(\Fml))$ & $ 1-\Theta(\kappa)$ \newline(\Cref{cor:noaccel}) & $ 1-\Theta(\kappa)$ \newline(\Cref{thm:incompatibility}) & $ 1-\Theta(\kappa)$ &  $1-\Theta(\kappa)$\\
\midrule
NAG $(\gamma =1/L, \beta=\frac{1-\sqrt{\kappa}}{1+\sqrt{\kappa}})$ & $(1-\sqrt{\kappa})^{1/2}$ (see~\cite{Nest03a}) & $1-\sqrt{\kappa}$ \newline(see~\cite{hagedorn2023iteration})&   $1-\tfrac{1}{2}\sqrt{\kappa}$& $1-\sqrt{\kappa} $\\
\midrule
Information-theoretic exact method   & \framebox{${1-\sqrt{\kappa}}$}\newline(see~\citep{taylor2023optimal}) & $1-\sqrt{\kappa}$\newline(see~\citep{taylor2023optimal}) & \framebox{${1-\sqrt{\kappa}}$} & $1-\sqrt{\kappa}$ \\
\midrule
Triple momentum method & \framebox{${1-\sqrt{\kappa}}$}\newline(see~\citep{van2017fastest}) & $1-\sqrt{\kappa}$\newline(see~\citep{van2017fastest}) & \framebox{${1-\sqrt{\kappa}}$} & $1-\sqrt{\kappa}$  \\
\midrule
Lower complexity bounds & \framebox{${1-\sqrt{\kappa}}$} (see~\cite{drori2022oracle}) & \framebox{${\frac{1-\sqrt{\kappa}}{1+\sqrt{\kappa}}}$} (see~\cite{nemirovskinotes1995}) & \framebox{${1-\sqrt{\kappa}}$} & \framebox{${1-2\sqrt{\kappa}}$}\\
\bottomrule
\end{tabular}}
\end{center}}\label{tab:rates_fml}
\end{table*}

\end{document}

%% file: header.tex
\usepackage{times}
\usepackage[utf8]{inputenc}
\usepackage[T1]{fontenc}
\usepackage{url}
\usepackage{booktabs}
\usepackage{multirow}
\usepackage{amsfonts}
\usepackage{yhmath}
\usepackage{nicefrac}
\usepackage{microtype}
\usepackage{faktor}
\usepackage{hhline}

\usepackage{thm-restate}

\usepackage{tikz}
\usepackage{graphicx}
\usepackage{grffile}
\usepackage{mathtools, stmaryrd}
\usepackage{tocloft}
\usepackage[hidelinks]{hyperref}
\usepackage{xcolor}
\hypersetup{
    colorlinks,
    linkcolor={red!50!black},
    citecolor={blue!50!black},
    urlcolor={blue!80!black}
}
\definecolor{ao(english)}{rgb}{0.0, 0.5, 0.0}

\usepackage[algo2e, algoruled, boxed, vlined]{algorithm2e}
\usepackage{algorithm}
\usepackage{algpseudocode}

\SetKwInput{KwInput}{Input}
\SetKwInput{KwOutput}{Output}

\usepackage[noabbrev]{cleveref}

\usepackage{caption}
\usepackage{subcaption}

\usepackage{amssymb}
\usepackage{amsmath}
 \usepackage{natbib}
\usepackage{multirow}

\usepackage{pifont}

\usepackage{blindtext}
\usepackage{cancel}

\usepackage{enumitem}
\usepackage{float}
\usepackage{wrapfig}

\usepackage{tikz}

\usepackage{booktabs}
\usepackage{mdframed}
\newmdenv[topline=false,rightline=false]{leftbot}

\newtheorem{Th}{Theorem}[section]
\newtheorem{Lemma}[Th]{Lemma}
\newtheorem{Cor}[Th]{Corollary}
\newtheorem{Prop}[Th]{Proposition}

\newtheorem{Def}[Th]{Definition}

\newtheorem{Not}[Th]{Notation}

\newtheorem{Rem}[Th]{Remark}
\newtheorem{Fact}[Th]{Fact}
\newtheorem{Conjecture}[Th]{Conjecture}
\newtheorem{?}[Th]{Problem}
\newtheorem{Ex}[Th]{Example}

\crefname{Def}{Definition}{Definitions}
\crefname{Rem}{Remark}{Remarks}
\crefname{Th}{Theorem}{Theorems}

\newenvironment{proof}
{\ \\
\noindent \textit{Proof.}
}
{
$\hfill\blacksquare$}

\newenvironment{sketch}
{\ \\\noindent \textit{Sketch of proof.}
}
{
$\hfill\blacksquare$ \\ }

\def\GD{\operatorname{GD}}
\def\HB{\operatorname{HB}}
\def\NAG{\operatorname{NAG}}
\newcommand{\HBgb}{\eqref{eq:hb}$\,_{\gamma, \beta}~$}

\DeclareMathOperator*{\argmin}{arg\,min}

\DeclareMathOperator*{\conv}{{convh}}
\newcommand{\Sp}{\mathrm{Sp}}

\newcommand{\Id}{\mathrm{I}}
\newcommand{\Tr}{\mathrm{Tr}}
\newcommand{\op}{\mathrm{op}}

\renewcommand{\leq}{\leqslant}
\renewcommand{\geq}{\geqslant}
\renewcommand{\le}{\leqslant}

\renewcommand{\epsilon}{\varepsilon}

\renewcommand{\mod}[1]{\;(\mathrm{mod}\; #1)}
\newcommand{\floor}[1]{\lfloor #1 \rfloor}
\newcommand{\eqdef}{:=}
\newcommand{\range}[2]{\llbracket #1, #2 \rrbracket}

\newcommand{\R}{\mathbb{R}}
\newcommand{\fs}{f_{\star}}

\newcommand{\xs}{x_{\star}}

\newcommand{\F}{\mathcal{F}}
\newcommand{\Qml}{\mathcal{Q}_{\mu, L}}
\newcommand{\Fml}{\mathcal{F}_{\mu, L}}
\newcommand{\Ftml}{\mathcal{F}^{\tau}_{\mu, L}}

\def\Ghadimi{\mathrm{Ghad.}}
\def\TVL{\mathrm{TVL}}

\newcommand{\Oghml}{\Omega_{\Ghadimi}(\Fml)}
\newcommand{\Otaml}{\Omega_{\TVL}(\Fml)}

\newcommand{\Oqml}{\Omega_{\mathrm{cv}}(\Qml)}
\newcommand{\Ocvml}{\Omega_{\mathrm{cv}}(\Fml)}
\newcommand{\OcvF}{\Omega_{\mathrm{cv}}(\mathcal{F})}

\newcommand{\cycle}{\mathrm{Cycle}}
\newcommand{\Ocyml}{\Omega_{\cycle}(\Fml)}
\newcommand{\Ocytml}{\Omega_{\cycle}(\Ftml)}
\newcommand{\Orouml}{\Omega_{\circ\text{-}\cycle}(\Fml)}
\newcommand{\Oroutml}{\Omega_{\circ\text{-}\cycle}(\Ftml)}
\newcommand{\OKrouml}{\Omega_{K\text{-}\circ\text{-}\cycle}(\Fml)}
\newcommand{\OmaxKrouml}{\Omega_{\leq K\text{-}\circ\text{-}\cycle}(\Fml)}

\newcommand{\ml}{\mu,L}
\newcommand{\gb}{\gamma,\beta}
\newcommand{\gbqml}{\gamma^\star(\Qml),\beta^\star(\Qml)}
\newcommand{\rhogbQml}{\rho_{\gb}(\Qml)}
\newcommand{\psigbml}{\psi_{\gb,\ml}^K}
\newcommand{\phigbmleps}{\varphi_{\epsilon,\gb,\ml}^K}
\newcommand{\phigbml}{\varphi_{\gb,\ml}^K}
\newcommand{\rK}{\range{0}{K-1}}
\newcommand{\Vk}{\mathcal{V}_k}

\newcommand\TikCircle[1][2.5]{\tikz[baseline=-#1]{\draw[thick](0,0)circle[radius=1.6mm];}}

\newcommand{\LS}{\mathrm{LS}}
\newcommand{\SLS}{\mathrm{SLS}}

%% file: main.bbl
\begin{thebibliography}{38}
\providecommand{\natexlab}[1]{#1}
\providecommand{\url}[1]{{#1}}
\providecommand{\urlprefix}{URL }
\expandafter\ifx\csname urlstyle\endcsname\relax
  \providecommand{\doi}[1]{DOI~\discretionary{}{}{}#1}\else
  \providecommand{\doi}{DOI~\discretionary{}{}{}\begingroup
  \urlstyle{rm}\Url}\fi
\providecommand{\eprint}[2][]{\url{#2}}

\bibitem[{Berthier et~al.(2020)Berthier, Bach, and
  Gaillard}]{berthier2020accelerated}
Berthier R, Bach F, Gaillard P (2020) Accelerated gossip in networks of given
  dimension using {J}acobi polynomial iterations. SIAM Journal on Mathematics
  of Data Science 2(1):24--47

\bibitem[{Bottou and Bousquet(2007)}]{bottou2007tradeoffs}
Bottou L, Bousquet O (2007) The tradeoffs of large scale learning. In: Advances
  in Neural Information Processing Systems (NIPS)

\bibitem[{Br{\'e}zis(2011)}]{brezis2011functional}
Br{\'e}zis H (2011) Functional analysis, Sobolev spaces and partial
  differential equations, vol~2. Springer

\bibitem[{Bubeck(2015)}]{bubeck2015convex}
Bubeck S (2015) Convex optimization: Algorithms and complexity. Found~and
  Trends in Machine Learning 8(3-4):231--357

\bibitem[{Dobson et~al.(2023)Dobson, Sanz-Serna, and
  Zygalakis}]{dobson2023connections}
Dobson P, Sanz-Serna JM, Zygalakis K (2023) On the connections between
  optimization algorithms, {L}yapunov functions, and differential equations:
  theory and insights. arXiv preprint arXiv:230508658

\bibitem[{Drori and Taylor(2022)}]{drori2022oracle}
Drori Y, Taylor A (2022) On the oracle complexity of smooth strongly convex
  minimization. Journal of Complexity 68:101590

\bibitem[{Drori and Teboulle(2014)}]{drori2014performance}
Drori Y, Teboulle M (2014) Performance of first-order methods for smooth convex
  minimization: a novel approach. \Mathematical Programming 145(1):451--482

\bibitem[{d’Aspremont et~al.(2021)d’Aspremont, Scieur, and
  Taylor}]{d2021acceleration}
d’Aspremont A, Scieur D, Taylor A (2021) Acceleration methods. Foundations
  and Trends{\textregistered} in Optimization 5(1-2):1--245

\bibitem[{Fischer(2011)}]{fischer2011polynomial}
Fischer B (2011) {Polynomial based iteration methods for symmetric linear
  systems}. SIAM

\bibitem[{Ghadimi et~al.(2015)Ghadimi, Feyzmahdavian, and
  Johansson}]{ghadimi2015global}
Ghadimi E, Feyzmahdavian HR, Johansson M (2015) Global convergence of the
  heavy-ball method for convex optimization. In: European control conference
  (ECC)

\bibitem[{Goujaud et~al.(2022{\natexlab{a}})Goujaud, Moucer, Glineur,
  Hendrickx, Taylor, and Dieuleveut}]{goujaud2022pepit}
Goujaud B, Moucer C, Glineur F, Hendrickx J, Taylor A, Dieuleveut A
  (2022{\natexlab{a}}) {PEPit: computer-assisted worst-case analyses of
  first-order optimization methods in Python}. preprint arXiv:220104040

\bibitem[{Goujaud et~al.(2022{\natexlab{b}})Goujaud, Scieur, Dieuleveut,
  Taylor, and Pedregosa}]{goujaud2022super}
Goujaud B, Scieur D, Dieuleveut A, Taylor AB, Pedregosa F (2022{\natexlab{b}})
  Super-acceleration with cyclical step-sizes. In: International Conference on
  Artificial Intelligence and Statistics (AISTATS)

\bibitem[{Goujaud et~al.(2022{\natexlab{c}})Goujaud, Taylor, and
  Dieuleveut}]{goujaud2022quadratic}
Goujaud B, Taylor A, Dieuleveut A (2022{\natexlab{c}}) Quadratic minimization:
  from conjugate gradient to an adaptive heavy-ball method with {P}olyak
  step-sizes. arXiv preprint arXiv:221006367

\bibitem[{Goujaud et~al.(2023)Goujaud, Dieuleveut, and
  Taylor}]{goujaud2023counter}
Goujaud B, Dieuleveut A, Taylor A (2023) Counter-examples in first-order
  optimization: a constructive approach. IEEE Control Systems Letters {(See
  \href{https://arxiv.org/abs/2303.10503}{arXiv 2303 10503} for complete
  version with appendices)}

\bibitem[{Goujaud et~al.(2024)Goujaud, Taylor, and
  Dieuleveut}]{goujaud2022optimal}
Goujaud B, Taylor A, Dieuleveut A (2024) Optimal first-order methods for convex
  functions with a quadratic upper bound. Open Journal of Mathematical
  Optimization (OJMO) 5(9)

\bibitem[{Goujaud et~al.(2025)Goujaud, Taylor, and
  Dieuleveut}]{goujaud2025open}
Goujaud B, Taylor A, Dieuleveut A (2025) Open problem: Two riddles in
  heavy-ball dynamics. arXiv:250219916

\bibitem[{Gray(2006)}]{gray2006toeplitz}
Gray RM (2006) Toeplitz and circulant matrices: A review. Foundations and
  Trends{\textregistered} in Communications and Information Theory
  2(3):155--239, original publication 1971

\bibitem[{Gupta et~al.(2021)Gupta, Balakrishnan, and Ramdas}]{gupta2021path}
Gupta C, Balakrishnan S, Ramdas A (2021) Path length bounds for gradient
  descent and flow. The Journal of Machine Learning Research 22(1):3154--3216

\bibitem[{Hagedorn and Jarre(2023)}]{hagedorn2023iteration}
Hagedorn M, Jarre F (2023) Iteration complexity of fixed-step methods by
  nesterov and polyak for convex quadratic functions. Journal of Optimization
  Theory and Applications pp 1--19

\bibitem[{Kim et~al.(2024)Kim, Gidel, Kyrillidis, and
  Pedregosa}]{kim2022extragradient}
Kim JL, Gidel G, Kyrillidis A, Pedregosa F (2024) When is momentum
  extragradient optimal? a polynomial-based analysis. Transactions on Machine
  Learning Research

\bibitem[{Lessard et~al.(2016)Lessard, Recht, and
  Packard}]{lessard2016analysis}
Lessard L, Recht B, Packard A (2016) Analysis and design of optimization
  algorithms via integral quadratic constraints. SIAM Journal on Optimization
  26(1):57--95

\bibitem[{MOSEK(2019)}]{mosek}
MOSEK A (2019) MOSEK Optimizer API for C 9.3.6.
  \urlprefix\url{https://docs.mosek.com/latest/capi/index.html}

\bibitem[{Nemirovskii(1994)}]{nemirovskinotes1995}
Nemirovskii AS (1994) Information-based complexity of convex programming.
  Lecture notes
  (\href{https://www2.isye.gatech.edu/~nemirovs/Lec_EMCO.pdf}{link})

\bibitem[{Nemirovsky and Yudin(1983)}]{nemirovskij1983problem}
Nemirovsky AS, Yudin DB (1983) {Problem Complexity and Method Efficiency in
  Optimization}. Willey-Interscience, New York

\bibitem[{Nesterov(1983)}]{nesterov1983method}
Nesterov Y (1983) A method of solving a convex programming problem with
  convergence rate ${O}(1/k^2)$. Soviet Mathematics Doklady 27(2):372--376

\bibitem[{Nesterov(2003)}]{Nest03a}
Nesterov Y (2003) Introductory Lectures on Convex Optimization. Springer

\bibitem[{Pedregosa and Scieur(2020)}]{pedregosa2020acceleration}
Pedregosa F, Scieur D (2020) Acceleration through spectral density estimation.
  In: International Conference on Machine Learning (ICML)

\bibitem[{Polyak(1964)}]{polyak1964some}
Polyak BT (1964) {Some methods of speeding up the convergence of iteration
  methods}. {USSR} computational mathematics and mathematical physics

\bibitem[{Polyak(1987)}]{Book:polyak1987}
Polyak BT (1987) Introduction to optimization. Optimization Software New York

\bibitem[{Rockafellar(1997)}]{rockafellar1997convex}
Rockafellar RT (1997) Convex analysis, vol~11. Princeton university press

\bibitem[{Taylor and Drori(2023)}]{taylor2023optimal}
Taylor A, Drori Y (2023) An optimal gradient method for smooth strongly convex
  minimization. Mathematical Programming 199(1-2):557--594

\bibitem[{Taylor et~al.(2018)Taylor, Van~Scoy, and
  Lessard}]{taylor2018lyapunov}
Taylor A, Van~Scoy B, Lessard L (2018) Lyapunov functions for first-order
  methods: Tight automated convergence guarantees. In: International Conference
  on Machine Learning (ICML)

\bibitem[{Taylor et~al.(2017{\natexlab{a}})Taylor, Hendrickx, and
  Glineur}]{taylor2017exact}
Taylor AB, Hendrickx JM, Glineur F (2017{\natexlab{a}}) Exact worst-case
  performance of first-order methods for composite convex optimization. SIAM
  Journal on Optimization 27(3):1283--1313

\bibitem[{Taylor et~al.(2017{\natexlab{b}})Taylor, Hendrickx, and
  Glineur}]{taylor2017performance}
Taylor AB, Hendrickx JM, Glineur F (2017{\natexlab{b}}) {Performance estimation
  toolbox (PESTO): automated worst-case analysis of first-order optimization
  methods}. In: 56th Annual Conference on Decision and Control (CDC), pp
  1278--1283

\bibitem[{Taylor et~al.(2017{\natexlab{c}})Taylor, Hendrickx, and
  Glineur}]{taylor2017smooth}
Taylor AB, Hendrickx JM, Glineur F (2017{\natexlab{c}}) Smooth strongly convex
  interpolation and exact worst-case performance of first-order methods.
  \Mathematical Programming 161(1-2):307--345

\bibitem[{Van~Scoy et~al.(2017)Van~Scoy, Freeman, and Lynch}]{van2017fastest}
Van~Scoy B, Freeman RA, Lynch KM (2017) The fastest known globally convergent
  first-order method for minimizing strongly convex functions. IEEE Control
  Systems Letters 2(1):49--54

\bibitem[{Vandenberghe and Boyd(1996)}]{vandenberghe1996semidefinite}
Vandenberghe L, Boyd S (1996) Semidefinite programming. SIAM review
  38(1):49--95

\bibitem[{Wang et~al.(2022)Wang, Lin, Wibisono, and Hu}]{wang2022provable}
Wang JK, Lin CH, Wibisono A, Hu B (2022) Provable acceleration of heavy ball
  beyond quadratics for a class of {P}olyak-{L}ojasiewicz functions when the
  non-convexity is averaged-out. In: International Conference on Machine
  Learning (ICML)

\end{thebibliography}
